\documentclass[12pt]{amsart}

\usepackage[all]{xypic}
\usepackage{tikz}
\usetikzlibrary{arrows} 
\usetikzlibrary{decorations.markings}
\usepackage{graphicx}
\usepackage{bm}
\usepackage{epsf}
\usepackage{verbatim} 
\usepackage{amsmath}
\usepackage{amsfonts}
\usepackage{amssymb}
\usepackage{mathrsfs}
\usepackage{amsthm}
\usepackage{newlfont}
\usepackage[answerdelayed,lastexercise]{exercise}

\newtheorem{thm}{Theorem}[section]
\newtheorem{prop}[thm]{Proposition}
\newtheorem{lem}[thm]{Lemma}
\newtheorem{cor}[thm]{Corollary}

\theoremstyle{definition}
\newtheorem{defn}[thm]{Definition}
\newtheorem{notn}[thm]{Notation}
\newtheorem{example}[thm]{Example}

\theoremstyle{remark}
\newtheorem{rem}[thm]{Remark}

\numberwithin{equation}{section}
\numberwithin{thm}{section}

\DeclareMathOperator{\Hom}{Hom}

\DeclareMathOperator{\rank}{rank}

\DeclareMathOperator{\chr}{char}
\DeclareMathOperator{\re}{Re}
\DeclareMathOperator{\im}{Im}
\DeclareMathOperator{\Mat}{Mat}
\DeclareMathOperator{\ord}{ord}

\DeclareMathOperator{\Vol}{Vol}

\DeclareMathOperator{\GL}{GL}
\DeclareMathOperator{\PGL}{PGL}
\DeclareMathOperator{\SL}{SL}

\DeclareMathOperator{\Pic}{Pic}

\DeclareMathOperator{\Tr}{Tr}

\DeclareMathOperator{\diag}{diag}
 
\DeclareMathOperator{\Nr}{Nr}

\DeclareMathOperator{\Ver}{Ver}
\DeclareMathOperator{\type}{type}

\newcommand{\dd}{\mathrm{d}}
\newcommand{\sep}{\mathrm{sep}}
\renewcommand{\diag}{\mathrm{diag}}
\renewcommand{\div}{\mathrm{div}}

\newcommand{\Har}{\mathrm{Har}}

\newcommand{\ab}{\mathrm{ab}}

\newcommand{\tor}{\mathrm{tor}}

\newcommand{\fm}{\mathfrak{m}}
\newcommand{\fn}{\mathfrak{n}}

\newcommand{\fp}{\mathfrak{p}}

\newcommand{\cB}{\mathcal{B}}
\newcommand{\cC}{\mathcal{C}}
\renewcommand{\cD}{\mathcal{D}}
\newcommand{\cE}{\mathcal{E}}
\newcommand{\cF}{\mathcal{F}}

\newcommand{\cI}{\mathcal{I}}

\renewcommand{\cL}{\mathcal{L}}

\newcommand{\cO}{\mathcal{O}}
\newcommand{\cP}{\mathcal{P}}

\newcommand{\cS}{\mathcal{S}}

\newcommand{\cW}{\mathcal{W}}

\newcommand{\A}{\mathbb{A}}

\newcommand{\C}{\mathbb{C}}

\newcommand{\F}{\mathbb{F}}
\newcommand{\gm}{\mathbb{G}}

\newcommand{\p}{\mathbb{P}}
\newcommand{\Q}{\mathbb{Q}}
\newcommand{\R}{\mathbb{R}}

\newcommand{\Z}{\mathbb{Z}}

\newcommand{\bone}{{\bm{1}}}

\newcommand{\boz}{\boldsymbol{z}}

\newcommand{\bok}{\boldsymbol{k}}

\newcommand{\Ed}{\vec{E}}

\newcommand{\eps}{\varepsilon}
\newcommand{\G}{\Gamma}
\newcommand{\To}{\longrightarrow}
\newcommand{\bs}{\setminus}
\newcommand{\Fi}{F_\infty}
\newcommand{\Ci}{\C_\infty}

\newcommand{\bcdot}{\boldsymbol{\cdot}}

\newcommand{\La}{\Lambda}
\newcommand{\la}{\lambda}

\newcommand{\Mod}[1]{\ (\mathrm{mod}\ #1)}
\newcommand{\abs}[1]{\lvert#1\rvert}
\newcommand{\ls}[2]{#1\left(\!\left(#2\right)\!\right)}
\newcommand{\dvr}[2]{#1\left[\!\left[#2\right]\!\right]}

\newcommand{\CC}{\mathbb{C}}
\newcommand{\RR}{\mathbb{R}}
\newcommand{\QQ}{\mathbb{Q}}
\newcommand{\FF}{\mathbb{F}}
\newcommand{\ZZ}{\mathbb{Z}}
\newcommand{\nfk}{\mathfrak{n}}
\newcommand{\pfk}{\mathfrak{p}}
\newcommand{\Ical}{\mathcal{I}}
\newcommand{\Pcal}{\mathcal{P}}
\newcommand{\Ecal}{\mathcal{E}}

\begin{document}
	
\title
{Drinfeld discriminant function and Fourier expansion of harmonic cochains}

\author{Mihran Papikian}
\address{Department of Mathematics, Pennsylvania State University, University Park, PA 16802, USA}
\email{papikian@psu.edu}
\author{Fu-Tsun Wei}
\address{Department of Mathematics, National Tsing Hua University, No.\ 101, Sec.\ 2, Kuang-Fu Rd., Hsinchu 30013, Taiwan}
\email{ftwei@math.nthu.edu.tw}

\subjclass[2010]{11G09, 11G18}

\keywords{Drinfeld discriminant function; modular units; harmonic cochains; cuspidal divisor group}

\begin{abstract} Let $\Fi=\ls{\F_q}{1/T}$ be the completion of $\F_q(T)$ at $1/T$. 
	We develop a theory of Fourier expansions for harmonic cochains on the edges of the Bruhat-Tits building 
	of $\PGL_r(\Fi)$, $r\geq 2$, generalizing an earlier construction of Gekeler for $r=2$. We then apply this 
	theory to study modular units on the Drinfeld symmetric space $\Omega^r$ over $\Fi$, and the cuspidal divisor groups of Satake compactifications 
	of certain Drinfeld modular varieties. In particular, we obtain a higher dimensional analogue of a result of Ogg 
	for classical modular curves $X_0(p)$ of prime level. 
\end{abstract}

\maketitle


\section{Introduction}

\subsection{Motivation} 
 The analogies between function fields and number fields have been 
extensively studies for more than a century and led to major discoveries in number theory. In particular, 
in \cite{Drinfeld}, Drinfeld introduced certain function field analogues of abelian varieties and their moduli spaces. 
Drinfeld modular varieties and their generalizations have played a crucial role in the   
proof of the Langlands conjectures over function fields. 

Let $F$ be the field of rational functions on a projective smooth curve over a finite field $\F_q$. Let 
$\Fi$ be the completion of $F$ at a  chosen place $\infty$, and $\Ci$ 
be the completion of an algebraic closure of $\Fi$. Over $\Ci$ 
there are two different analogues of the Siegel upper half-space. 
One is the Drinfeld symmetric space $\Omega^r$, $r\geq 2$, which is a rigid analytic space over $\Fi$ of dimension $r-1$, and the 
other is the Bruhat-Tits building $\cB^r$ of $\PGL_r(\Fi)$, which provides a ``skeleton'' for $\Omega^r$. 
Correspondingly, in the function field setting there are two different analogues of Siegel modular forms: Drinfeld modular forms, 
which are $\C_\infty$-valued holomorphic functions on $\Omega^r$, and automorphic forms, which are 
$\C$-valued functions on $\cB^r$. 

In \cite{GekelerExacSeq}, \cite{GekelerDMHR1}, \cite{GekelerDMHR2}, Gekeler observed that given a nowhere vanishing 
holomorphic function $u$ on $\Omega^r$, the logarithm $\log|u|$ of its absolute value may be regarded as a 
function on $\cB^r$.  Using this observation, he constructed a homomorphism $\cP: \cO(\Omega^r)^\times\to \Har^1(\cB^r, \Z)$ 
from the group of such $u$ to the group of $\Z$-valued harmonic $1$-cochains on $\cB^r$ introduced by de Shalit in \cite{deShalit}. 
Moreover, Gekeler proved that the sequence 
\begin{equation}\label{eq-1}
0\To \Ci^\times \To \cO(\Omega^r)^\times \overset{\cP}{\To} \Har^1(\cB^r, \Z) \To 0
\end{equation}
is exact and $\GL_r(\Fi)$-equivariant. This is a higher-rank analogue of a result of van der Put \cite{vdPut} for $r=2$. 
One can think of $\cP$ as a positive-characteristic version of the logarithmic derivative over $\C$. 

One prominent example of a holomorphic nonvanishing function on $\Omega^r$ is the Drinfeld discriminant function $\Delta_r$, which  
plays a role similar to the classical discriminant $\Delta$ in the context of modular forms on $\SL_2(\Z)$. In particular, $\Delta_r$ 
can be used to construct modular units for certain congruence groups, with applications in the arithmetic 
of Drinfeld modular varieties. On the other hand, the group $\Har^1(\cB^r, \Z)$ is important because it gives a natural $\Z$-structure on the 
first $\ell$-adic cohomology of $\Omega^r$, and it is also isomorphic to a space of automorphic forms on the adele group $\GL_r(\A_F)$; 
cf. \cite{Drinfeld}, \cite{deShalit}, \cite{SS}, \cite{AA}. 

The aim of this article is to develop a theory of Fourier expansions of harmonic $1$-cochains on $\cB^r$, and 
to apply this theory to deduce interesting facts about modular units, harmonic cochains, and Drinfeld modular varieties.  
For example, we determine to what extent roots may be extracted from certain modular units arising from $\Delta_r$, 
determine an explicit set of free generators of the subgroups of $\cO(\Omega^r)^\times$ fixed by certain congruence groups, 
and determine the exact orders of cuspidal divisor groups of certain Drinfeld modular varieties.  Our results extend to arbitrary $r\geq 2$ 
the results obtained by Gekeler for $r=2$ in \cite{GekelerImproper}, \cite{GekelerDelta}. 

Next, we describe in more details the main results of the paper. But before doing so, it is 
convenient to introduce some of the notation that will be used throughout the paper. 


\subsection{Notation} 
\begin{itemize}
	\item $\F_q$ denotes the finite field with $q$ elements. 
	\item $A=\F_q[T]$ is the polynomial ring in an indeterminate $T$. Given a nonzero ideal $\fn\lhd A$, by abuse of notation, we denote 
	by the same symbol the monic generator of $\fn$. The term ``prime of $A$'' will be used to mean ``nonzero prime ideal of $A$''.  
	\item $\F_\fp=A/\fp$, where $\fp\lhd A$ is prime. 
	\item $A_+$ is the set of nonzero  monic polynomials in $A$. 
	\item $F=\F_q(T)$ is the fraction field of $A$. In this paper we will be primarily dealing with the simplest function field $\F_q(T)$, which 
	is the field of rational functions on the projective line over $\F_q$. 
	\item $\deg: F\to \Z\cup \{-\infty\}$ assigns to a nonzero polynomial its degree in $T$ and $\deg(0)=-\infty$. 
	Note that $\ord_\infty:=-\deg$ is a discrete valuation on $F$.  The corresponding place of $F$ 
	is called the \textit{place at infinity} and denoted by $\infty$. 
	\item $\Fi=\ls{\F_q}{1/T}$ is the completion of $F$ with respect to $\ord_\infty$. Note that $1/T$ is a uniformizer at $\infty$.  
	\item $O_\infty=\dvr{\F_q}{1/T}$ is the ring of integers of $\Fi$. 
	\item $\Ci$ is the completed algebraic closure of $\Fi$. 
	\item $|\cdot|$ is the normalized absolute value on $\Fi$, extended to $\Ci$. In particular, for a nonzero $a\in A$ we have $\abs{a}=q^{\deg(a)}=\#A/(a)$ .  
	\item $r\geq 2$ is a fixed integer. 
	\item $\Omega^r$ is the Drinfeld symmetric space
	$$
	\left \{ (z_1, \dots, z_r)\in \Ci^r \mid \text{$z_1, \dots, z_r$ are linearly independent over $\Fi$ and $z_r=1$}\right \}. 
	$$ 
	\item $\cB^r$ is the Bruhat-Tits building of $\PGL_r(\Fi)$; see Section \ref{sHC} for a detailed description. 
	\item 
	$
	\cI^1=\left\{\begin{pmatrix} a & b \\ c & d\end{pmatrix} \in \GL_r(O_\infty)\ \bigg|\ c=\begin{pmatrix} c_2\\ \vdots\\ c_r\end{pmatrix} 
	\equiv \begin{pmatrix} 0\\ \vdots\\ 0\end{pmatrix}\Mod{T^{-1}} \right\}$. 
	\item $P(\Fi)=\begin{pmatrix} 1 & \ast \\ 0 & \GL_{r-1}(\Fi)\end{pmatrix}\subset \GL_r(\Fi)$. 
	\item $\G_0^r(\fn)=\left\{ \begin{pmatrix} a & b \\ c & d\end{pmatrix}\in \GL_r(A) \quad \bigg | \quad c=\begin{pmatrix} c_2\\ \vdots\\ c_r\end{pmatrix} 
	\equiv \begin{pmatrix} 0\\ \vdots\\ 0\end{pmatrix}\Mod{\fn} \right\}$, $0\neq \fn\in A$. 
	\item $\G_\infty=\begin{pmatrix} \GL_1(A) & \ast \\ 0 & \GL_{r-1}(A)\end{pmatrix}\subset \G_0^r(\fn)$. 
	
\end{itemize}

\subsection{Main results} In Section \ref{sHC}, we work over an arbitrary local field. Here 
we recall de Shalit's definition of harmonic 1-cochains on $\cB^r$ 
with values in an abelian group $R$. In fact, de Shalit \cite{deShalit} defined harmonic $s$-cochains for any $1\leq s\leq r-1$. These are functions 
on pointed $s$-simplices of $\cB^r$ satisfying four conditions. The significance of the group $\Har^s(\cB^r, R)$  of $R$-valued 
harmonic $s$-cochains is that it is isomorphic to the $\ell$-adic cohomology group $H^s_{\mathrm{et}}(\Omega^r, \Z_\ell)$ of $\Omega^r$ when $R=\Z_\ell$. 
(For local fields of characteristic $0$ this is proved in \cite{deShalit}, and for local fields of positive characteristic $p\neq \ell$ this 
follows from \cite{AA} and \cite{SS}.)  Also, $\Har^s(\cB^r, R)$ can be interpreted as a space of automorphic forms  
on the adele group $\GL_r(\A_F)$ with a special restriction at $\infty$; see \cite{AA} and \cite{SS}.
In this paper we will be dealing only with harmonic 1-cochains, which are functions on oriented 
edges of $\cB^r$. When $r=2$, harmonic 1-cochains correspond to ``flows'' on a tree and their 
story is much older (see \cite{Drinfeld}, \cite{SerreT}). Three of the conditions defining harmonic 1-cochains on $\cB^r$ for $r\geq 3$ 
are natural extensions of the corresponding conditions for $r=2$, but one of the conditions ((3) of Definition \ref{defnHarm-1}) 
is not visible when $r=2$. This condition guarantees that a harmonic 1-cochain on $\cB^r$ is uniquely determined by 
its restriction to the edges of type 1. On the other hand, the set of type-1 edges of $\cB^r$ is in natural bijection with the set of cosets $\GL_r(\Fi)/\Fi^\times \cI^1$. 
Our main result in Section \ref{sHC} is a translation of harmonicity conditions to $R$-valued functions on $\GL_r(\Fi)/\Fi^\times \cI^1$ 
(see Proposition \ref{propHarGLr}).  This is used in Section \ref{sFEHC}. 

In  Section \ref{sFEHC}, we define 
the Fourier expansion of a $\C$-valued function $h$ on $P(\Fi)$, left invariant under $\G_\infty$:
$$
h\left(\begin{pmatrix} 1 & \vec{x}y \\ & y\end{pmatrix}\right)=\sum_{\vec{a}\in A^{r-1}} h^*(\vec{a}, y)\psi(\vec{a}\vec{x}^t),
$$
where $\psi: \Fi\to \C^\times$ is an additive character. 
  Our construction generalizes a construction of Weil \cite{Weil} and Gekeler \cite{GekelerImproper} for $r=2$ (see also \cite{Pal}).  
We then translate  
the harmonicity condition of a $\C$-valued function on $\G_\infty\bs \GL_r(\Fi)/\Fi^\times \cI^1$ into a condition on its Fourier coefficients 
(see Proposition \ref{prop3.10}).  It essentially says that $h$ is harmonic if and only if $h^*(\vec{a}, y)$ depends only on $\abs{\det(y)}$, 
but not on the actual matrix $y$. 
Although this result is not essential for our later purposes, it is certainly of independent interest, since it gives a novel perspective 
on harmonic $1$-cochains. 

In Section \ref{sDDF}, we recall the definition of the Drinfeld discriminant function $\Delta_r$ and the Gekeler--van der Put map $\cP$ 
from \eqref{eq-1}. To a point $\boz=(z_1, \dots, z_r)\in \Omega^r$ one associates the lattice $\La_{\boz}=Az_1+\cdots+Az_r\subset \Ci$. 
On the other hand, the theory of Drinfeld modules associates to $\La_{\boz}$ a Drinfeld module $\phi^{\boz}$ of rank $r$ over $\Ci$, defined 
by an $\F_q$-linear polynomial  
$$
\phi^{\boz}_T=Tx+g_1(\boz)x^q+\cdots+g_r(\boz)x^{q^r}. 
$$
The $\Ci$-valued function $\Delta_r(\boz):=g_r(\boz)$ on $\Omega^r$ is the Drinfeld discriminant function. It is a Drinfeld cusp  
form for $\GL_r(A)$ of weight $q^r-1$. Moreover, $\Delta_r$ does not vanish on $\Omega^r$.  

In Section \ref{sES}, we compute the Fourier coefficients of $\cP(\Delta_r)$. The Kronecker Limit Formula (see Theorem \ref{thmKLF}) proved 
by the second author in \cite{WeiKLF} is essential for this calculation. It reduces the calculation of the 
Fourier coefficients of $\cP(\Delta_r)$ to a calculation of the Fourier coefficients of a certain Eisenstein series. The main result 
of this section is Theorem \ref{thmFEPDelta}, which gives a relatively simple expression for the Fourier coefficients of $\cP(\Delta_r)$. 
One interesting corollary of this theorem is the following. Breuer and Basson \cite{BB}, and independently Gekeler \cite{GekelerDMHR1}, 
have shown that there exists a holomorphic function $H$ on $\Omega^r$ such that $H^{q-1} =\Delta_r$. Using the Fourier 
coefficients of $\cP(\Delta_r)$, we prove that $q-1$ is the largest integer $m$ such that  there exists an $m$-th root of 
$\Delta_r$ in $\cO(\Omega^r)^\times$, i.e., the result of Breuer, Basson, and Gekeler is optimal. 

In Section \ref{sMU}, for a given nonzero $\fn\lhd A$ we define the modular unit 
$$\Theta_\fn(\boz)=\Delta_r(z_1,z_2, \dots, z_r)/\Delta_r(\fn z_1,z_2, \dots, z_r), \quad \boz=(z_1,z_2, \dots, z_r)\in \Omega^r.$$ 
It is easy to show that $\Theta_\fn(\boz)$ is invariant under the action of $\G_0^r(\fn)$ on $\Omega^r$. We compute the 
Fourier coefficients of $\cP(\Theta_\fn)$ using our earlier calculation of the Fourier coefficients of $\cP(\Delta_r)$. 
Then, using the Fourier coefficients of $\cP(\Theta_\fn)$ and some auxiliary arguments, we prove the following:
\begin{itemize}
	\item[(i)] The largest integer $m$ such that there exists an $m$-th root of $\Theta_\fn$ in $\cO(\Omega^r)^\times$ 
	is $$(q-1)\cdot \gcd(q^r-1, |\fn|-1).$$
	\item[(ii)] If $\fn$ is square-free, then the largest integer $m$ such that there exists an $m$-th root of $\Theta_\fn$ in $(\cO(\Omega^r)^\times)^{\G_0^r(\fn)}$ 
	is $\gcd(q^r-1, |\fn|-1).$
\end{itemize}

In Section \ref{sCD}, we combine the results of Section \ref{sMU} with some geometric arguments to prove two theorems 
about the group of $\G_0^r(\fn)$-invariant modular units $(\cO(\Omega^r)^\times)^{\G_0^r(\fn)}$, and the Satake compactification 
$X_0^r(\fn)$ of the Drinfeld modular variety $Y_0^r(\fn):=\G_0^r(\fn)\bs \Omega^r$. 
The modular variety $Y_0^r(\fn)$ is affine of dimension $r-1$, as was proved by Drinfeld \cite{Drinfeld}. Its 
Satake compactification $X_0^r(\fn)$ is obtained by gluing to $Y_0^r(\fn)$ certain $(r-2)$-dimensional irreducible varieties, called \textit{cusps}.  
The Satake compactifications 
of Drinfeld modular varieties were constructed (at different levels of generality and details of proof) by 
Gekeler \cite{GekelerSatake}, \cite{GekelerDMHR4}, 
Kapranov \cite{Kapranov},  Pink \cite{Pink}, and H\"aberli \cite{Haberli}.
In Section \ref{sCD}, we first show that a function $f\in (\cO(\Omega^r)^\times)^{\G_0^r(\fn)}$ 
is always meromorphic at the cusps. (This result might be of independent interest.) Then, using the orders of modular units at the 
cusps and the Fourier coefficients of $\cP(\Theta_\fn)$, we prove the following: 

\begin{thm}\label{thmA}
	Assume $\fn$ is square-free with $s$ prime divisors. The group $(\cO(\Omega^r)^\times)^{\G_0^r(\fn)}/\Ci^\times$ 
	is a free abelian group of rank $2^s-1$. Moreover, the $2^s-1$ modular units  
	$$\{\Theta_\fm \mid \fm\neq 1 \text{ is monic and divides $\fn$}\}$$ 
	generate a subgroup of finite index in $(\cO(\Omega^r)^\times)^{\G_0^r(\fn)}/\Ci^\times$. 
\end{thm}

The \textit{cuspidal divisor group} $\cC^r(\fn)$ of $X^r_0(\fn)$ is the subgroup  
of the divisor class group of $X^r_0(\fn)$ generated by the differences of the cusps.  The second main theorem, 
whose proof uses Theorem \ref{thmA} and the result about the roots of modular units from Section \ref{sMU}, is the following: 
 
 \begin{thm}\label{thmB}
 	Assume $\fn$ is square-free. The group $\cC^r(\fn)$  is finite. Moreover, if $\fn=\fp$ is prime, 
 	then $\cC^r(\fp)$ is a cyclic group of order 
 	\begin{equation}\label{eqCD1}
 	\frac{|\fp|^{r-1}-1}{\gcd(q^r-1, |\fp|-1)}. 
 	\end{equation}
 \end{thm}

The more general result that $\cC^r(\fn)$ is finite for any $\fn$ was proved by Kapranov \cite{Kapranov} 
by a different argument, but the fact that $\cC^r(\fp)$ is cyclic of given order is, as far as we know, the 
first example where the cuspidal divisor group is computed explicitly in higher dimensions.   

\subsection{Some remarks about cuspidal divisor groups}
The cuspidal divisor groups have a long history and have played an important role in the arithmetic of modular curves. 
In the early 1970s, Ogg computed that the cuspidal divisor group of the classical modular curve $X_0(p)$ of prime level $p$
is cyclic of order 
\begin{equation}\label{eqCD2}
\frac{p-1}{\gcd(12, p-1)}. 
\end{equation}
He also conjectured that this cuspidal divisor group is the full $\Q$-rational torsion subgroup of the Jacobian variety of $X_0(p)$; see \cite{OggBAMS}. 
In his seminal paper \cite{MazurEis}, Mazur proved Ogg's conjecture by developing a comprehensive theory of what he called the 
\textit{Eisenstein ideal} of the Hecke algebra of prime level. 

For the Drinfeld modular curve $X_0(\fp):=X_0^2(\fp)$ of prime level $\fp$, Gekeler proved the analogue of Ogg's result 
(see \cite{GekelerUber}, \cite{GekelerDelta}): the cuspidal divisor group is cyclic of order 
\begin{equation}\label{eqCD3}
\frac{|\fp|-1}{\gcd(q^2-1, |\fp|-1)}.   
\end{equation}
Moreover, P\'al  \cite{Pal} proved the analogue of Mazur's result about the rational torsion subgroup of the Jacobian variety of $X_0(\fp)$, 
i.e., $\mathrm{Jac}(X_0(\fp))(F)_\tor=\cC^2(\fp)$.  

Comparing \eqref{eqCD1}, \eqref{eqCD2}, and \eqref{eqCD3}, we see a striking similarity, especially when taking into account the fact that 
$12$ is the weight of the classical discriminant function and $q^r-1$ is the weight of $\Delta_r$. Then one might wonder whether 
there is an analogue of Mazur's result for $\cC^r(\fp)$ when $r\geq 3$, e.g., whether $\cC^r(\fp)$ is the full rational torsion subgroup of $\Pic^0(X_0^r(\fp))$. 
We do not know the answer to this question. On the other hand, there is probably no good substitute for the Eisenstein ideal in this context, 
as we explain below.  

As we already mentioned, the space $\Har^1(\cB^r, \Z)^{\G_0^r(\fp)}$ of harmonic $1$-cochains invariant under the action of $\G_0^r(\fp)$ on $\cB^r$ 
can be interpreted as a space of automorphic forms; cf.\ \cite{AA}. When $r=2$, the Eisenstein ideal essentially measures the 
congruences between the cusp forms in $\Har^1(\cB^r, \Z)^{\G_0^r(\fp)}$ and the Eisenstein series. One exploits this interplay 
between the cusp forms and the Eisenstein series 
to deduce results about the geometry of $\mathrm{Jac}(X_0(\fp))$; cf.\ \cite{Pal}. On the other hand, when $r\geq 3$, 
the map $$(\cO(\Omega^r)^\times)^{\G_0^r(\fn)}\overset{\cP}{\To} \Har^1(\cB^r, \Z)^{\G_0^r(\fp)}$$
obtained from \eqref{eq-1} by taking $\G_0^r(\fn)$-invariants has finite cokernel. This is a consequence of a 
deep property of lattices of semi-simple groups of rank at least 2 over local fields (see Theorem \ref{thmKazdanT}). 
Thus, the rank of $\Har^1(\cB^r, \Z)^{\G_0^r(\fp)}$ is $1$ and it is generated by the image of a root of $\Theta_\fp$. 
Since $\cP(\Theta_\fp)$ is an Eisenstein series, we conclude that there are no cuspidal harmonic $1$-cochains invariant under the 
action of $\G_0^r(\fp)$. (To obtain a good analogue of the Eisenstein ideal one probably has to work 
with $\Har^{r-1}(\cB^r, \Z)^{\G_0^r(\fp)}$ and $H^{r-1}_{\mathrm{et}}(X_0^r(\fp), \Z_\ell)$.)


\section{Harmonic $1$-cochains}\label{sHC}

Let $K$ be a non-archimedean local field with ring of integers $O$, a fixed uniformizer $\pi$, 
and finite residue field $O/(\pi)=\F_q$ of cardinality $q$. We normalize the absolute value $|\cdot|$ 
on $K$ by $|\pi|=q^{-1}$. 

Let $r\geq 2$ be an integer. Let $V$ be an $r$-dimensional vector space over $K$. 
An $O$-lattice in $V$ is a free $O$-module of rank $r$ 
which contains a basis of $V$.  Two lattices $L_1$ and $L_2$ are \textit{similar} if there exists $\alpha\in K^\times$ 
with $\alpha\cdot L_1=L_2$. This defines an equivalence relation on the set of lattice in $V$. 
We denote the equivalence class of $L$ by $[L]$. Since $K$ is a local field, $[L]$ can be identified 
with $\{\pi^i L\mid i\in \Z\}$. 
The \textit{Bruhat-Tits building} of $\PGL_r(K)$ is a simplicial complex $\cB=\cB^r$ 
with set of vertices 
$$
\Ver(\cB)=\{[L]\mid L \text{ is a lattice in }V\}. 
$$
The vertices $[L_0], \dots, [L_n]$ form an $n$-simplex if and only if 
there is $L'_i\in [L_i]$, $0\leq i\leq n$, such that 
$$L_0'\supsetneq L_1'\supsetneq\cdots\supsetneq L'_n\supsetneq \pi L_0'.$$ 
Thus, the simplicial complex $\cB$ has dimension $r-1$.

Let 
$$
\Ed(\cB)=\{([L_0], [L_1])\mid L_0\supsetneq L_1\supsetneq \pi L_0\}
$$ be the set of \textit{oriented edges}  of $\cB$ (equiv.\ \textit{pointed $1$-simplices} in the terminology of 
\cite{deShalit}, or \textit{arrows} in the terminology of \cite{GekelerExacSeq}). For $e=(v_0, v_1)=([L_0], [L_1])\in \Ed(\cB)$ 
we write  
\begin{itemize}
	\item[] $o(e)=v_0$, origin of $e$; 
	\item[] $t(e)=v_1$, terminus of $e$; 
	\item[] $\bar{e}=(v_1, v_0)$, $e$ with reverse orientation; 
	\item[] $\type(e)=\dim_{\F_q}(L_0/L_1)$, type of $e$.  
\end{itemize}
Note that $1\leq \type(e)\leq r-1$ and $\type(\bar{e})=r-\type(e)$. We decompose $\Ed=\Ed(\cB)$ into a disjoint union 
$$
\quad \quad \Ed=\coprod_{1\leq s\leq r-1} \Ed_s, \quad \quad \text{where} \quad \Ed_s=\{e\in \Ed\mid \type(e)=s\}.
$$
The following definition is due to de Shalit  \cite[$\S$3.1]{deShalit}: 

\begin{defn}\label{defnHarm-1} 
    A \textit{harmonic $1$-cochain} on $\cB$ with values in an abelian group $R$ is a function $h: \Ed\to R$ satisfying: 
	\begin{enumerate}
		\item  For any $e\in \Ed$, $$h(e)=-h(\bar{e}).$$ 
		\item For $v\in \Ver(\cB)$ and $1\leq s\leq r-1$,  put 
		$\Ed_s^t(v):=\{e\in \Ed_s\mid t(e)=v\}$.
		Then  
		$$
		\sum_{e\in \Ed_s^t(v)} h(e)=0\quad \text{for all $v$ and $s$}. 
		$$
		\item For $e=([L_0], [L_1])\in \Ed$, put 
		$$
		\Ed_1^\triangle(e):=\{e'=([L_0'], [L_1'])\in \Ed_1\mid L_0\supset L_0'\supset L_1'=L_1\}, 
		$$ 
		i.e., $e'\in \Ed_1^\triangle(e)$ if $t(e)=t(e')$ and $e, e'$ lie in a common $2$-simplex.
		Then
		$$
		\sum_{e'\in \Ed_1^\triangle(e)}h(e')=h(e)\quad \text{ for all $e$}.
		$$
		\item For any pointed $2$-simplex $(v_0, v_1, v_2)$ of $\cB$, 
		$$
		h((v_0, v_1))+h((v_1, v_2))=h((v_0, v_2)). 
		$$
	\end{enumerate}
    We denote the space of $R$-valued harmonic $1$-cochains by $\Har^1(\cB, R)$. 
\end{defn}

\begin{rem}\label{remDefdS} \hfill
	\begin{enumerate}
		\item[(i)] It is easy to show that conditions (1) and (4) are equivalent to 
		$$
		\sum h(e)=0, 
		$$ 
		whenever $e$ ranges over the oriented edges of a closed path in $\cB$. 
		\item[(ii)]  Condition (3) implies that a harmonic $1$-cochain is uniquely determined by its values 
		on $\Ed_1$. Note that if $e\in \Ed_1$, then $\Ed_1^\triangle(e)=\{e\}$, so this condition does not impose 
		additional restrictions on the values of $h$ on $\Ed_1$. 
		As is explained in \cite[p. 141]{deShalit}, condition (3) does not follow from (1), (2), and (4). 
	\item[(iii)] From condition (3), we may reformulate condition (2) as
	$$
	\sum_{e \in \Ed_1^t(v)}h(e) = 0, \quad \text{for all } v \in \Ver(\cB).
	$$
	\item[(iv)] For $v \in \Ver{\cB}$, put $\Ed_s^o(v)=\{e\in \Ed_s\mid o(e)=v\}$. Assuming (1), condition (2) is equivalent to  
	$$
	\sum_{e\in \Ed_s^o(v)} h(e)=0, \quad \text{ for all $v\in \Ver(\cB)$ and $1\leq s\leq r-1$}.
	$$
    This is actually the condition that appears in \cite{deShalit}. 
    \item[(v)] In the $r=2$ case, harmonic $1$-cochains on $\cB$ already appear in \cite{Drinfeld}. 
 	\end{enumerate}
\end{rem}

We regard $V$ as a space of row vectors on which $\GL_r(K)$ acts as a matrix group from the right. Hence $\GL_r(K)$ 
acts also from the right on $\cB$. (If the syntax requires a left action, we shift this action to the left by the usual formula 
$g x:=xg^{-1}$.) Let $$\La_0:=O^r\subset K^r=V.$$
Since $\GL_r(K)$ acts transitively on $\Ver(\cB)$ and the stabilizer of $[\La_0]$ is $K^\times\GL_r(O)$, we have a bijection 
$$
\GL_r(K)/K^\times\GL_r(O)\overset{\sim}{\To}\Ver(\cB), \quad g\longmapsto [\La_0g^{-1}].
$$ 
Similarly, $\GL_r(K)$ acts transitively on $\Ed_s$, $1\leq s\leq r-1$. Denote
$$
e^s:=\big([\La_0], [\La_0\begin{pmatrix} \pi I_s & \\ & I_{r-s}\end{pmatrix}]\big)\in \Ed_s. 
$$ 
Let $\cI^s$ be the $(s, r-s)$-parahoric subgroup of $\GL_r(O)$, i.e.
$$
\cI^s:=\left\{\begin{pmatrix} a & b \\ c & d\end{pmatrix} \in \GL_r(O)\ \bigg|\ c\in \Mat_{r-s, s}(\pi O)\right\}.
$$
Since the stabilizer of $e^s$ in $\GL_r(K)$ is $K^\times \cI^s$, we have a bijection 
\begin{align*}
\GL_r(K)/K^\times\cI^s & \ \overset{\sim}{\To}\ \Ed_s,\\ 
g&\ \longmapsto\ e^s_g:=ge^s=\big([\La_0 g^{-1}], [\La_0\begin{pmatrix} \pi I_s & \\ & I_{r-s}\end{pmatrix}g^{-1}]\big).
\end{align*}

For $h\in \Har^1(\cB, R)$, let $h_s:=h\big|_{\Ed_s}$ denote the restriction of $h$ to $\Ed_s$. The above 
bijection enables us to view $h_s$ as a function on $\GL_r(K)$, left invariant under $K^\times \cI^s$.  
We shall translate the conditions (1)--(4) into the corresponding matrix operations for the functions $h_1$,...,$h_{r-1}$ on $\GL_r(K)$.

For $1\leq i\leq r$, denote 
$$
T_i:=\left\{ \begin{pmatrix}  & & I_{r-i}\\ \pi & \vec{u} & \\ & I_{i-1} &\end{pmatrix}\ \Bigg|\ \vec{u}\in \F_q^{i-1}\right\}. 
$$
Condition (2) of Definition \ref{defnHarm-1} implies:

\begin{lem}\label{lem:har-1-1} Let $h\in \Har^1(\cB, R)$. 
	For all $g\in \GL_r(K)$, we have 
	$$
	\sum_{1\leq i\leq r}\sum_{\alpha\in T_i} h_1(g\alpha)=0. 
	$$
\end{lem}
\begin{proof}
	Let $v=[\La_0 g^{-1}]$. It is enough to show 
	that  
	$$
	\Ed_1^t(v)=\{e^1_{g\alpha}\mid \alpha\in T_i, 1\leq i\leq r\}. 
	$$
	Note that $\#(\Ed_1^t(v))$ is equal to the number of $1$-dimensional subspaces of $\F_q^r$. Thus, we have
	$\#(\Ed_1^t(v))=(q^r-1)/(q-1)$. 
	On the other hand, since $\# (T_i)=q^{i-1}$, we get
	$$
	\sum_{i=1}^r \# (T_i) = \#(\Ed_1^t(v)). 
	$$
	Next, by a simple calculation, $\alpha\begin{pmatrix} \pi^{-1} & \\ & I_{r-1}\end{pmatrix}\in \GL_r(O)$ for any $\alpha\in T_i$, so 
	$$
	t(e^1_{g\alpha})=[\La_0 \begin{pmatrix} \pi & \\ & I_{r-1}\end{pmatrix} \alpha^{-1}g^{-1}]=[\La_0 g^{-1}]=v.
	$$
	Finally, note that 
	$$
	\begin{pmatrix}  & & I_{r-i}\\ \pi & \vec{u} & \\ & I_{i-1} &\end{pmatrix}^{-1}=\begin{pmatrix}  
	& \pi^{-1} & -\vec{u}\cdot \pi^{-1}\\ & & I_{i-1}\\ I_{r-i}&  &\end{pmatrix}. 
	$$
	Using this, a straightforward calculation shows that if $\alpha\neq \beta$ then 
	$\alpha^{-1}\beta\not\in K^\times \GL_r(O)$  for any $\alpha\in T_i$, $\beta\in T_j$, $1\leq i, j\leq r$. 
	 Thus, distinct $\alpha$'s represent different cosets in $\GL_r(K)/K^\times\GL_r(O)$, so that $e^1_{g\alpha}$'s are 
	distinct type-$1$ edges terminating at $v$. 
\end{proof}

Given $g\in \GL_r(K)$ and $1\leq s\leq r-1$, let 
$$
g_s':=g\cdot \begin{pmatrix} & I_{r-s}\\ \pi I_s &\end{pmatrix}. 
$$
Condition (3) of Definition \ref{defnHarm-1} implies:
\begin{lem}\label{lem-harm-1-4} Let $h\in \Har^1(\cB, R)$. 
	For all $g\in \GL_r(K)$ and $1\leq s\leq r-1$, we have 
	$$
	\sum_{1\leq i\leq s}\sum_{\alpha\in T_i} h_1(g\alpha)=h_s(g_s'). 
	$$
\end{lem}
\begin{proof}
	Note that 
	$$
	\begin{pmatrix} \pi I_s & \\ & I_{r-s}\end{pmatrix} \begin{pmatrix} & I_{r-s}\\ \pi I_s &\end{pmatrix}^{-1} = 
	\begin{pmatrix} & I_{r-s}\\ I_s &\end{pmatrix}\in \GL_r(O), 
	$$
	so 
	\begin{equation}\label{eqe_{g_s'}^s}
	e_{g_s'}^s=([\La_0 (g_s')^{-1}], [\La_0 g^{-1}]). 
	\end{equation}
	In particular,  $t(e_{g_s'}^s)=[\La_0 g^{-1}]$. 
	Now it is enough to show that 
	\begin{equation}\label{eqEd_1(e_{g_s'}^s)}
	\Ed_1^\triangle(e_{g_s'}^s)=\{e_{g\alpha}^1\mid \alpha\in T_i \text{ with }1\leq i\leq s\}, 
	\end{equation}
	as then the claim of the lemma follows from condition (3) of Definition \ref{defnHarm-1}. 

    As in the proof of Lemma \ref{lem:har-1-1}, for an edge $e$ of type $s$ the order of the set $\Ed_1^\triangle(e)$ is equal to the number of $1$-dimensional subspaces in $\F_q^s$. 
    Thus, 
    $$
    \#(\Ed_1(e))=(q^s-1)/(q-1)=\sum_{i=1}^s \#(T_i). 
    $$
    On the other hand, one checks by a direct calculation that for any $\alpha\in T_i$ with $1\leq i\leq s$ we have the inclusions of lattices 
    $$
    \La_0 \cdot \begin{pmatrix} & \pi^{-1}I_s \\ I_{r-s} & \end{pmatrix} \alpha \supsetneq \La_0\supsetneq \La_0\cdot \alpha. 
    $$
Therefore, 
$$
\La_0\cdot  (g_s')^{-1}\supsetneq \La_0\cdot (g\alpha)^{-1}\supsetneq \La_0\cdot g^{-1},
$$
which means that 
$$
e_{g\alpha}^1=([\La_0 (g\alpha)^{-1}], [\La_0 g^{-1}])\in \Ed_1^\triangle([\La_0 (g_s')^{-1}], [\La_0 g^{-1}])=\Ed_1^\triangle(e_{g_s'}^s). 
$$
In the proof of Lemma \ref{lem:har-1-1} we have already established that the edges $e_{g\alpha}^1$ are distinct 
for distinct $\alpha$'s, so \eqref{eqEd_1(e_{g_s'}^s)} follows. 
\end{proof}

Conditions (1) and (4) of Definition \ref{defnHarm-1} imply: 
\begin{lem}\label{lem-harm-1-5} Let $h\in \Har^1(\cB, R)$. Put $h_0=0$. 
	For all $g\in \GL_r(K)$ and $1\leq s\leq r-1$, we have 
	\begin{equation}\label{eq1-lem-harm-1-5}
h_s(g_s')-h_{s-1}(g_{s-1}')=h_1(g_s'),
	\end{equation}
	and 
	\begin{equation}\label{eq2-lem-harm-1-5}
	h_{r-1}(g_{r-1}')=-h_1(g). 
	\end{equation}
\end{lem}
\begin{proof}
	Since the claim is trivial for $s=1$, we assume $s\geq 2$. By \eqref{eqe_{g_s'}^s}, 
	the oriented edges $e_{g_s'}^s$ and $e_{g_{s-1}'}^{s-1}$ have the same terminus $[\La_0 g^{-1}]$. Moreover, 
	$$
	t(e_{g_s'}^1)=o(e_{g_{s-1}'}^{s-1}),
	$$
	since 
	\begin{align*}
	g_s' \cdot \begin{pmatrix} \pi^{-1} & \\ & I_{r-1}\end{pmatrix} & = 
	g\cdot \begin{pmatrix}  & & I_{r-s}\\ 1 & & \\ & \pi I_{s-1} & \end{pmatrix} \\ 
	& = g\cdot \begin{pmatrix}  & I_{r-s+1}\\ \pi I_{s-1} & \end{pmatrix}\cdot 
	\begin{pmatrix}  & I_{r-1}\\ 1 & \end{pmatrix} \\ 
	& \equiv g_{s-1}'\mod \GL_r(O). 
	\end{align*}
	Since $e_{g_s'}^1$ and $e_{g_s'}^s$ have the same origin, condition (4) of Definition \ref{defnHarm-1} gives 
	\begin{align*}
	h_s(g_s')=h(e_{g_s'}^s) &=h(e_{g_s'}^1)+h(e_{g_{s-1}'}^{s-1}) \\ 
	& =h_1({g_s'})+h_{s-1}(g_{s-1}'). 
	\end{align*}
	This proves the first claim. On the other hand, it is easy to check that 
	$$
	e_{g_{r-1}'}^{r-1}=\overline{e_g^1}. 
	$$
	Therefore, condition (1) of Definition \ref{defnHarm-1} gives
	$$
	h_{r-1}(g_{r-1}')=h(e_{g_{r-1}'}^{r-1})=-h(e_g^1)=-h_1(g), 
	$$
	which is the second claim. 
\end{proof}

The above lemmas then lead us to the following properties of the function $h_1$ for a given $h \in \Har^1(\cB,R)$:

\begin{lem}
    Given $h \in \Har^1(\cB,R)$, the associated function $h_1$ on $\GL_r(K)/K^\times \cI^1$ satisfies:
    \begin{equation}\label{eq-prop-harm-1}
	\sum_{u\in \F_q^{s-1}} h_1\left(g\begin{pmatrix}\pi & \vec{u} & \\ & I_{s-1} & \\ & & I_{r-s} \end{pmatrix}\right)
	=h_1\left(g\begin{pmatrix}\pi I_s& \\ & I_{r-s} \end{pmatrix}\right), 
	\end{equation}
	for all $g\in \GL_r(K)$ and $1\leq s\leq r$;  
	\begin{equation}\label{eq-prop-harm-2}
	\sum_{1\leq s\leq r} h_1\left(g\begin{pmatrix} & I_{r-s}\\ \pi I_s& \end{pmatrix}\right)=0, 
	\end{equation}
	all $g\in \GL_r(K)$.
\end{lem}

\begin{proof}
    Let $h\in \Har^1(\cB, R)$.
	For $1\leq s\leq r-1$, we have  
	$$
	\sum_{\alpha\in T_s}h_1(g\alpha)= h_s(g_s')-h_{s-1}(g_{s-1}')=h_1(g_s'),
	$$
where the first equality follows from Lemma \ref{lem-harm-1-4} and the second from Lemma \ref{lem-harm-1-5}. 
Thus, 
\begin{equation}\label{eq-proofLem1.6}
\sum_{u\in \F_q^{s-1}} h_1\left(g\begin{pmatrix} &  & I_{r-s}\\ \pi & \vec{u} & \\ & I_{s-1}&  \end{pmatrix}\right)
=h_1\left(g\begin{pmatrix} &  I_{r-s} \\ \pi I_s&\end{pmatrix}\right). 
\end{equation}
Since 
$$
\begin{pmatrix} &  I_{s} \\ I_{r-s} & \end{pmatrix}\cdot \begin{pmatrix} &  & I_{r-s}\\ \pi & \vec{u} & \\ & I_{s-1} &  \end{pmatrix}
=\begin{pmatrix}  \pi & \vec{u} & \\  &  I_{s-1} & \\&  & I_{r-s} \end{pmatrix}
$$
and 
$$
\begin{pmatrix} &  I_{s} \\ I_{r-s} & \end{pmatrix}\cdot \begin{pmatrix} &  I_{r-s} \\ \pi I_{s} & \end{pmatrix} 
= \begin{pmatrix} \pi I_{s} &  \\  &  I_{r-s} \end{pmatrix},
$$
replacing $g$ by $g\cdot \begin{pmatrix} &  I_{s} \\ I_{r-s} & \end{pmatrix}$ in \eqref{eq-proofLem1.6} we then
obtain \eqref{eq-prop-harm-1} for $1\leq s\leq r-1$. 

For the $s=r$ case, note that 
\begin{align*}
0 &=\sum_{1\leq i\leq r}\sum_{\alpha\in T_i} h_1(g\alpha) \qquad \text{(Lemma \ref{lem:har-1-1})} \\ 
& = \sum_{\alpha\in T_r} h_1(g\alpha)+\sum_{1\leq i\leq r-1}\sum_{\beta\in T_i} h_1(g\beta) \\ 
& =\sum_{\alpha\in T_r} h_1(g\alpha)+h_{r-1}(g_{r-1}') \qquad \text{(Lemma \ref{lem-harm-1-4})} \\ 
&= \sum_{\alpha\in T_r} h_1(g\alpha)-h_{1}(g)\qquad \text{(Lemma \ref{lem-harm-1-5})}. 
\end{align*}
Therefore, 
$$
\sum_{\vec{u}\in \F_q^{r-1}} h_1\left(g\begin{pmatrix}\pi & \vec{u}\\ & I_{r-1}\end{pmatrix}\right) = \sum_{\alpha\in T_r} h_1(g\alpha) = h_1(g)=h_1(g(\pi I_r)),
$$
where in the last equality we have used the invariance of $h_1$ under the action of scalar matrices. 
This proves \eqref{eq-prop-harm-1} for $s=r$. 

To show that $h_1$ satisfies \eqref{eq-prop-harm-2}, recall that $g_s'=g\cdot \begin{pmatrix} & I_{r-s}\\ \pi I_s &\end{pmatrix}$. Now 
\begin{align*}
\sum_{1\leq s\leq r-1} h_1\left(g\begin{pmatrix} & I_{r-s}\\ \pi I_s& \end{pmatrix}\right) 
& = \sum_{1\leq s\leq r-1} (h_s(g_s')-h_{s-1}(g_{s-1}')) \qquad \text{(Lemma  \ref{lem-harm-1-5})} \\ 
&=h_{r-1}(g_{r-1}') \\ 
&=-h_1(g) \qquad \text{(Lemma  \ref{lem-harm-1-5})} \\ 
& = -h_1(g(\pi I_r)),  
\end{align*}
which obviously can be rewritten as \eqref{eq-prop-harm-2}. 
\end{proof}

In fact, we obtain the following:

\begin{prop}\label{propHarGLr}
    The map $h\mapsto h_1$ gives a $\GL_r(K)$-equivariant isomorphism between $\Har^1(\cB, R)$ 
	and the space of $R$-valued functions on $\GL_r(K)/K^\times \cI^1$ satisfying
	\eqref{eq-prop-harm-1} and \eqref{eq-prop-harm-2}.
\end{prop}

\begin{proof} 
Since $h\in \Har^1(\cB, R)$ is uniquely determined by its values on type-1 edges, the map $h\mapsto h_1$ from $\Har^1(\cB, R)$ into the space 
of $R$-valued functions on $\Ed_1$ is injective. It remains to show that if a given function $f:\Ed_1\to R$  
satisfies \eqref{eq-prop-harm-1} and \eqref{eq-prop-harm-2}, 
then $f=h_1$ for some harmonic $1$-cochain $h$. 

We define a function $h:\Ed\to R$ associated to $f$ by 
\begin{equation}\label{eqDefn-h}
h(e)=\sum_{e'\in\Ed_1(e)^\triangle} f(e'). 
\end{equation}
Hence $h_1 = f$, and (3) in Definition \ref{defnHarm-1}  is automatically satisfied.

For an edge $e \in \Ed_s$ with $1\leq s \leq r-1$,
take $g \in \GL_r(K)$ so that $e = e^s_{g_s'}$.
Then the identification \eqref{eqEd_1(e_{g_s'}^s)} implies
$$
h(e)\ =\ \sum_{1\leq i \leq s} \sum_{\alpha \in T_i} f(g\alpha) \nonumber \\
\ =\ \sum_{1\leq i \leq s}f(g_i') \quad\quad \quad  \text{(from \eqref{eq-proofLem1.6} $\Leftrightarrow$ \eqref{eq-prop-harm-1}).}
$$
Given a $2$-simplex $([L_0],[L_1],[L_2])$ in $\cB$, suppose the edges $([L_0],[L_1])$ and $([L_1],[L_2])$ have types $s_1$ and $s_2$, respectively.
Take $g \in \GL_r(K)$ so that
$$[L_0] = [\Lambda_0 (g_{s_1+s_2}')^{-1}], \quad 
[L_1] = [\Lambda_0 (g_{s_2}')^{-1}], \quad
[L_2] = [\Lambda_0 g^{-1}].$$
Observe that for $i,j \in \ZZ_{\geq 1}$ with $i+j \leq r$, we have $g_{i+j}' = (g_{i}')_{j}'$.
Thus $$([L_0],[L_1]) = e^{s_1}_{(g_{s_2}')_{s_1}'}, \quad ([L_1],[L_2]) = e^{s_2}_{g_{s_2}'}, \quad
([L_0],[L_2]) = e^{s_1+s_2}_{g_{s_1+s_2}'},
$$
which shows 
\begin{eqnarray}\label{eq: cycle-condition}
h([L_0],[L_1]) + h([L_1],[L_2])
&=& \sum_{1\leq i \leq s_1} f\big((g_{s_2}')'_i\big) + \sum_{1\leq j \leq s_2}f(g_{j}') \nonumber \\
&=& \sum_{1\leq i \leq s_1}f(g_{s_2+i}') + \sum_{1\leq j \leq s_2} f(g_j') \nonumber \\
&=& \sum_{1\leq i \leq s_1+s_2} f(g_i') \\ &=&f([L_0],[L_2]). \nonumber 
\end{eqnarray}
Hence condition (4) in Definition~\ref{defnHarm-1} holds.

Note that the equality \eqref{eq: cycle-condition} also holds when $[L_2] = [L_0]$. Thus for $e = ([L_0],[L_1]) \in \Ed$, we have
$$h(e) + h(\overline{e}) = \sum_{1\leq i \leq r} f(g_i') = 0, \quad \text{(property \eqref{eq-prop-harm-2})}
$$
and so $h$ satisfies condition (1).

Finally, given $v \in \Ver(\cB)$, take $g \in \GL_r(K)$ so that $v = [\Lambda_0 g^{-1}]$.
Then
\begin{eqnarray}
\sum_{e \in \Ed_1^t(v)}h(e)
\ =\ \sum_{1\leq i \leq r} \sum_{\alpha \in T_i} f(g\alpha) & = & \sum_{1\leq i \leq r} f(g'_i) \quad \text{(equality \eqref{eq-proofLem1.6})} \nonumber \\
&=& 0 \quad \text{(property \eqref{eq-prop-harm-2}).} \nonumber
\end{eqnarray}
From Remark~\ref{remDefdS} (iii), condition (2) in Definition~\ref{defnHarm-1} holds, and the proof is complete.
\end{proof}

Therefore, we propose the following:

\begin{defn} A \textit{harmonic $1$-cochain} on $\GL_r(K)$ with values in $R$ is a 
	function $\GL_r(K)\to R$ which is right invariant 
	under the action of $K^\times\cI^1$, and satisfies the {\it harmonic} properties \eqref{eq-prop-harm-1} and \eqref{eq-prop-harm-2}.
\end{defn}

    When $r=2$, the above definition of harmonic cochains on $\GL_2(K)$ agrees with 
    Gekeler's definition in  \cite[(2.12)]{GekelerImproper}.


\section{Fourier expansion of harmonic $1$-cochains}\label{sFEHC}
In this section the local field $K$ is $\Fi$ and $\pi=1/T$. 
Although we are primarily interested in harmonic $1$-cochain, the discussion in this section 
applies to a larger class of functions. First, we record the following consequence of the Iwasawa decomposition:

\begin{lem}\label{lemIwasawaDecom}
	Let 
	\begin{equation}\label{eqParabolic1n-1}
	P(\Fi) =\left\{\begin{pmatrix} 1 & b \\ & d\end{pmatrix}\in \GL_r(\Fi)\mid d\in \GL_{r-1}(\Fi)\right\}. 
	\end{equation}
	Then 
	\begin{equation}\label{eq-Iwasawa}
	\GL_r(\Fi) =  P(\Fi) \Fi^\times \cI^1 \coprod P(\Fi) \begin{pmatrix}  
	& I_{r-1} \\ 1 & \end{pmatrix}\Fi^\times\cI^1. 
	\end{equation}
\end{lem}
\begin{proof} It is easy to check that the set of matrices 
	$$
	\left\{\begin{pmatrix} (\vec{u}_s | 0) & I_{r-s} \\ I_s & 0\end{pmatrix}\ \bigg|\ 1\leq s\leq r, \ \vec{u_s}\in \Mat_{(r-s)\times 1}(\F_q)\right\}
	$$
	is a set of representatives of the left cosets of $\cI^1$ in $\GL_r(O_\infty)$. Now 
	$$\begin{pmatrix} (\vec{u}_s | 0) & I_{r-s} \\ I_s & 0\end{pmatrix} = \begin{pmatrix}  I_{r-s} & (\vec{u}_s | 0)  \\ 0 & I_s\end{pmatrix}
	\begin{pmatrix} 0 & I_{r-s} \\ I_s & 0\end{pmatrix},
	$$
	and $\begin{pmatrix}  I_{r-s} & (\vec{u}_s | 0)  \\ 0 & I_s\end{pmatrix}\in P(\Fi)$. Hence, from the 
	Iwasawa decomposition $\GL_r(\Fi)=P(\Fi)\Fi^\times \cI^1$ (cf. \cite[Prop. 4.5.2]{Bump}), we obtain 
	$$
	\GL_r(\Fi)=\bigcup_{1\leq s\leq r} P(\Fi) \begin{pmatrix} 0 & I_{r-s} \\ I_s & 0\end{pmatrix} \Fi^\times \cI^1. 
	$$
	Next, observe that for a given $1\leq s\leq r-1$ we have  
	$$
	\begin{pmatrix} 1 & & \\ & & I_{s-1} \\ & I_{r-s} &\end{pmatrix} \begin{pmatrix}  & I_{r-s} \\ I_s & \end{pmatrix} 
	= \begin{pmatrix}  & I_{r-1} \\ 1 & \end{pmatrix} \begin{pmatrix} 1 & & & \\ & & 1 & \\ & I_{s-1} &  & \\ & &  & I_{r-s-1} \end{pmatrix}. 
	$$
	Hence $P(\Fi) \begin{pmatrix} 0 & I_{r-s} \\ I_s & 0\end{pmatrix} \Fi^\times \cI^1 = P(\Fi) \begin{pmatrix} 0 & I_{r-1} \\ 1 & 0\end{pmatrix} \Fi^\times \cI^1$ 
	for all $1\leq s\leq r-1$. It remains to show that $\begin{pmatrix} 0 & I_{r-1} \\ 1 & 0\end{pmatrix} \not\in P(\Fi)  \Fi^\times \cI^1$. 
	If on the contrary $\begin{pmatrix} 0 & I_{r-1} \\ 1 & 0\end{pmatrix} = g \kappa$ with $g\in P(\Fi)  \Fi^\times$ and $\kappa \in \cI^1$, then 
	$g\in P(\Fi)  \Fi^\times \cap \GL_r(O_\infty)\subset \cI^1$. But this implies that $\begin{pmatrix} 0 & I_{r-1} \\ 1 & 0\end{pmatrix}\in \cI^1$, 
	which is obviously false. 
\end{proof}

Let $h$ be a $\CC$-valued function on $\GL_r(\Fi)$. 
Assume $h(g \alpha \kappa)=h(g)$ for all $\alpha\in \Fi^\times$, $\kappa\in \cI^1$. 
Then $h$ can be identified with a $\CC$-valued function on set of left cosets 
$\GL_r(\Fi)/\Fi^\times \cI^1$. It follows from the previous lemma that $h$ is uniquely determined 
by its values on $P(\Fi)$ and $P(\Fi)\begin{pmatrix} 0 & I_{r-1} \\ 1 & 0\end{pmatrix}$. 
In this section we develop a theory of Fourier expansions of functions on $P(F_\infty)$. 

\begin{rem}\label{remFC}
	In later sections we will be primarily interested in harmonic $1$-cochains on $\GL_r(\Fi)$ which arise from the 
	Drinfeld discriminant function. These harmonic $1$-cochains will be left invariant under $\G_0^r(\fn)$ 
	for some nonzero ideal $\fn\lhd A$. In regard to this, we point out that
	\begin{equation}\label{eq-surj}
	\text{ the canonical map } P(F_\infty) \longrightarrow \Gamma_0^r(\nfk) \backslash \GL_r(F_\infty)/F_\infty^\times \Ical^1 \text{ is surjective.}
	\end{equation}
	(This is easy to see using Lemma \ref{lemIwasawaDecom}.) 
	Hence a $\CC$-valued function on $\Gamma_0^r(\nfk) \backslash \GL_r(F_\infty)/F_\infty^\times \Ical^1$ is uniquely determined by its values on $P(F_\infty)$, and so by its Fourier coefficients. 
\end{rem}

Let $h$ be a $\CC$-valued function on $P(F_\infty)$. If $h$ is the restriction of a function on $\GL_r(\Fi)$,  
right invariant under $\Fi^\times \cI^1$, then $h$ is right invariant under $P(O_\infty)=P(\Fi)\cap \cI^1$. 
This assumption appears in most results of this section, and it is a natural assumption since we are primarily interested in 
harmonic $1$-cochains. 

Write a matrix $g\in P(\Fi)$ as
$$
g=\begin{pmatrix} 1 & \vec{x} y \\ & y\end{pmatrix}
$$ 
for a uniquely determined $y\in \GL_{r-1}(\Fi)$ and a row vector $\vec{x}\in \Fi^{r-1}$.
To simplify the notation, we denote 
$$
h(\vec{x}, y)=h\left(\begin{pmatrix} 1 & \vec{x} y \\ & y\end{pmatrix}\right).
$$

Put
$$
\G_\infty=\begin{pmatrix} \GL_1(A) & \ast \\ & \GL_{r-1}(A)\end{pmatrix}\subset \GL_r(A),
$$
and 
$$
\G_\infty'=\left\{\begin{pmatrix} 1 & \vec{a} \\ & I_{r-1}\end{pmatrix}\ \bigg|\ \vec{a}\in A^{r-1} \right\}\subset \G_\infty. 
$$
Assume $$h(\gamma g)=h(g)\quad \text{for all } \gamma\in \G_\infty'.$$ Then for every $\vec{a}\in A^{r-1}$ we have 
\begin{align}\label{eqA^rinvar}
h(\vec{a}+\vec{x}, y) & =h\left(\begin{pmatrix} 1 & (\vec{a}+\vec{x}) y \\ & y\end{pmatrix}\right)\\ 
\nonumber & = h\left(\begin{pmatrix} 1 & \vec{a} \\ & I_{r-1}\end{pmatrix}\begin{pmatrix} 1 & \vec{x} y \\ & y\end{pmatrix}\right)\\ 
\nonumber &=h(\vec{x}, y). 
\end{align}
Fix a nontrivial additive character $\psi_0:\F_p\to \CC^\times$, and let 
\begin{align*}
\psi: \Fi &\To \CC^\times, \\ \sum a_i\pi^i & \longmapsto \psi_0(\Tr_{\F_q/\F_p}(a_1)). 
\end{align*}
The conductor of $\psi$ is $\pi^2O_\infty$ and $\psi(a) = 1$ for every $a \in A$.
Because of \eqref{eqA^rinvar}, $h$ has a Fourier expansion (cf.\ \cite[pp. 19-20]{Weil}, \cite[$\S$2]{Pal}) 
\begin{equation}\label{eqFEh}
h(\vec{x}, y) =\sum_{\vec{a}\in A^{r-1}} h^*(\vec{a}, y)\psi(\vec{a} \bcdot \vec{x}),
\end{equation}
where 
$\vec{a}\bcdot \vec{x}:=\vec{a}\vec{x}^t$ and 
\begin{equation}\label{eqFCh}
h^*(\vec{a}, y)=\int_{A^{r-1}\bs \Fi^{r-1}} h(\vec{u}, y)\psi(-\vec{a}\bcdot \vec{u})\dd \vec{u}
\end{equation}
with the self-dual Haar measure $\dd\vec{u}$ on $\Fi^{r-1}$, i.e., $\Vol((\pi O_\infty)^{r-1}, \dd\vec{u})=1$. 

\begin{lem}\label{lemPropFC}
	Let $h$ be a $\CC$-valued function on $P(F_\infty)/P(O_\infty)$ which is left invariant by $\Gamma_\infty'$.
	Given $\vec{a} \in A^{r-1}$ and $y \in \GL_{r-1}(F_\infty)$, we have:
	\begin{enumerate}
		\item $h^\ast(\vec{a}, y\kappa)=h^*(\vec{a}, y)$ for all $\kappa\in \GL_{r-1}(O_\infty)$. 
		\item $h^\ast(\vec{a}, y)=h^\ast(\vec{a}, y)\psi(\vec{a}\bcdot (\vec{w}y^{-1}))$ for all $\vec{w}\in O_\infty^{r-1}$. 
		\item $h^\ast(\vec{a}, y)=0$ unless $\vec{a}(y^t)^{-1}\in (\pi^2O_\infty)^{r-1}$. 
		Consequently, for $\vec{x} \in \Fi^{r-1}$, we get
		\begin{eqnarray}\label{eq-FE}
		h(\vec{x},y) &=& \sum_{\substack{\vec{a} \in A^{r-1}\\ \vec{a}(y^t)^{-1} \in (\pi^2O_\infty)^{r-1}}} h^*(\vec{a},y)\psi(\vec{a}\bcdot \vec{x}).
		\end{eqnarray}
			\end{enumerate}
\end{lem}
\begin{proof}
	(1) This follows from 
	$$
	h(\vec{a}, y\kappa)=h\left(\begin{pmatrix} 1 & \vec{a} y\kappa \\ & y\kappa\end{pmatrix}\right)=
	h\left(\begin{pmatrix} 1 & \vec{a} y \\ & y\end{pmatrix}\begin{pmatrix} 1 &  \\ & \kappa\end{pmatrix}\right)=
	h\left(\begin{pmatrix} 1 & \vec{a} y \\ & y\end{pmatrix}\right)=h(\vec{a}, y). 
	$$
	
	(2) Given $\vec{w} \in O_\infty^{r-1}$, we have $\begin{pmatrix} 1 & \vec{w}\\ & I_{r-1}\end{pmatrix}\in P(O_\infty)$.
	Thus
	\begin{align*}
	h^\ast(\vec{a}, y)\psi(\vec{a}\bcdot (\vec{w}y^{-1}))
	& = \int_{A^{r-1}\bs \Fi^{r-1}} h(\vec{u}, y)\psi(\vec{a}\bcdot (\vec{w}y^{-1}-\vec{u}))\dd \vec{u} \\
	&= \int_{A^{r-1}\bs \Fi^{r-1}} h(\vec{u}+\vec{w}y^{-1}, y)\psi(-\vec{a}\bcdot \vec{u})\dd \vec{u} \\
	&=\int_{A^{r-1}\bs \Fi^{r-1}} h\left(\begin{pmatrix} 1 & \vec{u}y\\ & y\end{pmatrix}\begin{pmatrix} 1 & \vec{w}\\ & 1\end{pmatrix}\right)\psi(-\vec{a}\bcdot \vec{u})\dd \vec{u}\\ 
	&=\int_{A^{r-1}\bs \Fi^{r-1}} h(\vec{u}, y)\psi(-\vec{a}\bcdot \vec{u})\dd \vec{u} \\
	&= h^*(\vec{a},y). 
	\end{align*}
	
	(3) The dot-product $\vec{a}\bcdot (\vec{w}y^{-1})$ is equal to $(\vec{a}(y^t)^{-1})\bcdot \vec{w}$. 
	Let $\vec{a}(y^t)^{-1}=(\alpha_1, \dots, \alpha_{r-1})$. If $\alpha_i\not\in \pi^2O_\infty$ for some $1\leq i\leq r-1$, 
	then we can find $w_i\in O_\infty$ such that $\psi(\alpha_iw_i)\neq 1$. Taking $\vec{w}=(0, \dots, w_i, 0,\dots 0)$, we 
	get $\psi(\vec{a}\bcdot (\vec{w}y^{-1}))\neq 1$. By (2), this implies $h^\ast(\vec{a}, y)=0$.  
\end{proof}

\begin{lem}\label{lemPropFC2} 
Let $h$ be a $\CC$-valued function on $P(F_\infty)/P(O_\infty)$ which is left invariant under $\G_\infty$.
For all $\gamma\in \GL_{r-1}(A)$ we have 
$$
h^\ast(\vec{a}\gamma^t, \gamma y)=h^*(\vec{a}, y). 
$$ 
\end{lem}
\begin{proof}
	\begin{align*}
	h^\ast(\vec{a}\gamma^t, \gamma y) & = \int_{A^{r-1}\bs \Fi^{r-1}} h(\vec{u}, \gamma y)\psi(-(\vec{a}\gamma^t)\bcdot \vec{u})\dd \vec{u} \\
	& = \int_{A^{r-1}\bs \Fi^{r-1}} h\left(\begin{pmatrix} 1 & \\ & \gamma\end{pmatrix} 
	\begin{pmatrix} 1 & (\vec{u}\gamma) y \\ & y\end{pmatrix}\right)
	\psi(-\vec{a}\bcdot (\vec{u}\gamma ))\dd \vec{u} \\
	& = \int_{A^{r-1}\bs \Fi^{r-1}} h\left(
	\begin{pmatrix} 1 & (\vec{u}\gamma) y \\ & y\end{pmatrix}\right)
	\psi(-\vec{a}\bcdot (\vec{u}\gamma ))\dd \vec{u} \\
	&=	h^\ast(\vec{a}, y), 
	\end{align*}
	where the last equality follows from the change of variables $\vec{u}\mapsto \vec{u}\gamma^{-1}$ in the integral. 
\end{proof}

\begin{rem}\label{rem-FC}
Let $h$ be a $\CC$-valued function on $P(F_\infty)/P(O_\infty)$ which is left invariant under $\G_\infty$.
\begin{itemize}
\item[(i)] It is a fact that every element of the double coset space 
$$
\GL_{r-1}(A)\bs \GL_{r-1}(\Fi)/\GL_{r-1}(O_\infty)
$$
is represented by a diagonal matrix
\begin{equation}\label{eq-yDiagForm}
\diag(T^{n_1}, \dots, T^{n_{r-1}}), \quad n_1, \dots, n_{r-1}\in \Z.
\end{equation}
	Note that Lemma \ref{lemPropFC} (3) implies that the Fourier expansion \eqref{eqFEh} of $h$ is in fact a finite sum. More 
	precisely, using Lemma \ref{lemPropFC2}, we can assume without loss of generality that $y=\diag(T^{n_1}, \dots, T^{n_{r-1}})$. Then 
	\begin{equation}\label{eqFrExpansion}
	h(\vec{x}, y)=\sum_{\substack{\vec{a}=(a_1, \dots, a_{r-1})\in A^{r-1}\\ \deg(a_i)\leq n_i-2}} h^\ast(\vec{a}, y)\psi(\vec{a}\bcdot \vec{x}). 
	\end{equation}
\item[(ii)] Given $y \in \GL_{r-1}(\Fi)$ and $\vec{a} \in A^{r-1}$,
we know from Lemma~\ref{lemPropFC2} (3) that $h^*(\vec{a},y) = 0$ if $\vec{a}(y^t)^{-1} \notin (\pi^2O_\infty)^{r-1}$.
When $\vec{a} (y^t)^{-1} \in (\pi^2O_\infty)^{r-1}$,
the integral \eqref{eqFCh} representing the Fourier coefficient $h^\ast(\vec{a}, y)$ is also a finite sum. To see this, by (i) of this remark and Lemma~\ref{lemPropFC2} we may assume $y = \diag(T^{n_1},...,T^{n_{r-1}})$.
Let  
$$
m=\max\{
n_1, \dots, n_{r-1}\}.
$$ 
For every $\vec{w}\in (\pi^m O_\infty)^{r-1}$, one has $\vec{w}y\in O_\infty^{r-1}$ and $\vec{a}\bcdot \vec{w}\in \pi^2 O_\infty$.
Thus,
$$
h(\vec{u}+\vec{w}, y)=h\left(\begin{pmatrix} 1 & (\vec{u}+\vec{w}) y \\ & y\end{pmatrix}\right)
= h\left(\begin{pmatrix} 1 & \vec{u} y \\ & y\end{pmatrix}\begin{pmatrix} 1 & \vec{w}y \\ & I_{r-1}\end{pmatrix}\right)=h(\vec{u}, y) 
$$ 
(due to the right $P(O_\infty)$-invariance of $h$) and
$\psi(-\vec{a}\bcdot (\vec{u}+\vec{w}))=\psi(-\vec{a}\bcdot \vec{u})$. 
Now note that 
$F_\infty=A+\pi O_\infty$.
Using the invariance of the integrand under the action of $(\pi^m O_\infty)^{r-1}$, we get 
\begin{align}\label{eqhFCsum}
h^\ast(\vec{a}, y) & =\int_{(\pi O_\infty)^{r-1}} h(\vec{u}, y)\psi(-\vec{a}\bcdot \vec{u})\dd \vec{u}\\ \nonumber 
&= q^{(1-m)(r-1)}\sum_{\vec{u}\in (\pi O_\infty/\pi^m O_\infty)^{r-1}} h(\vec{u}, y)\psi(-\vec{a}\bcdot \vec{u}), 
\end{align}
where $q^{(1-m)(r-1)}=\Vol((\pi^m O_\infty)^{r-1},\dd \vec{u})$.
\end{itemize}
\end{rem}

Let $h$ be a $\CC$-valued function on $\Gamma_\infty \backslash P(F_\infty)/P(O_\infty)$ satisfying 
\begin{equation}\label{eq-(1.5)-var}
\sum_{\vec{u} \in \FF_q^{s-1}}
h\left(g \begin{pmatrix} 
1 & T \cdot \vec{u} & \\ & T I_{s-1} & \\ & & T I_{r-s}
\end{pmatrix}\right)
= h\left(g \begin{pmatrix} I_s & \\ & TI_{r-s}\end{pmatrix}\right)\quad \forall g\in P(F_\infty), 1\leq s \leq r.
\end{equation}
We remark that if $h$ is the restriction of a harmonic $1$-cochain to $P(F_\infty)$, then \eqref{eq-(1.5)-var} coincides with \eqref{eq-prop-harm-1}.
For $1\leq i\leq r-1$ and $\alpha \in \Fi^\times$, let $$d_i(\alpha) := \diag(1,...,1,\alpha,1,...,1) \in \GL_r(F_\infty),$$ 
where $\alpha$ lies on the $i$-th diagonal entry.
We now translate property~\eqref{eq-(1.5)-var} into a condition on the Fourier coefficients:

\begin{lem}\label{lem-hp-FC}
    Let $h$ be a $\CC$-valued function on $\Gamma_\infty \backslash P(F_\infty)/P(O_\infty)$.
    Then $h$ has property \eqref{eq-(1.5)-var} if and only if for all $y \in \GL_{r-1}(\Fi)$ and $\vec{a} \in A^{r-1}$ with $\vec{a}(y^t)^{-1} \in (\pi^2O_\infty)^{r-1}$,
    we have
    \begin{eqnarray}\label{eq-prop-harm-1-FC}
        h^*\big(\vec{a},yd_i(T)\big) = q^{-1}h^*(\vec{a},y), \quad \text{for all } 1\leq i \leq r-1.
    \end{eqnarray}
\end{lem}

\begin{proof}
    For $1\leq s \leq r$, put $\tilde{d}_1(T) = I_{r-1}$ and
    $$
    \tilde{d}_s(T):= \prod_{1\leq i \leq s-1}d_i(T) = \begin{pmatrix}
    T\cdot I_{s-1} & \\ & I_{r-s}\end{pmatrix}.
    $$
    Then \eqref{eq-prop-harm-1-FC} is equivalent to 
    $$
    q^{s-1} h^*\big(\vec{a},y\tilde{d}_s(T)\big) \ = \ h^*(\vec{a},y), \quad \text{for all } 1\leq s \leq r.
    $$
    
    Let
    $$y_s := y \cdot \begin{pmatrix} I_{s-1} & \\ & T^{-1} I_{r-s}\end{pmatrix}
    = (y \tilde{d}_s(T)) \cdot T^{-1}.
    $$
    For $\vec{x} \in F_\infty^{r-1}$, put $g = \begin{pmatrix}1 & \vec{x}y_s \\ & y_s \end{pmatrix}$ into both sides of \eqref{eq-(1.5)-var}.
    We translate property \eqref{eq-(1.5)-var} into
    \begin{equation}\label{eq-harmonic-1var}
    \sum_{\vec{u} \in \FF_q^{s-1}} h \big(\vec{x}+(\vec{u},0)y_s^{-1}, Ty_s\big)
    = h(\vec{x},y).
    \end{equation}
    From the Fourier expansion of $h$, the left hand side of \eqref{eq-(1.5)-var} equals to 
    \begin{align*}
    & \sum_{\vec{u} \in \FF_q^{s-1}} h \big(\vec{x}+(\vec{u},0)y_s^{-1}, Ty_s\big) \nonumber \\
    &=
    \sum_{\substack{\vec{a} \in A^{r-1}\\ \vec{a}(Ty_s^t)^{-1} \in (\pi^2O_\infty)^{r-1}}}
    h^*(\vec{a},Ty_s)\psi(\vec{a}\bcdot\vec{x})\cdot \left(\sum_{\vec{u} \in \FF_q^{s-1}}\psi\big(\vec{a}(y_s^t)^{-1} \bcdot (\vec{u},0)\big)\right) \nonumber \\
    &=
    \sum_{\substack{\vec{a} \in A^{r-1}\\ \vec{a}(y^t)^{-1} \in (\pi^2O_\infty)^{r-1}}}
    \bigg(q^{s-1}h^*\big(\vec{a},y\tilde{d}_s(T)\big)\bigg)\psi(\vec{a}\bcdot \vec{x}),
    \nonumber
    \end{align*}
    where the last equality follows from the fact that $y = y_s \cdot \begin{pmatrix} I_{s-1}& \\ & T I_{r-s}\end{pmatrix}$ and for $\vec{a} \in A^{r-1}$ with $\vec{a}(Ty_s^t)^{-1} \in (\pi^2O_\infty)^{r-1}$ we have
    $$
    \sum_{\vec{u} \in \FF_q^{s-1}} \psi\big(\vec{a}(y_s^t)^{-1}\bcdot (\vec{u},0)\big)
    = \begin{cases} q^{s-1} & \text{ if $\vec{a}(y^t)^{-1} \in (\pi^2O_\infty)^{r-1}$,} \\ 
    0 & \text{ otherwise.}
    \end{cases}
    $$
    Meanwhile, the right hand side of \eqref{eq-harmonic-1var} becomes 
    \begin{equation}
    h(\vec{x},y)
    = \sum_{\substack{\vec{a} \in A^{r-1}\\ \vec{a}(y^t)^{-1} \in (\pi^2O_\infty)^{r-1}}}
    h^*(\vec{a},y)\psi(\vec{a}\bcdot\vec{x}).
    \end{equation}
    Therefore the uniqueness of the Fourier expansion of $h$ implies that property \eqref{eq-(1.5)-var} is equivalent to 
    $$
    q^{s-1} h^*\big(\vec{a},y\tilde{d}_s(T)\big) \ = \ h^*(\vec{a},y), \quad \text{for all }1\leq s \leq r.
    $$
    This completes the proof.
\end{proof}

\begin{notn}\label{notn-c}
	Let $\vec{a}=(a_1, \dots, a_{r-1})\in A^{r-1}$. Put  
	$$
	y_{\vec{a}}=\diag(T^{\alpha_1}, \dots, T^{\alpha_{r-1}}), 
	$$
	where 
	$$
	\alpha_i=\max(2, \deg(a_i)+2), \quad 1\leq i\leq r-1,
	$$
Denote 
$$
c_{\vec{a}}(h)\ =\ |\det y_{\vec{a}}| \cdot  q^{2(1-r)}  \cdot h^*(\vec{a},y_{\vec{a}}).
$$
\end{notn}

Then Lemma~\ref{lem-hp-FC} ensures:

\begin{lem}\label{lem-c-invariant}
    Let $h$ be a $\CC$-valued function on $\Gamma_\infty\backslash P(\Fi)/P(O_\infty)$ which satisfies \eqref{eq-(1.5)-var}.
    Then for every $\vec{a} \in A^{r-1}$ and $\gamma \in \GL_{r-1}(A)$, we have
    $$
    c_{\vec{a}\gamma}(h) = c_{\vec{a}}(h).$$
\end{lem}

\begin{proof}
    The claim is clear when $\gamma \in \GL_{r-1}(A)$ is a permutation matrix or $\gamma = \diag(\epsilon_1,...,\epsilon_{r-1})$ with $\epsilon_1,...,\epsilon_{r-1} \in \FF_q^\times$. Given $1\leq i<j\leq r-1$ and $a_{ij} \in A$, suppose
    $$(\gamma^t)^{-1}
    = \begin{pmatrix} 1 & & & \\
    & \ddots & a_{ij} & \\
    & & \ddots &   \\
    & & & 1\end{pmatrix} \in \GL_{r-1}(A).$$
    Observe that for $y = \diag(T^{n_1},...,T^{n_{r-1}})$, one has
    $$(\gamma^t)^{-1} y
    = y \kappa_{ij} \quad \text{ where } \quad \kappa_{ij} = \begin{pmatrix} 1 & & & \\
    & \ddots & a_{ij}T^{n_j-n_i} & \\
    & & \ddots &   \\
    & & & 1\end{pmatrix}.
    $$
    Given $\vec{a} \in A^{r-1}$, we may take $n_1,...,n_{r-1}$ sufficiently large so that:
$$\vec{a}(y^t)^{-1} \in (\pi^2O_\infty)^{r-1},
\quad \vec{a}\gamma (y^t)^{-1}  \in (\pi^2O_\infty)^{r-1}, \quad \text{ and } \quad 
a_{ij}T^{n_j-n_i} \in O_\infty.$$
Then Lemma~\ref{lemPropFC2} and Lemma \ref{lemPropFC} (1) imply
$$h^*(\vec{a}\gamma,y) = h^*(\vec{a},(\gamma^t)^{-1} y)
=h^*(\vec{a},y \kappa_{ij}) = h^*(\vec{a},y).
$$
Note that $y$, $y_{\vec{a}}$, and $y_{\vec{a}\gamma}$ are all diagonal matrices. Therefore
\begin{align*}
q^{2(r-1)}\cdot c_{\vec{a}\gamma}(h)
& = |\det (y_{\vec{a}\gamma})| \cdot h^*(\vec{a}\gamma, y_{\vec{a}\gamma}) \\ 
&= |\det (y_{\vec{a}\gamma})|\cdot |\det (y_{\vec{a}\gamma}^{-1} y)| \cdot h^*(\vec{a}\gamma,y) \quad (\text{Lemma~\ref{lem-hp-FC}})\\
& = |\det(y)| \cdot h^*(\vec{a},y)\\ 
&=  |\det(y_{\vec{a}})| \cdot h^*(\vec{a},y_{\vec{a}})\quad (\text{Lemma~\ref{lem-hp-FC}})\\
& = q^{2(r-1)}c_{\vec{a}}(h). 
\end{align*}

Since every element in $\GL_{r-1}(A)$ is a product of the elementary matrices, this invariance property holds for every $\gamma \in \GL_{r-1}(A)$.
\end{proof}

We conclude that:

\begin{prop}\label{prop3.10}  A $\CC$-valued function $h$ on $\G_\infty\backslash P(\Fi)/P(O_\infty)$ 
satisfies \eqref{eq-(1.5)-var} if and only if for all 
	 $y \in \GL_{r-1}(F_\infty)$
	 and $\vec{a} \in A^{r-1}$ with $\vec{a}(y^t)^{-1} \in (\pi^2O_\infty)^{r-1}$
	 we have 
	$$
	h^\ast(\vec{a}, y)=|\det(y)|^{-1}\cdot q^{2(r-1)}\cdot c_{\vec{a}}(h).
	$$
	In this case, we may write the Fourier expansion of $h$ as
	\begin{equation}\label{eqFE_under1.5}
	h(\vec{x}, y)=|\det(y)|^{-1}\cdot q^{2(r-1)} \sum_{\substack{\vec{a}\in A^{r-1}\\ \vec{a}(y^t)^{-1} \in (\pi^2O_\infty)^{r-1}}}
	c_{\vec{a}}(h)\psi(\vec{a}\bcdot \vec{x}), \quad \forall \vec{x} \in \Fi^{r-1}. 
	\end{equation}
\end{prop}

\begin{proof} 
    Suppose $h$ satisfies \eqref{eq-(1.5)-var}.
    Given $y \in \GL_{r-1}(F_\infty)$ and $\vec{a} \in A^{r-1}$ with $\vec{a}(y^t)^{-1} \in (\pi^2O_\infty)^{r-1}$, by 
    Remark~\ref{rem-FC} (i), we may take $\gamma \in \GL_{r-1}(A)$ and $\kappa \in \GL_{r-1}(O_\infty)$ so that
    $$
    \gamma y \kappa = \diag(T^{n_1},...,T^{n_{r-1}}) =: y_0 \quad \text{ with $n_1,...,n_{r-1} \in \ZZ$.}
    $$
    Then  
    \begin{eqnarray}
    h^*(\vec{a},y) &=& h^*(\vec{a}\gamma^t,\gamma y \kappa)\ =\ h^*(\vec{a}\gamma^t,y_0) \quad \quad (\text{Lemma~\ref{lemPropFC} (1) and \ref{lemPropFC2}}) \nonumber \\
    &=& |\det(y_0^{-1}y_{\vec{a}\gamma^t})| \cdot h^*(\vec{a}\gamma^t,y_{\vec{a}\gamma^t}) \quad \quad \! (\text{Lemma~\ref{lem-hp-FC}}) \nonumber \\
    &=&|\det(y)|^{-1}\cdot q^{2(r-1)} \cdot c_{\vec{a}\gamma^t}(h) \quad \quad \  (\text{Notation}~\ref{notn-c})\nonumber \\
    &=& |\det(y)|^{-1} \cdot q^{2(r-1)} \cdot c_{\vec{a}}(h). \quad \quad \quad \! (\text{Lemma~\ref{lem-c-invariant}}) \nonumber
    \end{eqnarray}
    
    Conversely, suppose for $y \in \GL_{r-1}(F_\infty)$ and $\vec{a} \in A^{r-1}$ with $\vec{a}(y^t)^{-1} \in (\pi^2O_\infty)^{r-1}$ we have
    $$
	h^\ast(\vec{a}, y)=|\det(y)|^{-1}\cdot q^{2(r-1)}\cdot c_{\vec{a}}(h).
	$$
	Then \eqref{eq-prop-harm-1-FC} holds, i.e., for any $1\leq i \leq r-1$ we have 
	\begin{align*}
	h^*\big(\vec{a},yd_i(T)\big) &= |\det(yd_i(T))|^{-1} q^{2(r-1)} c_{\vec{a}}(h)\\ 
	& = q^{-1} |\det(y)|^{-1} q^{2(r-1)} c_{\vec{a}}(h)\\  
	&=\  q^{-1} h^*(\vec{a},y).
	\end{align*}
	Therefore $h$ satisfies \eqref{eq-(1.5)-var} by Lemma~\ref{lem-hp-FC}.
\end{proof}


\section{Drinfeld discriminant function}\label{sDDF}

\subsection{Drinfeld modules over $\Ci$}
Let $\Ci\langle x\rangle$ be the non-commutative ring consisting of $\F_q$-linear polynomials $\sum_{i=0}^n a_i x^{q^i}$, $n\geq 0$, with addition 
given by the usual addition of polynomials but multiplication given by substitution $f\ast g=f(g(x))$. 
Let $r\geq 1$ be a positive integer. 
A \textit{Drinfeld $A$-module of rank $r$ over $\Ci$} is an embedding $\phi: A\to \Ci\langle x\rangle$, $a\mapsto \phi_a(x)$, defined by 
\begin{equation}\label{eqDMDef}
\phi_T(x)=Tx+g_1x^q+\cdots+g_rx^{q^r}, \text{ for some }g_1, \dots, g_r\in K, g_r\neq 0. 
\end{equation}
Two Drinfeld modules $\phi, \psi$ are isomorphic over $\Ci$ if $\phi=c\psi c^{-1}$ for some $c\in \Ci^\times$.
One can show that regarding $g_1, \dots, g_r$ in \eqref{eqDMDef} as indeterminates of respective weights $q^i-1$, the 
open subscheme $Y^r$ given by $g_r\neq 0$ of the weighted projective space $\mathrm{Proj}(\Ci [g_1, \dots, g_r])$ is a coarse moduli scheme 
for Drinfeld modules of rank $r$ over $\Ci$; cf. \cite{GekelerDMHR1}. 

An $A$-\textit{lattice} of rank $r$ in $\C_\infty$ is an $A$-submodule $\La\subset \Ci$, which is free of rank $r$ as an $A$-module 
and is discrete in $\Ci$, i.e., intersects 
each ball in finitely many points. The \textit{exponential function}  of $\La$ is 
$$
\exp_\La(x)=x\prod_{0\neq \la\in \La}\left(1-\frac{x}{\la}\right).  
$$
The following is well-known (see e.g. \cite{Drinfeld}, \cite{GreenBook}, \cite{Goss}):
\begin{prop}\label{propSec1:La}\hfill
	\begin{enumerate}
		\item[(i)] $\exp_\La:\Ci\to \Ci$ is an entire, surjective, $\F_q$-linear function with kernel $\La$. It can be expanded into a power series 
		$
		\exp_\La(x)=\sum_{n\geq 0} \alpha_n(\La) x^{q^n}$.  
		
		\item[(ii)] $\exp_{c\La}(x)=c\cdot \exp_\La(c^{-1}x)$ for any $c\in \Ci^\times$. 
		\item[(iii)] If $\La\subset \La'$ are $A$-lattices of the same rank, then 
		$f(\exp_\La(x))=e_{\La'}(x)$, where
		$$
		f(x)=x\prod_{0\neq \la\in \La'/\La}\left(1-\frac{x}{\exp_\La(\la)}\right). 
		$$
	\end{enumerate}
\end{prop}
Properties (ii) and (iii) applied to $\La\subset T^{-1} \La$ imply that there is a Drinfeld module $\phi^\La$ of rank $r$ over $\Ci$ such that 
\begin{align}\label{eqSec1:FuncEq}
\exp_\La(T x) &=\phi^\La_T(\exp_\La(x)), \\ 
\nonumber \phi^\La_T&=Tx+g_1(\La)x^q+\cdots+g_r(\La)x^{q^r}, 
\end{align} 
for some $g_1(\La), \dots, g_r(\La)\in \Ci$ with $g_r(\La)\neq 0$. 
Note that property (ii) and  \eqref{eqSec1:FuncEq} imply $\phi^{c\La}_T(x)=c\phi^\La_T(c^{-1}x)$. Hence the Drinfeld modules 
$\phi^{c\La}$ and $\phi^\La$ are isomorphic, and 
\begin{equation}\label{eqFEg_i}
g_i(c\La)=c^{1-q^i}g_i(\La), \quad 1\leq i\leq r. 
\end{equation}
Drinfeld proved in \cite{Drinfeld} that the assignment $\La \rightsquigarrow \phi^\La$ gives a bijection between the similarity classes of lattices of rank $r$ and 
the isomorphism classes of Drinfeld modules of rank $r$ over $\Ci$, where  
two $A$-lattices $\La$ and $\La'$ are \textit{similar} if there exists $c\in \Ci^\times$ such that 
$\La'=c\La$. 

\subsection{Drinfeld symmetric space}

To classify the similarity classes of $A$-lattices of rank $r$ in $\Ci$, one proceeds as follows. Given a lattice $\La$, choose a basis 
$z_1, \dots, z_r\in \Ci$ of $\La$, i.e., $\La=Az_1+\cdots+Az_r$.  Since we are interested in $\La$ only up to scaling, 
we associate to $\La$ the point $\boz=(z_1:\dots: z_r)\in \p^{r-1}(\Ci)$. It is not hard to prove that 
the discretness of $\La$ is equivalent to $z_1, \dots, z_r$ being linearly independent over $\Fi$. Hence $\boz$ lies 
in the \textit{Drinfeld symmetric space} 
$$
\Omega^r=\left \{ (z_1:\dots: z_r)\in \p^{r-1}(\Ci) \mid \text{$z_1, \dots, z_r$ are linearly independent over $\Fi$}\right \}. 
$$ 
$\Omega^r$ has a natural structure of a rigid-analytic space over $\Fi$; see \cite{Drinfeld}.  
The group $\GL_r(\Fi)$ acts on $\p^{r-1}(\Ci)$ from the left as on column vectors, and this action preserves $\Omega^r$. 
Note that $\boz, \boz'\in \Omega^r$ span the same $A$-lattice (up to scaling) if and only if $\boz'=\gamma \boz$ for some $\gamma\in \GL_r(A)$. 
Thus, the set of orbits $\GL_r(A)\bs \Omega^r$ is in bijection with the set of similarity classes of rank-$r$ $A$-lattices in $\Ci$.  The quotient 
$\GL_r(A)\bs \Omega^r$ inherits a structure of a rigid-analytic space from $\Omega$, and, in fact, 
$\GL_r(A)\bs \Omega^r$ is the analytification of the affine algebraic variety $Y^r_{\Fi}$; see \cite{Drinfeld}. 

We normalize the projective coordinates of points $\boz=(z_1:\dots:z_r)\in \Omega^r\subset \p^{r-1}(\Ci)$ by assuming $z_r=1$, 
and write $(z_1, \dots, z_r)=(z_1, \dots, z_{r-1}, 1)$ for the corresponding point. This allows us to identify 
$\Omega^r$ with a subset of $\Ci^r$:
$$
\Omega^r=\left\{ (z_1, \dots, z_r) \in \Ci^r\mid z_1, \dots, z_r \text{ are $\Fi$-linnearly independent and }z_r=1\right\}. 
$$ 
After this normalization, the action of 
$g=\begin{pmatrix} & \ast & &\ast \\ c_1 & \cdots & c_{r-1} & d\end{pmatrix}\in \GL_r(\Fi)$ on $\Omega^r$ becomes 
$$
g\boz:=j(g, \boz)^{-1} (z_1, \dots, z_r) g^t, 
$$
where 
$$
j(g, \boz)=c_1z_1+\cdots+ c_{r-1}z_{r-1}+d
$$
is the last entry of the row vector $(z_1, \dots, z_r) g^t$.
Note that a scalar matrix in $\GL_r(\Fi)$
acts as the identity on $\Omega^r$. Also, note that 
for $g_1, g_2\in \GL_r(\Fi)$ and $\boz\in \Omega^r$, we have 
$$
j(g_1g_2, \boz)=j(g_1, g_2\boz)\cdot j(g_2, \boz). 
$$

\subsection{Drinfeld discriminant function}

Let 
$$
\La_{\boz}=A z_1+\cdots+Az_{r-1}+A
$$
be the lattice associated to $\boz=(z_1, \dots, z_{r-1}, 1)\in \Omega^r$. 
We denote $\phi^{\boz}_T=\phi^{\La_{\boz}}_T$ and 
\begin{equation}\label{eqDrCoeffForms}
\phi^{\boz}_T=Tx+g_1(\boz)x^q+\cdots+g_r(\boz)x^{q^r}. 
\end{equation}
The coefficients of $\phi^{\boz}_T$ can be considered as $\C_\infty$-valued functions on $\Omega^r$. 
The \textit{Drinfeld discriminant function} is 
$$\Delta_r(\boz):=g_r(\boz).$$ 
\begin{prop}\label{propDeltaFE} \hfill
	\begin{enumerate} 
\item	$\Delta_r$ is holomorphic and non-vanishing on $\Omega^r$. 
\item 
$\Delta_r(\gamma\boz)=j(\gamma, \boz)^{q^r-1}\Delta_r(\boz)$ for all $\gamma\in \GL_r(A)$. 
\end{enumerate}
\end{prop}
\begin{proof} For a detailed discussion of rigid-analytic holomorphicity of $\Delta_r(\boz)$ we refer to \cite{BBP3}. 
	The function $\Delta_r(\boz)$ is non-vanishing since $g_r(\La)\neq 0$ for any $A$-lattice of rank $r$. 
	Finally, note that $\La_{\gamma \boz}=j(\gamma, \boz)^{-1}\gamma\La_{\boz}=j(\gamma, \boz)^{-1}\La_{\boz}$, 
	so the equality $\Delta_r(\gamma\boz)=j(\gamma, \boz)^{q^r-1}\Delta_r(\boz)$ follows from \eqref{eqFEg_i}. 
\end{proof}

As $\Delta_r$ is non-vanishing on $\Omega^r$, we may apply the so-called \textit{Gekeler--van der Put map} (introduced next) to get the corresponding harmonic 1-cochain on $\GL_r(F_\infty)$.

\subsection{Gekeler--van der Put map} 

Let $\cB^r(\R)$ denote the realization of the simplicial complex $\cB^r$. By a theorem of Goldman and Iwahori \cite{GI}, 
$\cB^r(\R)$ may be canonically identified with the set of similarity classes of non-archimedean norms $\nu:\Fi^r\to \R$. 
The similarity class of norms associated to a vertex $[L]\in \Ver(\cB^r)$ is defined through 
$$
\nu_L(\vec{x})=\min\big\{|\alpha|\ \big|\ \alpha\in \Fi,\ \vec{x}\in \alpha\cdot L \big\}. 
$$
Thus, $\nu_L(\vec{x})\leq 1$ if and only if $\vec{x}\in L$. The group $\GL_r(\Fi)$ acts from the left on the set of norm via 
$$
g\nu(\vec{x}):=\nu(\vec{x}g)
$$
for $\vec{x}\in \Fi^r$, a norm $\nu$, and $g\in \GL_r(\Fi)$. Note that 
\begin{equation}\label{eqNormGLequivariance}
g\nu_L=\nu_{Lg^{-1}}. 
\end{equation}
Each point 
$\boz=(z_1, \dots, z_r=1)\in \Omega^r$ determines a norm $\nu_{\boz}$ on $\Fi^r$ 
through 
$$
\nu_{\boz}(x_1, \dots, x_r):=\left|x_1z_1+\cdots+x_rz_r\right|, 
$$
which induces the \textit{building map} 
\begin{align*}
\la: \Omega^r &\To \cB^r(\R).\\ 
\boz &\longmapsto \nu_{\boz}
\end{align*}
The building map is $\GL_r(\Fi)$-equivariant: 
put $\vec{z} = (z_1,...,z_r) \in \CC_\infty^r$ to be the vector corrsponding to $\boz$, then
$$
\nu_{g \boz}(\vec{x})=|\vec{x}\bcdot (j(g, \boz)^{-1}\vec{z}g^t)|=|j(g, \boz)|^{-1} |\vec{x}\bcdot (\vec{z}g^t)|\sim |(\vec{x}g)\bcdot \vec{z})| 
=\nu_{\boz}(\vec{x}g)=(g \nu_{\boz})(\vec{x}). 
$$

Let $\cO(\Omega^r)^\times$ be the multiplicative group of 
holomorphic invertible functions on $\Omega^r$, i.e., holomorphic functions with no zeros. In \cite[$\S$2.7]{GekelerDMHR1}, 
Gekeler observed that for $f\in \cO(\Omega^r)^\times$ and $v\in \Ver(\cB^r)$, 
$$
|f(v)|:=|f(\boz)|, \quad \boz\in\la^{-1}(v), 
$$
does not depend on the choice of $\boz$. Thus, for an oriented edge $e=(v, w)$, we can define 
$$
\cP(f)(e)=\log_q \frac{|f(w)|}{|f(v)|}. 
$$
In \cite{GekelerDMHR1},  \cite{GekelerDMHR2}, and \cite{GekelerExacSeq}, Gekeler proved the following fundamental result, which generalizes to $r\geq 2$ 
an earlier result of van der Put \cite{vdPut} for $r=2$:
\begin{thm}\label{thm-GvdP} \hfill
\begin{itemize} 
	\item[(1)] $\cP(f)\in \Har^1(\cB^r, \Z)$. 
	\item[(2)] The sequence 
	\begin{equation}\label{eqGvdPseq}
	0\To \Ci^\times\To \cO(\Omega^r)^\times\overset{\cP}{\To} \Har^1(\cB^r, \Z)\To 0
	\end{equation}
	is short-exact and $\GL_r(\Fi)$-equivariant. 
\end{itemize}
\end{thm}

By Proposition \ref{propDeltaFE}, $\Delta_r\in \cO(\Omega^r)^\times$. 
Set $\cP_1(\Delta_r):=\cP(\Delta_r)\big|_{\vec{E}_1(\cB^r)}$, which we consider as a function on $\GL_r(\Fi)$.
More precisely,
fix a generator $\eps_r\in \F_{q^r}$ over $\F_q$ and put 
\begin{equation}\label{eqboz0}
\boz_0=(\eps_r^{r-1}, \eps_{r}^{r-2}, \dots, \eps_r, 1)\in \Omega^r.
\end{equation}
Notice that $\nu_{\boz_0} = \nu_{O_\infty^r}$, i.e.\ $\boz_0 \in \lambda^{-1}([O_\infty^r])$.
Put $S:=\diag(T,1,...,1) \in \GL_r(F_\infty)$.
Then for $g \in \GL_r(F_\infty)$, we get
$$
\Pcal_1(\Delta_r)(g) = \log_q |\Delta_r(gS\boz_0)| - \log_q|\Delta_r(g\boz_0)|.
$$

\begin{prop}
    Given $\gamma \in \Gamma_\infty$ and $g \in \GL_r(F_\infty)$, we have
    $$ \Pcal_1(\Delta_r)(\gamma g) = \Pcal_1(\Delta_r)(g).$$
\end{prop}

\begin{proof}
   
   For $\gamma \in \Gamma_\infty$, it is straightforward that
   $$j(\gamma,z) =j(\gamma,Sz), \quad \forall z \in \Omega^r.$$
   Let $\boz_g := g\boz_0 \in \Omega^r$. We have
   \begin{align*}
   \Pcal_1(\Delta_r)(\gamma g)
   &= \log_q|\Delta_r(\gamma \boz_g)| - \log_q|\Delta_r(\gamma S\boz_g)|  \\
   &= \log_q\big|j(\gamma,\boz_g)^{q^r-1}\Delta_r(\boz_g)\big| - \log_q\big|j(\gamma,S\boz_g)^{q^r-1}\Delta_r(S\boz_g)\big|  \\
   &= \log_q|\Delta_r(\boz_g)| - \log_q|\Delta_r(S\boz_g)|  \\
   &= \Pcal_1(\Delta_r)(g). 
   \end{align*}
\end{proof}

    The above proposition implies that  $\Pcal_1(\Delta_r)$  admits a Fourier expansion as in \eqref{eqFE_under1.5}. In the next section, we apply a Kronecker-type limit formula, which connects $\Delta_r$ with ``mirabolic Eisenstein series", to compute explicitly the Fourier coefficients of $\Pcal_1(\Delta_r)$.

\section{Eisenstein series}\label{sES} Let $\bone_{O_\infty^r}$ be the characteristic function of $O_\infty^r$ on $\Fi^r$, i.e., 
$\bone_{O_\infty^r}(\vec{u})=1$ (resp. $\bone_{O_\infty^r}(\vec{u})=0$) if $\vec{u}\in O_\infty^r$ (resp. $\vec{u}\in \Fi^r-  O_\infty^r$). 
We define a mirabolic Eisenstein series on $\GL_r(\Fi)$ by 
$$
\cE_r(g, s):=|\det(g)|^s\cdot \int_{\Fi^\times}\left(\sideset{}{'}\sum_{\vec{c}\in A^r}\bone_{O_\infty^r}(\alpha \vec{c} g)\right)|\alpha|^{rs}\dd^\times \alpha, \quad g\in \GL_r(\Fi),\ s\in \C,
$$
where the Haar measure $\dd^\times \alpha$ is normalized so that 
$\Vol(O_\infty^\times,\dd^\times \alpha)=1$. 
\begin{lem}\label{lemEisInvar}\hfill
	\begin{enumerate}
		\item For all $\beta\in \Fi^\times$, $\kappa\in \GL_r(O_\infty)$ and $\gamma\in \GL_r(A)$, we have 
		$$
		\cE_r(\gamma g\beta\kappa, s)=\cE_r(g, s). 
		$$
	\item $\cE_r(g, s)$ converges absolutely for $\mathrm{Re}(s)>1$.  
	\item $\cE_r(g, s)$ has a meromorphic continuation to the 
	whole $s$-plane with poles at $s=0$ and $s=1$. 
	\end{enumerate}
\end{lem}
\begin{proof}
    The analytic properties of arbitrary mirabolic Eisenstein series can be found in \cite[p.~120]{Jacquet-Shalika}.
    For the sake of completeness, we provide here a direct proof in this particular case.

	We clearly have $\bone_{O_\infty^r}(\vec{x}\kappa)=\bone_{O_\infty^r}(\vec{x})$ for any $\kappa\in \GL_r(O_\infty)$ and $\vec{x}\in \Fi^r$.  
	Next, it is easy to see that for any $g\in \GL_r(\Fi)$ and $\alpha\in \Fi^\times$, only finitely many 
	terms in the sum 
	$\sideset{}{'}\sum_{\vec{c}\in A^r}\bone_{O_\infty^r}(\alpha \vec{c} g)$ are nonzero. Since for any $\gamma\in \GL_r(A)$, 
	$\vec{c}\gamma$ runs over $A^r-(0,\dots, 0)$ 
	as $\vec{c}$ runs over $A^r-(0,\dots, 0)$, we conclude that 
	$$\sideset{}{'}\sum_{\vec{c}\in A^r}\bone_{O_\infty^r}(\alpha \vec{c} \gamma g) = 
	\sideset{}{'}\sum_{\vec{c}\in A^r}\bone_{O_\infty^r}(\alpha \vec{c} g).$$ 
	Thus, by \eqref{eq-yDiagForm} we may assume that
	$g=\diag(T^{n_1}, \dots, T^{n_r})$. Let $m=\max\{-n_1, \dots, -n_r\}$. Then we have 
	\begin{align*}
		\int_{\Fi^\times}\left(\sideset{}{'}\sum_{\vec{c}\in A^r}\bone_{O_\infty^r}(\alpha \vec{c} g)\right)|\alpha|^{rs}\dd^\times \alpha 
		& =\sum_{n=-\infty}^\infty \int_{\pi^nO_\infty^\times}\left(\sideset{}{'}\sum_{\vec{c}\in A^r}\bone_{O_\infty^r}(\alpha \vec{c} g)\right)|\alpha|^{rs}\dd^\times \alpha \\ 
		& = \sum_{n=m}^\infty (q^{rn+\sum_{i=1}^r n_i +r}-1)q^{-rns}. 
	\end{align*}
This last sum converges absolutely for $\mathrm{Re}(s)>1$, and is equal to 
\begin{equation}\label{eqExplicitE_r}
|\det(g)|\cdot q^r\cdot \frac{q^{rm(1-s)}}{1-q^{r(1-s)}} - \frac{q^{-rms}}{1-q^{-rs}}. 
\end{equation}
This implies that $\cE_r(g, s)$ has a meromorphic continuation to the whole $s$-plane with poles at $s=0$ and $s=1$. 
Finally, the invariance $\cE_r(\beta g, s)=\cE_r(g, s)$ for $\beta \in \Fi^\times$ easily follows by a change of variables $\alpha\mapsto \beta\alpha$ in the integral. 
\end{proof}

\begin{rem}
    Switching the integration and the summation in the definition of $\Ecal_r(g,s)$, we get
    $$
    \Ecal_r(g,s) \ = \ |\det(g)|^s\cdot
    \sideset{}{'}\sum_{\vec{c} \in A^r}\int_{\Fi^\times}\mathbf{1}_{O_\infty^r}(\alpha\vec{c} g)|\alpha|^{rs}\dd^\times \alpha, \quad g \in \GL_r(F_\infty),\ \re(s)>1. 
    $$
    Recall that for $\vec{x} \in F_\infty^r$, $$\nu_{O_\infty^r}(\vec{x}) = \min\big\{|\alpha|\ \big|\ \vec{x} \in \alpha O_\infty^r\big\}.
    $$
    In particular,
    $$
    \mathbf{1}_{O_\infty^r}(\alpha \vec{c} g) = 1
    \quad \text{ if and only if } \quad 
    |\alpha| \leq \nu_{O_\infty^r}(\vec{c} g)^{-1}.
    $$
    Thus
    \begin{eqnarray}
    \int_{\Fi^\times}\mathbf{1}_{O_\infty^r}(\alpha\vec{c} g)|\alpha|^{rs}\dd^\times \alpha
    &=& \int_{|\alpha| \leq \nu_{O_\infty^r}(\vec{c} g)^{-1} } |\alpha|^{rs} \dd^\times \alpha \nonumber \\
    &=&\frac{1}{1-q^{-rs}} \cdot \frac{1}{\nu_{O_\infty^r}(\vec{c}g)^{rs}}. \nonumber 
    \end{eqnarray}
    Observe that 
    $$
    \nu_{O_\infty^r}(\vec{c}g) \ =\ \nu_{\boz_0}(\vec{c}g) 
    \ =\  |j(g,\boz_0)| \cdot \nu_{g\boz_0}(\vec{c}),
    $$
    where $\boz_0 \in \Omega^r$ is the chosen base point in \eqref{eqboz0}.
    Take $\boz_g:= g \boz_0 = (z_1,...,z_r=1) \in \Omega^r$. 
    We then obtain the following expression of $\Ecal_r(g,s)$:
    \begin{align*}
    \Ecal_r(g,s) &=
    \frac{\det(g)^s}{(1-q^{-rs}) \cdot |j(g,z_0)|^{rs}} \cdot \sideset{}{'}\sum_{\vec{c} \in A^r} \frac{1}{\nu_{\boz_g}(\vec{c})^{rs}} \\
    &=
    \frac{1}{1-q^{-rs}} \cdot \sideset{}{'}\sum_{c_1,...,c_r \in A} \frac{\im(\boz_g)^s}{|c_1z_1+\cdots +c_rz_r|^{rs}}, \quad \re(s)>1,
    \end{align*}
    where $\im(\boz_g) = \det(g)/|j(g,\boz_0)|^r$ is the \textit{imaginary part of $\boz_g$} (cf.\ \cite[formula (2.5) and Lemma 2.11]{WeiKLF}).
This expression makes it clear that  $\Ecal_r(g,s)$ is an analogue of the classical Eisenstein series for $\SL_2(\Z)$; cf.\ \cite[p.\ 65]{Bump}. 
\end{rem}
Let 
$$\widetilde{\Ecal}_r(g,s) := (1-q^{-rs}) \cdot \Ecal_r(g,s)
= \sideset{}{'}\sum_{c_1,...,c_r \in A} \frac{\im(\boz_g)^s}{|c_1z_1+\cdots +c_rz_r|^{rs}}\quad \text{ as } \re(s)>1.
$$
The following Kronecker-type limit formula holds (cf.\ \cite[Theorem 1.1 (2) and Remark 1.2 (2)]{WeiKLF}):
\begin{thm}\label{thmKLF}
    $\widetilde{\Ecal}_r(g,s)$ is holomorphic at $s=0$,
    and
    \begin{align*}
    &\widetilde{\Ecal}_r(g,0) = -1, \\
    &\frac{\partial}{\partial s} \widetilde{\Ecal}_r(g,s)\bigg|_{s=0} = -\ln\im(\boz_g) - \frac{r}{q^r-1}\ln|\Delta_r(\boz_g)|.
    \end{align*}
\end{thm}
Define 
$$
\log\Delta_r(g) := \log_q|\Delta_r(g\boz_0)| - (q^r-1)\cdot \log_q|j(g,\boz_0)|, \quad \forall g \in \GL_r(F_\infty).
$$
Then for $\gamma \in \GL_r(A)$, $\beta \in F_\infty^\times$, and $\kappa \in \GL_r(O_\infty)$, one has
$$\log \Delta_r(\gamma \beta g \kappa) = \log \Delta_r(g) - (q^r-1) \cdot \log_q|\beta|.$$
Moreover, Theorem~\ref{thmKLF} says
\begin{equation}\label{eqKLF}
	\log\Delta_r(g)=\frac{1-q^r}{r}\left(\log_q|\det(g)|+\frac{1}{\ln q}\cdot \frac{\partial}{\partial s}\widetilde{\cE}_r(g,s)\bigg|_{s=0}\right).
\end{equation}
Recall that $S=\diag(T,1,...,1) \in \GL_r(F_\infty)$, and for $g \in \GL_r(F_\infty)$
$$\Pcal_1(\Delta_r)(g)\ =\ \log_q|\Delta_r(gS\boz_0)|-\log_q|\Delta_r(g\boz_0)|
\ =\ \log\Delta_r(gS) - \log\Delta_r(g).
$$
Therefore, from the equality~\eqref{eqKLF}, the Fourier coefficients of the harmonic 1-cochain $\Pcal_1(\Delta_r)$ can be understood by analysing the Eisenstein series $\Ecal_r(g,s)$.

\subsection{Fourier coefficients of Eisenstein series}

The invariant properties of $\cE_r(g, s)$ in Lemma \ref{lemEisInvar} 
imply that $\cE_r(g, s)$ is uniquely determined by its values on $P(\Fi)$ and has a Fourier expansion (cf. Remark \ref{remFC}):
$$
\cE_r\left( \begin{pmatrix} 1 & \vec{x}y \\ & y\end{pmatrix}, s\right)
=\sum_{\vec{a}\in A^{r-1}}\cE_r^\ast(\vec{a}, y, s) \cdot \psi(\vec{a}\bcdot \vec{x}),  
$$
where 
$$
\cE_r^\ast(\vec{a}, y, s) =\int_{A^{r-1}\bs \Fi^{r-1}} \cE_r\left( \begin{pmatrix} 1 & \vec{u}y \\ & y\end{pmatrix}, s\right)\psi(-\vec{a}\bcdot \vec{u})
\dd\vec{u}. 
$$

\begin{lem}\label{lemFCES}
For $\re(s)>1$, we have 
\begin{align*}
\cE_r^\ast(\vec{a}, y, s) = & |\det(y)|^{\frac{s}{1-r}}\cdot \cE_{r-1}\left(y, \frac{r}{r-1}s\right)\cdot \int_{A^{r-1}\bs \Fi^{r-1}}\psi(\vec{a}\bcdot\vec {u}) \dd\vec{u}
\\ 
& + |\det(y)|^{s-1}\sideset{}{'}\sum_{c_1\in A}|c_1|^{-rs}\cdot 
\left(\sum_{\vec{c}_2\in A^{r-1}/c_1A^{r-1}}\psi\left(\vec{a}\bcdot (c_1^{-1}\vec{c}_2)\right)\right) \\ 
& \cdot \int_{|\alpha|\leq 1} \left(\int_{\Fi^{r-1}}\bone_{O_\infty^{r-1}}(\vec{u})\psi(-\alpha^{-1}\vec{a}
\bcdot \vec{u}y^{-1})\dd \vec{u}\right)|\alpha|^{rs-r+1}\dd^\times \alpha. 
\end{align*}
\end{lem}
\begin{proof}
	By definition, 
	$$
	\cE_r^\ast(\vec{a}, y, s) =	|\det(y)|^s\cdot  \int_{A^{r-1}\bs \Fi^{r-1}} \left(
\int_{\Fi^\times}\left(\sideset{}{'}\sum_{\vec{c}\in A^r}\bone_{O_\infty^r}\left(\alpha \vec{c} \begin{pmatrix} 1 & \vec{u}y\\ & y\end{pmatrix}\right)\right)|\alpha|^{rs}\dd^\times \alpha
	\right) \psi(-\vec{a}\bcdot \vec{u}) \dd\vec{u}. 
	$$
	For each $\vec{c} \in A^r$, write $\vec{c}=(c_1, \vec{c}_2)$ with $\vec{c}_2\in A^{r-1}$. We split the sum in the above integral into two sums. 
	
	The first sum is 
	over those $\vec{c}$ for which $c_1=0$: 
	$$
	\sideset{}{'}\sum_{\substack{\vec{c_2}\in A^{r-1}}}\bone_{O_\infty^r}\left(\alpha (0, \vec{c_2}) \begin{pmatrix} 1 & \vec{u}y\\ & y\end{pmatrix}\right) =
	\sideset{}{'}\sum_{\substack{\vec{c_2}\in A^{r-1}}}\bone_{O_\infty^{r-1}}(\alpha \vec{c_2}y),  
	$$
	which, when substituted into the integral, gives  
	\begin{align*}
&\int_{A^{r-1}\bs \Fi^{r-1}} \left(	|\det(y)|^s\cdot  
	\int_{\Fi^\times}\left(\sideset{}{'}\sum_{\substack{\vec{c_2}\in A^{r-1}}}\bone_{O_\infty^{r-1}}(\alpha \vec{c_2}y)\right)|\alpha|^{rs}\dd^\times \alpha
	\right) \psi(-\vec{a}\bcdot \vec{u}) \dd\vec{u} \\ 
& = 	|\det(y)|^{\frac{s}{1-r}}\cdot \cE_{r-1}\left(y, \frac{r}{r-1}s\right)\cdot \int_{A^{r-1}\bs \Fi^{r-1}}\psi(\vec{a}\bcdot\vec {u}) \dd\vec{u}. 
	\end{align*}
This is the first summand of the claimed formula for $\cE_r^\ast(\vec{a}, y, s)$. 

The second sum is over those $\vec{c}$ for which $c_1\neq 0$, which, when substituted into the integral, gives 
\begin{align*}
& |\det(y)|^s\\ & \cdot \int_{\Fi^\times}\sideset{}{'}\sum_{c_1\in A} \bone_{O_\infty}(\alpha c_1)|\alpha|^{rs}
\left( \int_{A^{r-1}\bs \Fi^{r-1}}\left(  \sum_{\vec{c}_2\in A^{r-1}}\bone_{O_\infty^{r-1}}(\alpha c_1(\vec{u}+c_1^{-1} \vec{c}_2)y) \right) \psi(-\vec{a}\bcdot \vec{u}) 
\dd \vec{u}
\right) \dd^\times \alpha. 
\end{align*}
After a change of variables $\alpha\mapsto \alpha/c_1$, this expression can be rewritten as 
\begin{align}\label{eqSecondSumInt}
& |\det(y)|^s\sideset{}{'}\sum_{c_1\in A} \left(|c_1|^{-rs}\int_{\Fi^\times} 
\bone_{O_\infty}(\alpha)\cdot |\alpha|^{rs}\cdot B(c_1, \alpha, y, a)\dd^\times \alpha\right) \\
\nonumber 
& = |\det(y)|^s\sideset{}{'}\sum_{c_1\in A} \left(|c_1|^{-rs}\int_{|\alpha|\leq 1} 
|\alpha|^{rs}\cdot B(c_1, \alpha, y, a)\dd^\times \alpha\right), 
\end{align}
where 
\begin{align*}
& B(c_1, \alpha, y, a):=\int_{A^{r-1}\bs \Fi^{r-1}}\left(  \sum_{\vec{c}_2\in A^{r-1}}\bone_{O_\infty^{r-1}}(\alpha (\vec{u}+c_1^{-1} \vec{c}_2)y) \right) \psi(-\vec{a}\bcdot \vec{u}) \dd \vec{u} \\ 
&=
\int_{A^{r-1}\bs \Fi^{r-1}}\left(  \sum_{\vec{c}_2\in A^{r-1}}\bone_{O_\infty^{r-1}}(\alpha (\vec{u}+c_1^{-1} \vec{c}_2)y)  \psi(-\vec{a}\bcdot (\vec{u}+c_1^{-1} \vec{c}_2))\right)\psi(\vec{a}\bcdot (c_1^{-1}\vec{c}_2)) \dd \vec{u} \\
&=\sum_{\vec{c}_2\in A^{r-1}/c_1A^{r-1}} \psi(\vec{a}\bcdot (c_1^{-1}\vec{c}_2))  
\int_{\Fi^{r-1}}\bone_{O_\infty^{r-1}}(\alpha \vec{u} y) \cdot \psi(-\vec{a}\bcdot \vec{u})\dd \vec{u}. 
\end{align*}
After a change of variables $\vec{u}\mapsto \alpha^{-1}\vec{u}y^{-1}$, we obtain 
(note that $\dd \alpha^{-1}\vec{u}y^{-1}= \frac{|\det(y)|^{-1}}{|\alpha|^{r-1}}\dd \vec{u}$)
$$
\int_{\Fi^{r-1}}\bone_{O_\infty^{r-1}}(\alpha \vec{u} y) \cdot \psi(-\vec{a}\bcdot \vec{u})\dd \vec{u} 
= \frac{|\det(y)|^{-1}}{|\alpha|^{r-1}}\int_{\Fi^{r-1}}\bone_{O_\infty^{r-1}}(\vec{u}) \cdot \psi(-\alpha^{-1}\vec{a}\bcdot \vec{u}y^{-1})\dd \vec{u}. 
$$
Substituting this expression into \eqref{eqSecondSumInt} gives the second 
summand of the claimed formula for $\cE_r^\ast(\vec{a}, y, s)$.
\end{proof}

Because of the invariant properties of $\cE_r(g, s)$ in Lemma \ref{lemEisInvar}, 
the Fourier coefficients $\cE_r^\ast(\vec{a}, y, s)$ have the properties listed in Lemma \ref{lemPropFC}. 

\begin{notn} 
	For $\vec{a}\in A^{r-1}$ and $y \in \GL_{r-1}(F_\infty)$, define
	$$
	m(\vec{a}, y)=\max\big\{m \in \ZZ\cup\{\infty\}\ \big|\ \vec{a}(y^t)^{-1} \in (\pi^mO_\infty)^{r-1} \big\}. 
	$$
	In particular, $m(\vec{a}, y)<\infty$ if and only if $\vec{a}\neq 0$. 
	
	Write $\vec{a} =(a_1, \dots, a_{r-1})$. 
	We say that $0\neq c\in A$ divides $\vec{a}$, and write $c\mid \vec{a}$, if $c$ divides all $a_i$, $1\leq i\leq r-1$. \
\end{notn}

The following lemma is straightforward:

\begin{lem}\label{lem-character-sum}
Let $\vec{a} \in A^{r-1}$.
\begin{itemize}
    \item[(1)] 
    $$
    \int_{A^{r-1}\bs F_\infty^{r-1}}\psi(\vec{a}\bcdot\vec{u}) \dd \vec{u} = 
    \begin{cases} 1 & \text{ if $\vec{a} = 0$,} \\
    0 & \text{ otherwise.}
    \end{cases}
    $$
    \item[(2)]
    For $c \in A$ with $c \neq 0$, we have
    $$
    \sum_{\vec{c}_2 \in A^{r-1}/cA^{r-1}} \psi\big(\vec{a}\bcdot (c^{-1}\vec{c}_2)\big) = 
    \begin{cases}|c|^{r-1} & \text{ if $c \mid \vec{a}$,}\\
    0 & \text{ otherwise.}
    \end{cases}
    $$
\end{itemize}
\end{lem}

Put
	$$
	\sigma(s, \vec{a}):= 
	\begin{cases}
	\sum\limits_{\substack{c\in A_+\\ c\mid \vec{a}}}|c|^s & \text{ if }\vec{a} \in A^{r-1}-\{0\}, \\ 
	(1-q^{1+s})^{-1} & \text{ if }\vec{a}=0. 
	\end{cases}
	$$

\begin{cor}\label{corFCEis}
    Let $y \in \GL_{r-1}(\Fi)$.
	\begin{enumerate}
		\item 
		\begin{eqnarray}
		\cE_r^\ast(0, y, s)&=& |\det(y)|^{\frac{s}{1-r}}\cdot \cE_{r-1}\left(y, \frac{r}{r-1}\cdot s\right) \nonumber \\
		&& +|\det(y)|^{s-1}(q^r-q^{r-1})\cdot \sigma(r-1-rs,0)\cdot  \frac{1}{1-q^{r-1-rs}}. \nonumber 
		\end{eqnarray}
		\item 
		Let $\vec{a} \in A^{r-1}-\{0\}$.
		If $m(\vec{a}, y)\leq 1$, then
		 $\cE_r^\ast(\vec{a}, y, s)=0$. If $m(\vec{a}, y)\geq 2$, then 
		$$
		\cE_r^\ast(\vec{a}, y, s)\ =\ |\det(y)|^{s-1}(q^r-q^{r-1})\cdot \sigma(r-1-rs, \vec{a})\cdot \frac{1-q^{(r-1-rs)(m(\vec{a}, y)-1)}}{1-q^{r-1-rs}}. 
		$$ 
	\end{enumerate}
\end{cor}
\begin{proof}
It suffices to prove the result for $\re(s)>1$.
By Lemma~\ref{lemFCES} and \ref{lem-character-sum}, we get
\begin{eqnarray}
\Ecal^*_r(0,y,s) &=& 
|\det(y)|^{\frac{s}{1-r}}\cdot \Ecal_{r-1}\left(y,\frac{r}{r-1} s\right) \nonumber \\
&& +|\det(y)|^{s-1}\left(\sideset{}{'}\sum_{c \in A}|c|^{-rs}\cdot |c|^{r-1}\right) \cdot q^{r-1} \int_{|\alpha|\leq 1}|\alpha|^{rs-r+1}\dd^\times \alpha. \nonumber
\end{eqnarray}
Thus (1) follows from
$$
\sideset{}{'}\sum_{c \in A}|c|^{r-1-rs} = \frac{q-1}{1-q^{r-rs}} = (q-1)\sigma(r-1-rs,0)
$$
and 
$$
\int_{|\alpha|\leq 1}|\alpha|^{rs-r+1}\dd^\times \alpha = \frac{1}{1-q^{r-1-rs}}.
$$

For (2), if $m(\vec{a}, y)\leq 1$, then $\cE_r^\ast(\vec{a}, y, s)=0$ by Lemma \ref{lemPropFC} (3).
When $m(\vec{a},y)\geq 2$, observe that for $\alpha \in F_\infty^\times$, we have
$$
\int_{O_\infty^{r-1}} \psi(-\alpha^{-1}\vec{a}\bcdot \vec{u}y^{-1})\dd \vec{u}
\ =\ 
\begin{cases}
q^{r-1} & \text{ if $m(\vec{a},\alpha y)\geq 2$;}\\
0 & \text{otherwise.}
\end{cases}
$$
Note that $m(\vec{a},\alpha y) \geq 2$ if and only if $|\alpha|\geq q^{2-m(\vec{a},y)}$.
Hence by Lemma~\ref{lem-character-sum} (2), we get
\begin{eqnarray}
\Ecal^*_r(\vec{a},y,s)
&=& |\det(y)|^{s-1}\left(\sideset{}{'}\sum_{\substack{c \in A}{c\mid \vec{a}}}|c|^{-rs}\cdot |c|^{r-1}\right)
\cdot q^{r-1} \cdot \int_{q^{2-m(\vec{a},y)}\leq |\alpha|\leq 1}
|\alpha|^{rs-r+1} \dd^\times \alpha \nonumber \\
&=& |\det(y)|^{s-1}\cdot (q^r-q^{r-1}) \cdot \sigma(r-1-rs,\vec{a})\cdot \frac{1-q^{(r-1-rs)(m(\vec{a},y)-1)}}{1-q^{r-1-rs}}. \nonumber
\end{eqnarray}

\end{proof}

\subsection{Fourier coefficients of $\log \Delta_r$}

Recall that
$$\widetilde{\Ecal}_r(g,s) = (1-q^{-rs})\cdot \Ecal_r(g,s), \quad \forall g \in \GL_r(F_\infty).$$
Thus, by Corollary \ref{corFCEis}, for $y \in \GL_{r-1}(F_\infty)$ we have 
\begin{eqnarray}
\widetilde{\Ecal}^*_r(0,y,s)
&=&|\det(y)|^{\frac{s}{1-r}} \cdot \widetilde{\Ecal}_{r-1}\left(y,\frac{r}{r-1}\cdot s\right) \nonumber \\
&& +|\det(y)|^{s-1}(q^r-q^{r-1})\cdot \sigma(r-1-rs,0)\cdot  \frac{1-q^{-rs}}{1-q^{r-1-rs}}; \nonumber
\end{eqnarray}
and for $\vec{a} \in A^{r-1}-\{0\}$ with $m(\vec{a},y)\geq 2$, we get
$$\widetilde{\Ecal}_r^*(\vec{a},y,s) = (1-q^{-rs}) \cdot \Ecal_r^*(\vec{a},y,s)$$
with $\Ecal_r^*(\vec{a},y,s)$ holomorphic at $s=0$.
In particular, Corollary~\ref{corFCEis} gives
\begin{eqnarray}
\frac{\partial}{\partial s} \widetilde{\Ecal}^*_r(0,y,s)\bigg|_{s=0}
&=& \frac{1}{1-r} \cdot \ln|\det(y)| \cdot \widetilde{\Ecal}_{r-1}(y, 0) + \frac{r}{r-1} \cdot \frac{\partial}{\partial s} \widetilde{\Ecal}^*_{r-1}(y,s)\bigg|_{s=0} \nonumber \\
&& + \ r \ln q \cdot |\det(y)|^{-1}\cdot \sigma(r-1,0)\cdot \frac{q^r-q^{r-1}}{1-q^{r-1}}.
\nonumber
\end{eqnarray}
By \eqref{eqKLF}, we get
\begin{align*} &\widetilde{\Ecal}_{r-1}(y,0) = -1, \quad \text{and } \\ 
& \frac{\partial}{\partial s} \widetilde{\Ecal}^*_{r-1}(y,s)\bigg|_{s=0}
= -\ln|\det(y)|-\frac{(r-1)\ln q}{q^{r-1}-1} \cdot \log \Delta_{r-1}(y).
\end{align*}
Hence
\begin{eqnarray}\label{eq-FCE1}
\frac{\partial}{\partial s} \widetilde{\Ecal}^*_r(0,y,s)\bigg|_{s=0}
&=& - \ln|\det(y)| -\frac{r\ln q}{q^{r-1}-1} \cdot \log \Delta_{r-1} (y) \nonumber \\
&& + \ r \ln q \cdot |\det(y)|^{-1}\cdot \sigma(r-1,0)\cdot \frac{q^r-q^{r-1}}{1-q^{r-1}}.
\end{eqnarray}
For $\vec{a} \in A^{r-1}-\{0\}$ with $m(\vec{a},y)\geq 2$, Corollary~\ref{corFCEis} says
\begin{eqnarray}\label{eq-FCE2}
\frac{\partial}{\partial s} \widetilde{\Ecal}^*_r(\vec{a},y,s)\bigg|_{s=0}
&=& r\ln q \cdot \Ecal^*_r(\vec{a},y,0) \nonumber \\
&=& r\ln q  \cdot |\det(y)|^{-1}\cdot \sigma(r-1,\vec{a}) \cdot \frac{q^r-q^{r-1}}{1-q^{r-1}} \cdot (1-q^{(r-1)(m(\vec{a},y)-1)}).
\end{eqnarray}


\begin{cor}\label{corFClogD}
	Given $\vec{x}\in \Fi^{r-1}$ and $y \in \GL_{r-1}(F_\infty)$, we have 
	the following Fourier expansion 
	$$
	\log\Delta_r\begin{pmatrix}1&\vec{x}y\\ & y\end{pmatrix} = 
	\log\Delta_r(\vec{x}, y)=\sum_{\substack{\vec{a}\in A^{r-1}\\ m(\vec{a},y)\geq 2}} \log\Delta_r^\ast(\vec{a}, y)\psi(\vec{a}\bcdot \vec{x}), 
	$$
	where 
	$$
	\log\Delta_r^\ast(0, y)=\frac{q^r-1}{q^{r-1}-1}\log\Delta_{r-1}(y)-\frac{q^r-q^{r-1}}{q^{r-1}-1}|\det(y)|^{-1}, 
	$$
	and for $0\neq \vec{a}\in A^{r-1}$ with $m(\vec{a},y)\geq 2$, 
	$$
	\log\Delta_r^\ast(\vec{a}, y)=
	\frac{(q^r-1)(q^{r}-q^{r-1})}{q^{r-1}-1}\left(1-q^{(r-1)(m(\vec{a}, y)-1)}\right)\cdot |\det(y)|^{-1}\sigma(r-1, \vec{a}).
	$$
\end{cor}

\begin{proof} 
    From \eqref{eqKLF}, we have
    $$
    \log \Delta_r^*(0,y) = \frac{1-q^r}{r}\left(\log_q|\det(y)| + \frac{1}{\ln q}\frac{\partial}{\partial s} \widetilde{\cE}_r^*(0,y,s)\bigg|_{s=0}\right),
    $$
    and for $\vec{a} \in A^{r-1}$ with $m(\vec{a},y)\geq 2$,
    $$
    \log \Delta_r^*(\vec{a},y) = \frac{1-q^r}{r\ln q}\cdot \frac{\partial}{\partial s} \widetilde{\cE}_r^*(\vec{a},y,s)\bigg|_{s=0}.
    $$
    Hence the result follows immediately from \eqref{eq-FCE1} and \eqref{eq-FCE2}.
\end{proof}

\subsection{Fourier coefficients of $\Pcal_1(\Delta_r)$}
For $g \in \GL_r(F_\infty)$, recall that
$$\cP_1(\Delta_r)(g)=\log\Delta_r(gS)-\log\Delta_r(g).$$
Suppose $g = \begin{pmatrix}1 & \vec{x}y \\ & y\end{pmatrix}$,
where $\vec{x} \in F_\infty^{r-1}$ and $y \in \GL_{r-1}(F_\infty)$.
Since
$gS=\begin{pmatrix} T & \vec{x}y\\ & y\end{pmatrix}=\begin{pmatrix} 1 & \vec{x}yT^{-1}\\ & yT^{-1}\end{pmatrix}T$, we get
$$\log \Delta_r(gS) = \log \Delta_r(\vec{x},yT^{-1}) - (q^r-1),$$
and arrive at the formula
\begin{equation}\label{eqP_1lodD}
\cP_1(\Delta_r)(\vec{x}, y)=-(q^r-1)+\log\Delta_r(\vec{x}, yT^{-1})-\log\Delta_r(\vec{x}, y). 
\end{equation}
\begin{thm}\label{thmFEPDelta}
	Given $\vec{x}\in \Fi^{r-1}$ and $y=\diag(T^{n_2}, \dots, T^{n_r})$ with $n_2, \dots, n_r\in \Z$, we have 
	$$
	\cP_1(\Delta_r)(\vec{x}, y)=\sum_{\substack{\vec{a}\in A^{r-1}\\ m(\vec{a}, y)\geq 2}}
	\cP_1(\Delta_r)^\ast(\vec{a}, y)\cdot \psi(\vec{a}\bcdot \vec{x}), 
	$$
	where 
	$$
\cP_1(\Delta_r)^\ast(0, y)=-(q-1)q^{r-1}|\det(y)|^{-1}, 
	$$
	and for $0\neq \vec{a}\in A^{r-1}$ with $m(\vec{a}, y)\geq 2$, 
	$$
	\cP_1(\Delta_r)^\ast(\vec{a}, y)=(q^r-1)(q-1)q^{r-1} |\det(y)|^{-1}\cdot \sigma(r-1, \vec{a}). 
	$$
\end{thm}
\begin{proof}
	This easily follows from \eqref{eqP_1lodD} and Corollary \ref{corFClogD}, combined with the following observations  
	\begin{align*}
	|\det(yT^{-1})|^{-1} & \ =\ q^{r-1}|\det(y)|^{-1},\\ 
	\log\Delta_{r-1}(yT^{-1}) &\ =\ \log\Delta_{r-1}(y) +(q^{r-1}-1),\\ 
	m(\vec{a}, yT^{-1})&\ =\ m(\vec{a}, y)-1. 
	\end{align*}
	A small extra observation needed to justify the deduction of the formula for $\cP_1(\Delta_r)^\ast(\vec{a}, y)$ when $m(\vec{a}, y)=2$ 
	is that the formula for $\log\Delta_r^\ast(\vec{a}, y)$ in Corollary \ref{corFClogD} 
	for $m(\vec{a}, y)\geq 2$ is also valid for $m(\vec{a}, y)=1$ as it is equal to $0$ in that case. 
\end{proof}

\begin{rem}\hfill
    \begin{itemize}
	\item[(i)] Recall that
	$
	\sigma(r-1, 0)=
	(1-q^{r})^{-1}. 
	$
	Thus the formulas in Theorem~\ref{thmFEPDelta} can be combined into 
	$$
	\cP_1(\Delta_r)^\ast(\vec{a}, y)=(q^r-1)(q-1)q^{r-1} |\det(y)|^{-1}\cdot \sigma(r-1, \vec{a}) 
	$$
	for all $\vec{a}\in A^{r-1}$ and $y \in \GL_{r-1}(F_\infty)$ with $m(\vec{a},y)\geq 2$. 
	\item[(ii)]
	When $r=2$, the above formula of the Fourier coefficients of $\Pcal_1(\Delta_r)$ coincides with Gekeler's formula in 
	\cite[(2.5) and Cor.\ 2.8]{GekelerDelta}.
	\end{itemize}
\end{rem}

\begin{cor}\label{corPDeltaWeyl}
	Let $y=\diag(T^{n_2}, \dots, T^{n_r})$ with $n_i\leq 1$ for all $i=2, \dots, r$. 
	Then 
	for any $\vec{x}\in \Fi^{r-1}$ we have 
	$$
	\cP_1(\Delta_r)(\vec{x}, y)=-(q-1) \cdot q^{r-1-(n_2+\cdots+n_r)}. 
	$$ 
	In particular, $\cP_1(\Delta_r)(0, T I_{r-1})=-(q-1)$. 
\end{cor}
\begin{proof}
	For the given $y$, we have $m(\vec{a}, y)\leq 1$ for any $\vec{a}\neq 0$. Thus $\cP_1(\Delta_r)(\vec{x}, y)=\cP_1(\Delta_r)^\ast(0, y)$. 
	The claim now follows from Theorem \ref{thmFEPDelta}. 
\end{proof}

\begin{rem}\label{remOurFvsGek} 
    Consider the Weyl chamber $\cW^r$, which is the subcomplex of $\cB^r$  with set of vertices 
	$$\La_{\bok}=[T^{-k_1}O_\infty\oplus \cdots \oplus T^{-k_r}O_\infty] = [O_\infty^r\cdot (\diag(T^{k_1},...,T^{k_r})^{-1}],$$ 
	where $\bok=(k_1, \dots, k_r)\in \Z^r$ and $k_1\geq k_2\geq \cdots\geq k_r=0$. 
	It is a well-known fact that $\cW^r$ is a fundamental domain for the action of $\GL_r(A)$ on $\cB^r$; cf.\ \cite[(2.2)]{GekelerDMHR1}. 
	Corollary \ref{corPDeltaWeyl} can be used to give a formula for some of the edges of $\cW^r$, namely for 
	$e_{\bok}=(\La_{\bok}, \La_{\bok'})$, where $\bok'=\bok+(1, 0, \dots, 0)$. 
	Indeed, take $y=\diag(T^{k_2-k_1}, T^{k_3-k_1}, \dots, T^{k_r-k_1})$. 
	From 
	Corollary \ref{corPDeltaWeyl} we obtain 
	$$
	\cP(\Delta_r)(e_{\bok}) = \cP_1(\Delta_r)(0, y)= -(q-1)\cdot q^{(r-1)(k_1+1)-(k_2+\cdots+k_r)}.
	$$
Theorem 5.5 in \cite{GekelerDMHR1} gives an explicit formula for the values of $\cP(\Delta_r)$ on all edges of $\cW^r$ 
(see also \cite[Cor. 2.9]{GekelerDelta} for $r=2$). 
\end{rem}

\begin{cor}\label{corRootDelta} The largest integer $m$ such that there exists an $m$-th root of $\Delta$ in $\cO(\Omega^r)^\times$ 
	is $q-1$. 
\end{cor}
\begin{proof} 
	In \cite{BB} and \cite{GekelerDMHR1}, it is shown that there is a holomorphic function $H$ on $\Omega^r$ such that 
	\begin{align}\label{eqh}
	H^{q-1} & =\Delta_r, \quad \text{and} \\ 
	\nonumber 
	H(\gamma \boz) &=\det(\gamma)^{-1}\cdot  j(\gamma, \boz)^{\frac{q^r-1}{q-1}}\cdot H(\boz) \quad \text{for all }\quad \gamma\in \GL_r(A).  
	\end{align}
	Hence it is enough to show that if $\Delta_r^{1/m}\in \cO(\Omega^r)^\times$ then $m$ divides $q-1$. 
	For this observe that from the Gekeler--van der Put exact sequence in \eqref{eqGvdPseq} and Corollary \ref{corPDeltaWeyl} we get:
	$$
	\cP_1(\Delta_r^{1/m})(0, TI_{r-1})=\frac{1}{m}\cP_1(\Delta_r)(0, TI_{r-1}) =-\frac{q-1}{m}\in \Z.$$ 
\end{proof}


\section{Modular units}\label{sMU} 
In this section we are interested in certain elements of $\cO(\Omega^r)^\times$ which are invariant 
under the action of $\G_0^r(\fn)$ for a given nonzero ideal $\fn\lhd A$. Because the short-exact sequence \eqref{eqGvdPseq} 
is $\GL_r(\Fi)$-equivariant, taking the long exact cohomology sequence for the action of $\G^r_0(\fn)$ on \eqref{eqGvdPseq}, one obtains 
\begin{equation}\label{eqLongCohomSeq}
0\To \Ci^\times\To (\cO(\Omega^r)^\times)^{\G_0^r(\fn)}\overset{\cP}{\To} \Har^1(\cB^r, \Z)^{\G_0^r(\fn)}\To \Hom(\G_0^r(\fn), \Ci^\times). 
\end{equation}
Let $(\G_0^r(\fn))^\ab:=\G_0^r(\fn)/[\G_0^r(\fn), \G_0^r(\fn)]$ be the abelianization $\G_0^r(\fn)$. When $r=2$, 
$(\G_0^r(\fn))^\ab$ is a finitely generated abelian group whose rank goes to infinity as $\deg(\fn)\to \infty$; cf.\ \cite{GN}. 
The situation is quite different for $r\geq 3$: 

\begin{thm}\label{thmKazdanT}
	If $r\geq 3$, then $(\G_0^r(\fn))^\ab$ is a finite group. 
\end{thm} 
\begin{proof} This is a consequence of Kazhdan's property (T) for lattices of semi-simple groups of rank at least 2 
	over local fields; see \cite[Ch.\ VIII]{Margulis}. 
\end{proof}

Hence the group $(\cO(\Omega^r)^\times)^{\G_0^r(\fn)}$, up to constant multiples, is a subgroup $\Har^1(\cB^r, \Z)^{\G_0^r(\fn)}$, 
and when $r\geq 3$, the image of the Gekeler--van der Put map $\cP$ has finite index in $\Har^1(\cB^r, \Z)^{\G_0^r(\fn)}$. In the special case 
when $\fn=1$, this latter group is trivial. 

\begin{lem}\label{lemH(1)=0}
	$\Har^1(\cB^r, \Z)^{\GL_r(A)}=0$.
\end{lem}
\begin{proof} 
	There is a unique edge in the Weyl chamber $\cW^r$ of a given type $1\leq s\leq r-1$ terminating at $[\La_0]$. 
	On the other hand, as we have mentioned in Remark \ref{remOurFvsGek}, $\cW^r$ is a fundamental domain for the action of $\GL_r(A)$ on $\cB^r$. 
Thus, $\GL_r(A)$ acts transitively on the set of edges $\Ed_s^t([\La_0])$. 
Suppose $h\in \Har^1(\cB^r, \Z)$  is invariant under the action of $\GL_r(A)$. Using property (2) of Definition \ref{defnHarm-1}, we get  
	$$\sum_{e\in \Ed_s^t([\La_0])} h(e)=\# (\Ed_s^t([\La_0]))\cdot  h(e_0)=0$$ for any fixed $e_0\in \Ed^t_s([\La_0])$. 
	Thus, we get $h(e)=0$
	for all $e$ with $t(e)=[\La_0]$. Now, using properties (1) and (4) of Definition \ref{defnHarm-1}, we conclude that $h(e)=0$ for all edges 
	of the (unique) $(r-1)$-simplex $\sigma_0$ in $\cW^r$ containing $[\La_0]$. 
	
	Next, let $v$ be any vertex of $\sigma_0$ different from $v_0$. 
    If we disregard the edges of $\sigma_0$, then there is a unique edge in $\cW^r$ of a given type terminating at $v$, so we can apply the same 
	argument as above to $v$. Repeating this process eventually shows that $h(e)=0$ for all edges of $\cW^r$, hence 
	$h(e)=0$ for all edges of $\cB^r$. 
\end{proof}

From now on we denote $\Delta(\boz):=\Delta_r(\boz)$. Given a nonzero ideal $\fn\lhd A$, define  
$$
\fn\ast (z_1, z_2, \dots, z_r)=(\fn z_1, z_2, \dots, z_{r-1}, z_r). 
$$
Denote 
$$
\Delta_\fn(\boz):=\Delta(\fn\ast \boz)=\Delta(\fn z_1, z_2, \dots, z_r) 
$$
and 
$$
\Theta_\fn(\boz)=\Delta(\boz)/\Delta_\fn(\boz)\in \cO(\Omega^r)^\times. 
$$
(Recall that by abuse of notation $\fn$  denotes also the monic generator of this ideal.)

\begin{lem}
$\Theta_\fn(\gamma \boz)=\Theta_\fn(\boz)$ for all $\gamma\in \G_0^r(\fn)$. 
	In particular, $\cP(\Theta_\fm)\in \Har^1(\cB^r, \Z)^{\G_0^r(\fn)}$ for any $\fm\mid \fn$. 
\end{lem}
\begin{proof}
	Given $\gamma=(\gamma_{i,j})\in \G^r_0(\fn)$, define $\widetilde{\gamma}=(\widetilde{\gamma}_{i,j})$ by  
	$$
	\widetilde{\gamma}_{i,j} = 
	\begin{cases}
	\gamma_{i,j} & \text{if } i, j\geq 2 \text{ or }i=j=1,\\ 
	\gamma_{i,j}/\fn &  \text{if } j=1, i\geq 2, \\ 
	\fn \gamma_{i,j} &  \text{if } i=1, j\geq 2. 
	\end{cases}
	$$
	It is easy to check that $\widetilde{\gamma}\in \GL_r(A)$ and 
	$$
	j(\widetilde{\gamma}, \fn\ast \boz) =\frac{\gamma_{r, 1}}{\fn}(\fn z_1)+\gamma_{r, 2}z_2+\cdots+\gamma_{r,r}z_r = j(\gamma, \boz). 
	$$
	Observe that $\fn\ast(\gamma \boz)=\widetilde{\gamma}(\fn\ast \boz)$, which implies (using Proposition \ref{propDeltaFE})
	$$
	\Delta_\fn(\gamma\boz)=\Delta(\fn\ast(\gamma \boz))=\Delta(\widetilde{\gamma}(\fn\ast \boz))=j(\widetilde{\gamma}, \fn\ast \boz)^{q^r-1}
	\Delta_\fn(\boz) = j(\gamma, \boz)^{q^r-1}\Delta_\fn(\boz). 
	$$
	Since $\Delta(\gamma\boz)= j(\gamma, \boz)^{q^r-1}\Delta(\boz)$, we get $\Theta_\fn(\gamma\boz)=\Theta_\fn(\boz)$. 
	
	For the last claim of the lemma, note that $\G^r_0(\fn)$ is a subgroup of $\G^r_0(\fm)$ if $\fm\mid \fn$, so 
	$\Theta_\fm$ is $\G^r_0(\fn)$-invariant since it is $\G^r_0(\fm)$-invariant. 
\end{proof}

\subsection{Fourier expansion of $\Pcal_1(\Theta_\nfk)$}

For $\vec{a} \in A^{r-1}$,
denote
$$
\sigma_\fn(s, \vec{a}):=\sum_{\substack{c\in A_+\\ c\mid \vec{a},\ \fn\nmid c}} |c|^s \quad \text{ if $\vec{a} \neq 0$, \quad and } \quad 
\sigma_\fn(s,0) := (1-|\fn|^s)\cdot \sigma(s,0) = \frac{1-|\fn|^s}{1-q^{1+s}}. 
$$
Then 
$$\sigma_\fn(s,\vec{a}) = \sigma(s,\vec{a}) - |\nfk|^s \sigma(s,\nfk^{-1}\vec{a}),$$
where $\sigma(s,\nfk^{-1}\vec{a}):= 0$ if $\nfk \nmid \vec{a}$.

\begin{thm}\label{propThetaFC}
	Given $\vec{a} \in A^{r-1}$ and $y \in \GL_{r-1}(F_\infty)$ with $m(\vec{a},y) \geq 2$, we have 
	$$
	\cP_1(\Theta_\fn)^\ast(\vec{a}, y)=(q^r-1)(q-1)q^{r-1} |\det(y)|^{-1}\cdot \sigma_\fn(r-1, \vec{a}). 
	$$
\end{thm}

\begin{proof}
	From the definitions, for $g\in \GL_r(\Fi)$, we have 
	$$\cP_1(\Delta_\fn)(g)=\cP_1(\Delta)(g)-\cP_1(\Delta)\left(\begin{pmatrix} \fn & \\ & I_{r-1}\end{pmatrix}g\right).$$ 
	On the other hand, for $\vec{x} \in F_\infty^{r-1}$ and $y \in \GL_{r-1}(F_\infty)$ one has
	$$
	\begin{pmatrix} \fn & \\ & I_{r-1}\end{pmatrix} \begin{pmatrix} 1 & \vec{x}y\\ & y \end{pmatrix} =  
	\begin{pmatrix} 1 & (\fn\vec{x}) (\fn^{-1} y)\\ & \fn^{-1} y \end{pmatrix} (\fn I_r).
	$$
	Hence
	$$
	\cP_1(\Delta_\fn)(\vec{x}, y)=
	\cP_1(\Delta)(\vec{x},y) -
	\cP_1(\Delta)(\fn \vec{x}, \fn^{-1} y). 
	$$
	This implies that for $\vec{a} \in A^{r-1}$, 
	$$
		\cP_1(\Delta_\fn)^\ast(\vec{a}, y) = \cP_1(\Delta)^\ast(\vec{a},y) - 
		\begin{cases} 
		\cP_1(\Delta)^\ast (\fn^{-1}\vec{a}, \fn^{-1} y) & \text{if }\fn\mid \vec{a}; \\ 
		0 & \text{if }\fn\nmid \vec{a}. 
		\end{cases}
	$$
	Then Theorem \ref{thmFEPDelta} leads us to
	\begin{eqnarray}
	\cP_1(\Theta_\fn)^\ast(\vec{a}, y)&=&(q^r-1)(q-1)q^{r-1}|\det(y)|^{-1}\cdot\Big(\sigma(r-1, \vec{a})-|\fn|^{r-1}\sigma(r-1, \fn^{-1}\vec{a})\Big) \nonumber \\
	&=& (q^r-1)(q-1)q^{r-1}|\det(y)|^{-1} \cdot \sigma_\fn(r-1,\vec{a}). \nonumber
	\end{eqnarray}
\end{proof}

Similar to Corollary~\ref{corRootDelta}, we get:

\begin{cor}\label{corPDeltaNWeyl}
	Let $y=\diag(T^{n_2}, \dots, T^{n_r})$ with $n_i\leq 1$ for all $i=2, \dots, r$. 
	Then 
	for any $\vec{x}\in \Fi^{r-1}$ we have 
	$$
	\cP_1(\Theta_\fn)(\vec{x}, y)=(q-1)(|\fn|^{r-1}-1)q^{r-1-(n_2+\cdots+n_r)}. 
	$$ 
	In particular, $\cP_1(\Theta_\fn)(0, T I_{r-1})=(q-1)(|\fn|^{r-1}-1)$. 
\end{cor}

\begin{cor}\label{corRootTheta} 
	The largest integer $m$ such that there exists an $m$-th root of $\Theta_\fn$ in $\cO(\Omega^r)^\times$ 
	divides 
	$$
	(q-1)\cdot \gcd(|\fn|-1, q^r-1)=(q-1)(q^{\gcd(\deg(\fn), r)}-1).
	$$ 
\end{cor}
\begin{proof}
	Suppose $\Theta_\fn^{1/m}\in \cO(\Omega^r)^\times$. Then from the Gekeler--van der Put exact sequence \eqref{eqGvdPseq} we obtain:
	$$
	\cP_1(\Theta_\fn^{1/m})(0, T I_{r-1})=\frac{1}{m}\cP_1(\Theta_\fn)(0, T I_{r-1})=\frac{1}{m}(q-1)(|\fn|^{r-1}-1)\in \Z. 
	$$
	Hence $m$ divides $(q-1)(|\fn|^{r-1}-1)$. 
	On the other hand, 
	take $\vec{x} = (0,...,0,T^{-1}) \in F_\infty^{r-1}$ and $y = \diag(T^2,...,T^2)$.
	Then for $\vec{a} \in A^{r-1}$, one has $m(\vec{a},y) \geq 2$ if and only if $\vec{a} \in \FF_q^{r-1}$.
	From Theorem~\ref{propThetaFC}, we have
$$
	\Pcal_1(\Theta_\fn)(\vec{x},y) = (q^r-1)(q-1)q^{r-1}|\det y|^{-1} \sum_{\vec{a} \in \FF_q^{r-1}} \sigma_\fn(r-1,\vec{a}) \psi(\vec{a}\bcdot \vec{x}). 
$$
Since 
\begin{align*}
\sum_{\vec{a} \in \FF_q^{r-1}} \sigma_\fn(r-1,\vec{a}) \psi(\vec{a}\bcdot \vec{x}) &= 
\frac{|\fn|^{r-1}-1}{q^r-1}+\sum_{0\neq \vec{a} \in \FF_q^{r-1}}\psi(\vec{a}\bcdot \vec{x})\\ 
 & =\frac{|\fn|^{r-1}-1}{q^r-1}-1,
\end{align*}
we get 
$$
	\Pcal_1(\Theta_\fn)(\vec{x},y) = (q-1)q (q^{(\deg \fn -1)(r-1)-1}-1).
$$
Thus $m$ also divides $(q-1)q (q^{(\deg \fn -1)(r-1)-1}-1)$.
	Note that
    \begin{align*}
    \gcd((q-1)(q^{\deg(\fn)}-1), q(q-1)(q^{(\deg(\fn)-1)(r-1)-1}-1)) &=(q-1)\cdot (q^{\gcd(\deg(\fn), r)}-1) 
    \end{align*}
    Therefore the result holds.
\end{proof}

Later in this section we will show that the largest integer $m$ such that there exists an $m$-th root of $\Theta_\fn$ in $\cO(\Omega^r)^\times$ 
is actually 
$(q-1)\cdot \gcd(|\fn|-1, q^r-1)$. 

\begin{cor}\label{corLinIndepTheta}
	Let $\fp_1, \dots, \fp_k$ be distinct primes and 
	$\fn=\fp_1\fp_2\cdots\fp_k$. Let $S$ be the set of monic divisors of $\fn$ not equal to $1$. 
	The  
	$2^k-1$ harmonic $1$-cochains 
	$\cP(\Theta_\fm)$, $\fm\in S$,  are linearly independent over $\Q$. 
\end{cor}
\begin{proof}
	Denote $\vec{\fm}:=(\fm, \fm, \dots, \fm)\in A^{r-1}$. 
	It is enough to show that the matrix $$\left(\cP_1(\Theta_{\fm'})^\ast(\vec{\fm}, y)\right)_{\fm, \fm'\in S}$$ 
	has a nonzero determinant for some $y$. Let $y=T^{2+\deg(\fn)}I_{r-1}$, so that 
	$\cP_1(\Theta_{\fm'})^\ast(\vec{\fm}, y)\neq 0$. Using Theorem \ref{propThetaFC}, it is enough to show that 
	$\det(\sigma_{\fm'}(r-1, \vec{\fm}))_{\fm, \fm'\in S}\neq 0$. This follows from the next lemma. 
\end{proof}

\begin{lem}\label{lemDetSigma}
	Let $\fn$ and $S$ be as in Corollary \ref{corLinIndepTheta}. 
	Then 
	$$
	\det(\sigma_{\fm'}(s, \fm))_{\fm, \fm'\in S}= - |\fn|^{s\cdot (2^{k-1}-1)}.  
	$$
\end{lem}
\begin{proof}
	Let $L[x_1, \dots, x_k]$ be the polynomial ring in indeterminates $x_1, \dots, x_k$ 
	over a field $L$. Let $\cS$ be the set of monomials 
	$$
	\cS=\{x_{i_1}x_{i_2}\cdots x_{i_t}\mid 1\leq t\leq k,\quad  1\leq i_1<i_2<\cdots<i_t\leq k\}, 
	$$
	and $\cS'=\cS\cup \{1\}$. Let $R$ be a commutative ring and $\phi: \cS'\to R$ be any map. 
	For $a\in \cS'$ and $b\in \cS$ define 
	$$
	\varsigma(a; b)=\sum_{\substack{c\in \cS'\\ c\mid a,\  b\nmid c}} \phi(c). 
	$$
	We claim that 
	\begin{equation}\label{eqvarsigma1}
	\det(\varsigma(a; b))_{a, b\in \cS}= - \prod_{\substack{c\in \cS'\\ c\neq x_1\cdots x_k}} \phi(c)
	\end{equation}
	Assuming this, and taking $L= \QQ$, $R = \CC$, 
	$$
	\phi(x_{i_1}x_{i_2}\cdots x_{i_t}):= 
	|\pfk_{i_1}\cdots \pfk_{i_t}|^s \quad \text{ for } x_{i_1}x_{i_2}\cdots x_{i_t} \in \cS', 
	$$
	the lemma follows from \eqref{eqvarsigma1}, combined with the observation that 
	$$
	\prod_{c \in \cS'} \phi(c) = \prod_{\fm\in S}|\fm|^s=|\fn|^{s\cdot 2^{k-1}}. 
	$$

Let $M=(\varsigma(a; b))_{a, b\in \cS}$, where $a$ corresponds to a row and $b$ to a column of $M$, and we assume that 
	the elements of $\cS$ are arranged in the natural lexicographic order. Note that for any $a, b\in \cS$ we have
	\begin{equation}\label{eqvarsigma}
	\sum_{\substack{c\in \cS'\\ c\mid a}}(-1)^{\deg(c)}\varsigma(c; b)=
	\begin{cases} 
	0, & \text{if }b\mid a; \\ 
	(-1)^{\deg(a)}\phi(a), & \text{if }b\nmid a. 
	\end{cases}
	\end{equation}
	(Here $\deg(x_{i_1}x_{i_2}\cdots x_{i_t})=t$ and $\deg(1)=0$.) 
	Hence, using elementary row operations, we can transform $M$ into the matrix $M'$ whose last row is a row of $\big((-1)^{k+1} \cdot \phi(1)\big)$'s,
	and all other rows are the same as in $M$. Now, with the help of \eqref{eqvarsigma}, applying appropriate 
	elementary row operations to $M'$, we obtain 
	the matrix $M''=(\tau(a; b))_{a, b\in \cS}$, where 
	$$
	\tau(a; b)=
	\begin{cases} 
	(-1)^{k+1} \cdot \phi(1) , & \text{if }a=x_1x_2\cdots x_k;\\
	0, & \text{if }b\mid a\neq x_1x_2\cdots x_k; \\ 
	\phi(a), & \text{if }b\nmid a\neq x_1x_2\cdots x_k. 
	\end{cases}
	$$
	Without loss of generality, we may assume $\phi(1) \in R^\times$.
	Subtracting $(-1)^{k+1} \phi(1)^{-1}\phi(a)$ multiple of the last row of $M''$ from the $a$-th row for $a \neq x_1\cdots x_k$, we get the matrix
	$M'''=(\tau'(a;b))_{a,b \in \cS}$, where
	$$
	\tau'(a; b)=
	\begin{cases} 
	(-1)^{k+1} \cdot \phi(1) , & \text{if }a=x_1x_2\cdots x_k;\\
	-\phi(a), & \text{if }b\mid a\neq x_1x_2\cdots x_k; \\ 
	0, & \text{if }b\nmid a\neq x_1x_2\cdots x_k. 
	\end{cases}
	$$
	The order on $S$ forces that $M'''$ is a lower triangular matrix. Therefore 
	\begin{align*}
	\det(M''') &= (-1)^{k+1} \cdot \phi(1) \cdot \prod_{\substack{a \in \cS\\ a \neq x_1\cdots x_k}} \big(-\phi(a)\big) \\ 
	& = (-1)^{2^k+1}\cdot \prod_{\substack{c \in \cS'\\ c \neq x_1\cdots x_k}}\phi(c), 
	\end{align*}
	which implies \eqref{eqvarsigma1}.
\end{proof}


\subsection{Root of $\Theta_\nfk$}

We shall find the largest root of $\Theta_\fn$, and prove that the estimate of Corollary \ref{corRootTheta} is sharp.
Given a representative $\vec{u}=(u_1, \dots, u_r)\in A^r$ of a nonzero class in 
$(A/\fn)^r$, let 
$$
E_{\vec{u}}(\boz) := \exp_{\La_{\boz}}\left(\frac{u_1}{\fn}z_1+\cdots+\frac{u_{r-1}}{\fn}z_{r-1}+\frac{u_r}{\fn}\right)^{-1}, \quad 
\boz=(z_1, \dots, z_{r-1}, 1)\in \Omega^r. 
$$
Note that $E_{\vec{u}}$ depends only on the class of $\vec{u}$ in $(A/\fn)^r$. We have (cf.\ \cite[p.\ 833]{BB} or \cite[p.\ 885]{GekelerDMHR1})
$$
E_{\vec{u}}(\boz)=\sum_{\substack{\vec{a}\in A^r\\ \vec{a}\equiv \vec{u}\Mod{\fn A^r}}}\frac{1}{a_1z_1+\cdots+a_r}
$$
and 
\begin{equation}\label{eqE_uFunctEq}
E_{\vec{u}}(\gamma\boz) = j(\gamma, \boz)\cdot E_{\vec{u}\gamma} (\boz)\quad \text{for all }\gamma\in \GL_r(A).
\end{equation}
The \textit{Eisenstein series} $E_{\vec{u}}(\boz)$ 
is a 
modular form of weight $1$ for 
$$
\G^r(\fn):=\ker(\GL_r(A)\to \GL_r(A/\fn))
$$
in the sense of \cite{BBP1}. 
Moreover, for $\vec{u}$ with $u_1=0$, $E_{\vec{u}}(\boz)$  is a modular form for 
$$
\G^r_1(\fn):=\left\{\begin{pmatrix} a & b\\ c & d\end{pmatrix}\in \G^r_0(\fn)\quad \bigg|\quad 
d\equiv I_{r-1}\Mod{\fn} 
\right\}. 
$$

Let
$$
F_\nfk(\boz):= \prod_{\substack{\vec{u} \in (A/\fn)^r-\{0\}}{u_1 = 0}} E_{\vec{u}}(\boz).
$$
\begin{lem}
    Let $\nfk$ be a nonzero ideal of $A$. Then
    $$
    \Theta_\nfk(\boz) = \nfk^{q^r-1} \cdot \frac{F_\nfk(\boz)^{q^r-1}}{\Delta(\boz)^{|\nfk|^{r-1}-1}}, \quad \forall \boz \in \Omega^r.
    $$
\end{lem}

\begin{proof}
For $\boz=(z_1, \dots, z_{r-1}, 1)\in \Omega^r$, let  
$$
\fn^{-1}\La_{\fn\ast\boz} = Az_1+\frac{1}{\fn}(A z_2+\cdots+Az_{r-1}+A).
$$
Then $\nfk^{-1} \Lambda_{\nfk\ast\boz} \supset \Lambda_{\boz}$, and from the \lq\lq norm compatibility\rq\rq\ (cf.\ \cite[Lemma 2.5]{WeiKLF}) we get
\begin{eqnarray}
g_r(\nfk^{-1}\Lambda_{\nfk\ast\boz})
&=& g_r(\Lambda_{\boz})^{|\nfk|^{r-1}} \cdot \left(\sideset{}{'}\prod_{\lambda \in \frac{\nfk^{-1}\Lambda_{\nfk\ast\boz}}{\Lambda_{\boz}}}\exp_{\Lambda_{\boz}}(\lambda)\right)^{q^r-1}.
\nonumber \\
&=& \Delta(\boz)^{|\nfk|^{r-1}}
\cdot F_\nfk(\boz)^{1-q^r}.\nonumber
\end{eqnarray}
On the other hand,
$$g_r(\nfk^{-1}\Lambda_{\nfk\ast\boz}) = \nfk^{q^r-1} \cdot g_r(\Lambda_{\nfk\ast\boz}) =
\nfk^{q^r-1} \cdot
\Delta(\nfk\ast\boz).
$$
Therefore
$$
\Theta(\boz) = \frac{\Delta(\boz)}{\Delta(\nfk\ast\boz)}
= \nfk^{q^r-1} \cdot \frac{F_\nfk(\boz)^{q^r-1}}{\Delta(\boz)^{|\nfk|^{r-1}-1}}.
$$

\end{proof}

Denote 
\begin{align*}
U_\fn^{(i)} &= \left\{ (0,\dots, 0, u_i, \dots, u_r)\in A^r\mid \deg(u_i), \dots, \deg(u_r)< \deg(\fn) \text{ and }u_i\neq 0 \text{ is monic}\right\}, \\ 
U_\fn& =\bigsqcup_{i=2}^r U_\fn^{(i)},
\end{align*}
and 
\begin{equation}\label{eqG_fn_def}
G_\fn(\boz)= \prod_{\vec{u}\in U_\fn} E_{\vec{u}}(\boz). 
\end{equation}
By \eqref{eqE_uFunctEq}, we have $E_{\alpha\cdot \vec{u}}(\boz)=\alpha^{-1}E_{\vec{u}}(\boz)$ for any $\alpha\in \F_q^\times$. 
Since $\prod_{\alpha\in \F_q^\times}\alpha=-1$, we get 
$$
F_\fn(\boz)=(-1)^{(r-1)\deg(\fn)}\cdot G_\fn(\boz)^{q-1}. 
$$
On the other hand, recall from \eqref{eqh} that $\Delta = H^{q-1}$, where $H$ is a modular form of weight $(q^r-1)/(q-1)$ and type $1$ for $\GL_r(A)$.
Therefore, 
\begin{align}\label{eqTheta_fn_def}
\Theta_\fn(\boz)
& = \ \ \ \fn^{q^r-1} \cdot  \frac{F_\fn(\boz)^{q^r-1}}{\Delta(\boz)^{|\fn|^{r-1}-1}} \\ 
\nonumber 
& =\text{(const.)} \cdot \left(\frac{G_\fn(\boz)^{q^r-1}}{H(\boz)^{|\fn|^{r-1}-1}}\right)^{q-1}.
\end{align}
Put $\kappa = \gcd(\deg(\fn),r)$ and $m = (q-1)\cdot (q^\kappa-1) = (q-1)\cdot \gcd(|\nfk|^{r-1}-1,q^r-1)$.
Take
\begin{equation}\label{eq_def_theta_fp}
\theta_\fn:=\frac{G_\fn^{\frac{q^r-1}{q^\kappa-1}}}{H^{\frac{|\fn|^{r-1}-1}{q^\kappa-1}}}  \in \cO(\Omega^r)^\times.
\end{equation}
Then for some $\alpha\in \Ci^\times$, we have $(\alpha \theta_\fn)^m=\Theta_\fn$, 
Thus, 
the estimate given in Corollary \ref{corRootTheta} is sharp, and we have proved the following theorem.  

\begin{thm}\label{thmRootTheta}
	The largest integer $m$ such that there exists an $m$-th root of $\Theta_\fn$ in $\cO(\Omega^r)^\times$ 
	is $(q-1)(q^{\kappa}-1)$, where $\kappa=\gcd(\deg(\fn), r)$. 
\end{thm}

Next, we determine the largest integer $m'$ such that $\Theta_\fn$ has an $m'$-th root in $\cO(\Omega^r)^\times$ which is 
moreover $\G_0^r(\fn)$-invariant. 
To do this, we will compute how $\theta_\fn$ transforms under $\G_0^r(\fn)$, by 
generalizing Gekeler's approach in \cite{GekelerDelta} in the case of $r=2$.  

Let 
$
\fn=\fp_1^{m_1}\cdots \fp_s^{m_s}
$
be the prime decomposition of $\fn$. Define 
\begin{align*}
\chi: (A/\fn)^\times \To \prod_{i=1}^s (A/\fp_i)^\times & \To \F_q^\times \\ 
(a_1, \dots, a_s) &\longmapsto \prod_{i=1}^s \Nr_{\F_{\fp_i}/\F_q}(a_i)^{-m_i}, 
\end{align*} 
where the first map is the canonical projection. 

\begin{prop}
	For $\gamma\in \G^r_0(\fn)$, we have 
	$$
	G_\fn(\gamma\boz) =\widetilde{\chi}(\gamma) \cdot j(\gamma, \boz)^{\frac{|\fn|^{r-1}-1}{q-1}}\cdot  G_\fn(\boz),
	$$
	where 
	\begin{align*}
	\widetilde{\chi}: \G^r_0(\fn) & \To \G^r_0(\fn)/\G^r_1(\fn)\cong \GL_{r-1}(A/\fn)\xrightarrow{\det} (A/\fn)^\times \overset{\chi}{\To} \F_q^\times. \\
	\begin{pmatrix} a & b\\ c & d\end{pmatrix} & \longmapsto d \Mod{\fn} 
	\end{align*} 
\end{prop}
\begin{proof}
	First, note that we may view $U_\fn$ as a set of representatives of 
	$$
	\F_q^\times\bs \left((A/\fn)^{r-1}-\{0\}\right). 
	$$
	For $\gamma=\begin{pmatrix} a & b\\ c & d\end{pmatrix}\in \G^r_0(\fn)$, 
	the set $U_\fn d =\{(u_2, \dots, u_r)d\mid (0, u_2, \dots, u_r)\in U_\fn\}$  is still a set of representatives of 
	$\F_q^\times\bs \left((A/\fn)^{r-1}-\{0\}\right)$. 
	Since 
	$E_{\eps\cdot \vec{u}}(\boz)=\eps^{-1}E_{\vec{u}}(\boz)$, $\eps\in \F_q^\times$, one concludes that 
	there exists $\eps_\gamma\in \F_q^\times$ 
	such that 
	$$
	G_\fn(\gamma\boz) =\eps_\gamma \cdot j(\gamma, \boz)^{\frac{|\fn|^{r-1}-1}{q-1}}\cdot  G_\fn(\boz).
	$$
	Since $j(\gamma_1\gamma_2, \boz)=j(\gamma_1, \gamma_2\boz)\cdot 
	j(\gamma_2, \boz)$, we have 
	$$
	\eps_{\gamma_1\gamma_2}=\eps_{\gamma_1}\eps_{\gamma_2} \quad \text{for any}\quad \gamma_1, \gamma_2\in \G^r_0(\fn). 
	$$
	Moreover 
	$$\eps_\gamma=1\quad \text{for any}\quad \gamma\in \G^r_1(\fn). 
	$$
	Therefore, $$\eps: \G^r_0(\fn)\to \F_q^\times, \quad \gamma\mapsto \eps_\gamma,$$ is a homomorphism which factors through 
	\begin{align*}
	\G^r_0(\fn) & \To \G^r_0(\fn)/\G^r_1(\fn)\cong \GL_{r-1}(A/\fn).\\
	\gamma=\begin{pmatrix} a & b\\ c & d\end{pmatrix} & \longmapsto \bar{\gamma}=d \Mod{\fn} 
	\end{align*} 
	On the other hand, any homomorphism $\GL_{r-1}(A/\fn) \To \F_q^\times$ necessarily factors through the determinant 
	$$
	\GL_{r-1}(A/\fn) \xrightarrow{\det} (A/\fn)^\times\overset{\overline{\eps}}{\To} \F_q^\times, 
	$$
	i.e., $\eps(\gamma)=\overline{\eps}(\det(\bar{\gamma}))$. It remains to show that $\overline{\eps}=\chi$. For this we 
	evaluate $\overline{\eps}$ on elements of $\G^r_0(\fn)$ of special type. Namely, assume $\gamma=\begin{pmatrix} a & b\\ c & d\end{pmatrix} 
	\in \G^r_0(\fn)$ with 
	$$
	d\equiv \begin{pmatrix} I_{r-2} & \\ & d_r 
	\end{pmatrix} \Mod{\fn}, \qquad d_r\in (A/\fn)^\times. 
	$$
	If $\vec{u} \in U_\fn^{(i)}$, $2\leq i\leq r-1$, then $\vec{u}\gamma\in U_\fn^{(i)}$. Therefore 
	$$
	\prod_{\vec{u}\in U_\fn^{(i)}}E_{\vec{u}}(\gamma \boz)=\prod_{\vec{u}\in U_\fn^{(i)}} j(\gamma, \boz) E_{\vec{u}\gamma }(\boz) 
	= j(\gamma, \boz)^{\frac{|\fn|-1}{q-1}|\fn|^{r-i}} \prod_{\vec{u}\in U_\fn^{(i)}} E_{\vec{u}}(\boz). 
	$$
	On the other hand, using the argument in the proof of Theorem 3.20 in \cite{GekelerDelta}, one obtains  
	$$
		\prod_{\vec{u}\in U_\fn^{(r)}}E_{\vec{u}}(\gamma \boz)=j(\gamma, \boz)^{\frac{|\fn|-1}{q-1}}\chi(d_r)
		\prod_{\vec{u}\in U_\fn^{(r)}}E_{\vec{u}}(\boz). 
	$$
	Therefore 
	$$
	\eps(\gamma)=\chi(d_r)=\chi(\det \bar{\gamma}). 
	$$
	Since $d_r$ is an arbitrary element of $(A/\fn)^\times$, this implies $\overline{\eps}=\chi$, and hence also the formula 
	of the proposition. 
\end{proof}

\begin{cor}
	The function $\theta_\fn$
	transforms under $\G^r_0(\fn)$ according to the character 
	$$
	\omega_\fn= \widetilde{\chi}^{\frac{r}{\kappa}} \cdot {\det}^{\frac{(r-1) \deg(\fn)}{\kappa}}: \quad \G^r_0(\fn)\To \F_q^\times.   
	$$
	That is, for any $\gamma\in \G^r_0(\fn)$ we have 
	$$
	\theta_\fn(\gamma \boz)= \omega_\fn(\gamma)\cdot \theta_\fn(\boz). 
	$$
\end{cor}
\begin{proof} 
	This follows from \eqref{eqh} and the previous proposition. 
\end{proof}

Let $o(\omega_\fn)$ be the order of $\omega_\fn$. Then $(\theta_\fn)^{o(\omega_\fn)}$ is the least power of $\theta_\fn$ 
which is $\G^r_0(\fn)$-invariant. 

\begin{prop}\label{propCharacteroftheta}
	\begin{align*}
	o(\omega_\fn) & =\frac{q-1}{\gcd\left(q-1, \frac{r}{\kappa}m_1, \dots, \frac{r}{\kappa}m_s, \frac{(r-1)\deg{\fn}}{\kappa}\right)} \\ 
	& =\frac{q-1}{\gcd\left(q-1,m_1, \dots,m_s, \frac{(r-1)\deg{\fn}}{\kappa}\right)}. 
	\end{align*}
	In particular, $o(\omega_\fn)=q-1$ if $\fn$ is square-free. 
\end{prop}
\begin{proof}
	The assertion follows from the same argument as Proposition 3.22 in \cite{GekelerDelta}. 
\end{proof}

\section{Cuspidal divisors}\label{sCD}
The Satake compactifications 
of Drinfeld modular varieties were constructed (at different levels of generality and details of proof) by 
Gekeler \cite{GekelerSatake}, \cite{GekelerDMHR4}, 
Kapranov \cite{Kapranov},  Pink \cite{Pink}, and H\"aberli \cite{Haberli}. The constructions by Gekeler, H\"aberli, and Kapranov 
are rigid-analytic, whereas Pink's construction is algebro-geometric. H\"aberli also proved that the analytic and algebraic Satake 
compactifications give the same variety. 

The Satake compactification can be constructed for any Drinfeld modular variety $Y_\G:=\G\bs \Omega^r$, 
where $\G\subseteq \GL_r(A)$ is a congruence subgroup; we will denote this compactified variety by $X_\G$.   
This is a projective connected normal variety over $\Ci$ of dimension $r-1$ 
containing $Y_\G$ as an open subvariety; cf.\ \cite{GekelerSatake} and the other references listed above.  
The \textit{cusps} of $X_\G$ are the (geometrically) irreducible components of $X_\G-Y_\G$ of dimension $r-2$.  

In this section we study the cuspidal divisors of $X_\Gamma$ (i.e.\ the divisors of $X_\Gamma$ supported at cusps) 
when $\Gamma = \Gamma^r_0(\nfk)$. In particular, we determine the order of the cuspidal divisor class group when $\nfk$ is prime.

\subsection{Meromorphy at cusps}

We first examine the behavior of the elements of $\cO(\Omega^r)^\times$ 
near the ``boundary'' of $\Omega^r$. 
To do so, we first need several definitions.  

	For $\boz=(z_1, z_2, \dots, z_r=1)\in \Omega^r$, we may write $\boz=(z_1, \boz')$, where $\boz'\in \Omega^{r-1}$. 
	Let $\La':=\La_{\boz'}\subset \Ci$ be the lattice associated to $\boz'$ and 
	$\Fi^{r-1}\boz'=\Fi z_2+\Fi z_3+\cdots+\Fi z_r$ be the $\Fi$-vector subspace of $\Ci$ 
	spanned by $\boz'$. Note that $\dim_{\Fi}\Fi^{r-1}\boz'=r-1$. The \textit{parameter at infinity}  is 
	\begin{equation}\label{eq_u(z)}
	u(\boz):=\exp_{\La'}(z_1)^{-1}. 
	\end{equation}
	For $z\in \Ci$ and a subset $X\subset \Ci$, let  
	$$
	d(z, X):=\inf_{x\in X} |z-x|. 
	$$
	Let $\Im(\boz):=d(z_1, \Fi \boz')$. This is an analogue of the imaginary part on the complex upper half-plane. 
	For $n\in \Z_{>0}$, let $\Omega^r_n:=\{\boz\in \Omega^r\mid \Im(\boz)\geq n\}$. One can check that $\Omega^r_n$ 
	is an admissible subset of $\Omega^r$, stable under the action of $\G_\infty'$. These are the basic neighborhoods 
	of infinity in $\Omega^r$; cf.\ \cite[Def.\ 4.12]{BBP1}. Since for $\gamma\in \G_\infty'$ we have 
$u(\gamma \boz)=u(\boz)$, there is a well-defined map 
\begin{align*}
(\G_\infty'\cap \G^r(\fn))\bs \Omega^r_n &\To \Ci^\times\times \Omega^{r-1}, \\ 
\boz &\longmapsto (u(\boz), \boz') 
\end{align*}
which is an open embedding for any $0\neq \fn\lhd A$; see \cite{GekelerSatake}, \cite[Thm.\ 4.16]{BBP1}. 

\begin{rem}
For a congruence subgroup $\Gamma$, the boundary of $X_\G$, as a set, consists of finitely many irreducible components which themselves are Drinfeld modular varieties of smaller 
dimensions. 
The parameter at infinity $u(\boz)$ from \eqref{eq_u(z)} plays an important role in the analytic 
construction of the Satake compactifications. The map 
$(\G\cap \G_\infty')\bs \Omega^r_n\to \Ci^\times \times \Omega^{r-1}$ given by $\boz=(z_1, \boz')\mapsto (u(\boz), \boz')$ 
is an open embedding for $n\gg 0$ and allows one to ``glue" $\Omega^{r-1}$ to $(\G\cap \G_\infty')\bs \Omega^r_n$ as 
the divisor $(u=0)$.  Then, using the map $(\G\cap \G_\infty')\bs \Omega^r_n\to \G\bs \Omega^r$, one adjoints to $Y_\G$ 
a quotient of $\Omega^{r-1}$ by an appropriate congruence group. 
At other boundary neighborhoods of $Y_\G$ the construction is similar. 
Finally, as the glued pieces are themselves 
Drinfeld modular varieties of dimension one less than $Y_\G$, one proceeds  inductively to compactify the glued pieces.
\end{rem}

Following \cite{BBP1}, we define a 
\textit{weak modular form} of weight $k\in \Z_{>0}$ for $\G^r(\fn)$ to be a holomorphic (in the rigid analytic sense) 
function $f: \Omega^r\to \Ci$ satisfying 
$$
f(\gamma(\boz))=j(\gamma, \boz)^k\cdot  f(\boz), \quad \text{for all }\gamma\in \G^r(\fn). 
$$
In this paper, we already encountered such a function in the form of $\Delta_r$. In fact, the coefficient forms $g_i(\boz)$, $1\leq i\leq r$,  
from \eqref{eqDrCoeffForms} are weak modular forms of weight $q^i-1$ for $\GL_r(A)$. Another example is the 
Eisenstein series $E_{\vec{u}}(\boz)$. 

\begin{thm}\label{thm-u-expansion}
    On suitable neighborhoods of infinity $\Omega^r_n$, $n\gg 0$, 
	every weak modular form $f$ of weight $k$ for $\G^r(\fn)$ admits a uniformly convergent $u$-expansion 
	\begin{equation}\label{eqBBP1expansion}
	f(\boz)=\sum_{m\in \Z} f_m(\boz') u(\boz)^m,
	\end{equation}
	where the $f_m: \Omega^{r-1}\to \Ci$ are weak modular forms of weight $k-m$ for $\G^{r-1}(\fn)$, uniquely determined 
	by $f$. 
\end{thm}
\begin{proof}
See Proposition 5.4 and Theorem 5.9 in \cite{BBP1}.  
\end{proof}

\begin{example}\label{exampleE_u}
		Let $\vec{u}\in A^r$ be a representative of a nonzero class in $(A/\fn)^r$. Write $\vec{u}=(u_1, u_2, \dots, u_r)=(u_1, \vec{u}')$. 
		It is shown in 	\cite[Thm.\ 6.2]{BB}, that if $u_1=0$, then 
		$$ 
		E_{\vec{u}}(\boz)=E_{\vec{u}'}(\boz')+\text{higher degree terms in }u(\boz). 
		$$
\end{example}

A function satisfying an expansion of the form \eqref{eqBBP1expansion} is said to be 
\textit{holomorphic (resp.\ meromorphic) at infinity}  
if $f_m$ is identically zero for $m<0$ (resp.\ $m\ll 0$). We say that a weak modular form 
$f$ of weight $k$ is \textit{holomorphic (resp.\ meromorphic) at the cusps}, if $j(\gamma, \boz)^{-k}f(\gamma\boz)$ 
is holomorphic (resp.\ meromorphic) at infinity  for all $\gamma\in \GL_r(A)$. 
A \textit{modular form} is a weak modular form $f$ which is holomorphic at the cusps. For example, the coefficient 
forms $g_i(\boz)$ and $E_{\vec{u}}(\boz)$ are modular forms; cf.\ \cite{BBP3}.

\begin{prop}\label{propMerUnits}
	Assume $f\in \cO(\Omega^r)^\times$ is invariant under the action of $\G^r(\fn)$. Then $f$ 
	is meromorphic at the cusps. 
\end{prop}
\begin{proof}
	Since $f(\gamma\boz)$ is still in $(\cO(\Omega^r)^\times)^{\G^r(\fn)}$ for any $\gamma\in \GL_r(A)$, it is enough to 
	prove that $f(\boz)$ is meromorphic at infinity. 
	
   Since $f$ is non-vanishing, there exists a coefficient function $f_m$ in the $u$-expansion of $f$ which is not identically zero.
Notice that $\Omega^{r-1}$ is a rigid analytic space over the uncountable and algebraically closed field $\CC_\infty$.
As the family $\{f_m\mid m \in \ZZ\}$ is countable, we can always find a point $\boz_0' \in \Omega^{r-1}$ so that $f_m(\boz_0') \neq 0$ 
for all $f_m$ which are not identically zero.

Now, for each $n \in \ZZ$, consider the subspace
$$\Omega_n':=\{\boz_0 = (z_1,\boz_0')\mid \Im(\boz_0) \geq n\} \subset \Omega^r,$$
which is stable under $\Gamma_\infty'\cap \Gamma^r(\nfk)$.
The map $\Omega_n' \rightarrow \CC_\infty$, defined by $\boz_0 \mapsto u(\boz_0)$, identifies $(\Gamma_\infty'\cap \Gamma^r(\nfk))\bs\Omega_n'$ with a small punctured disc $D_{\rho_n}^0 = \{w \in \CC_\infty \mid 0 <|w|\leq \rho_n\}$ 
for some $\rho_n \in \RR_{>0}$; cf.\ \cite{BBP1}.
Since $f$ is invariant by $\Gamma_\infty'\cap\Gamma^r(\nfk)$, the restriction of $f$ to $\Omega_n'$ induces a non-vanishing holomorphic function on $D_{\rho_n}^0$, say $\bar{f}$.
Taking $n$ sufficiently large, from the $u$-expansion of $f$ in Theorem~\ref{thm-u-expansion}
we obtain:
$$\bar{f}(w) = \sum_{m \in \ZZ}f_m(\boz_0') w^m, \quad \forall w \in D_{\rho_n}^0.
$$
By the non-archimedean analogue of Picard's Big Theorem~\cite[p.\ 43]{FvdP},
$\bar{f}$ has a meromorphic continuation to the whole disc $D_{\rho_n} = \{z \in \CC_\infty\mid |z|\leq \rho_n\}$.
This means that there exists $m_0 \in \ZZ$ so that
the Laurent coefficient $f_m(\boz_0')$ of $\bar{f}$ vanishes if $m<m_0$.
From our choice of $\boz_0'$, we get that $f_m$ is identically zero if $m<m_0$.
Therefore $f$ is meromorphic at infinity.
\end{proof}

\subsection{Cuspidal divisors}

Given a non-zero ideal $\nfk$ of $A$,
let $Y^r_0(\fn):= \G_0^r(\fn)\bs \Omega^r$ and $X^r_0(\fn)$ be its Satake compactification.
We have:

\begin{lem}\label{lemNumbercusps} 
	If $\fn$ is a square-free ideal of $A$ with $s$ prime factors, then the number of cusps of $X^r_0(\fn)$ is equal to $2^s$. 
\end{lem}
\begin{proof} 
	From the analytic construction of $X^r_0(\fn)$ it follows that the cusps of $X^r_0(\fn)$ are in bijection 
	with the orbits of $\G^r_0(\fn)$ acting on the set of $(r-1)$-dimensional subspaces of $F^r$; cf.\ \cite[(1.2)]{Kapranov}, \cite[p.\ 75]{GekelerSatake}.  
	This set 
	of orbits is in natural bijection with the set of orbits of $\G_0^r(\fn)$ acting on $\p^{r-1}(F)=\p^{r-1}(A)$ from the left. 
	Note that $\GL_r(A)$ acts transitively on $\p^{r-1}(A)$. (For any column vector $(b_1, \dots, b_r)^t$ 
	with $\gcd(b_1, \dots, b_r)=1$ we can find a matrix $\gamma$ in $\GL_r(A)$ whose first column is $(b_1, \dots, b_r)^t$. Then 
	$\gamma (1, 0, \dots, 0)^t=(b_1, \dots, b_r)^t$.) The stabilizer of $[1:0:\cdots:0]\in \p^{r-1}(A)$ in $\GL_r(A)$ 
	is $\G_\infty$. 
	Thus, $\p^{r-1}(A)=\GL_r(A)/\G_\infty$. 
	
By associating to a matrix in $\GL_r(A)$ its first column, one obtains a bijective map 
\begin{align*}
	\G^r(\fn)\bs \p^{r-1}(A) = \G^r(\fn)\bs \GL_r(A)/\G_\infty \longleftrightarrow (A/\fn)^r_{\mathrm{prim}}/\F_q^\times,
\end{align*}	
where $(A/\fn)^r_{\mathrm{prim}}$ denotes the set of primitive vectors in $(A/\fn)^r$ (i.e., those that span a direct summand $\neq 0$).  
Let $\fn=\fp_1\cdots \fp_s$ be the prime decomposition of $\fn$. By reducing a given primitive vector modulo the primes $\fp_i$, we obtain a bijection 
$$
(A/\fn)^r_{\mathrm{prim}}/\F_q^\times\To \left(\prod_{i=1}^s (\F_{\fp_i}^r-\vec{0})\right)/\F_q^\times. 
$$
Thus, the cusps of $X_0^r(\fn)$ are in bijection with the orbits of $\overline{\G^r_0(\fn)}\cong \prod_{i=1}^s \overline{\G^r_0(\fp_i)}$ acting on 
this latter set, where $\overline{\G^r_0(\fn)}:=\G^r_0(\fn)/\G^r(\fn)$. The action of $\F_q^\times$ on 
$(A/\fn)^r_{\mathrm{prim}}$ can be subsumed into the action of $\overline{\G^r_0(\fn)}$, so it is enough to show 
that $\overline{\G^r_0(\fp)}$ acting on $(\F_\fp^r-\vec{0})$ (as on column vectors) has two orbits for a prime $\fp$.  

Note that 
$\overline{\G^r_0(\fp)}$ is the subgroup of $\begin{pmatrix} \GL_1(\F_\fp) & \ast \\ 0 & \GL_{r-1}(\F_\fp)\end{pmatrix}$ consisting of matrices 
whose determinant is in $\F_q^\times$. Clearly, the set $\{(a,0,\dots,0)^t\mid a\in \F_\fp^\times\}$ is one of the orbits.  
On the other hand, any $(b_2, \dots, b_r)^t$, where $b_i\in \F_\fp$ are not all zero, 
can be the last column of $\GL_{r-1}(\F_\fp)$.  It is easy to see that there is a matrix in $\overline{\G^r_0(\fp)}$ 
whose last column is $(\ast, b_2,\dots, b_r)^t$, where $\ast$ is an arbitrary element of $\F_\fp$. 
This implies that the orbit of $(0, 0,\dots, 1)$ includes all nonzero vectors of $\F_\fp^r$ except those of the form $(a,0,\dots,0)$, $a\in \F_\fp$. 
\end{proof}

\begin{thm}\label{thmModUnits}
	Assume $\fn$ is squate-free with $s$ prime factors. 
	\begin{enumerate}
		\item The group $(\cO(\Omega^r)^\times)^{\G_0^r(\fn)}/\Ci^\times$ is a free abelian group of rank $2^s-1$. 
		\item If $r\geq 3$, then the $2^s-1$ harmonic $1$-cochains 
		$$\{\cP(\Theta_\fm)\ :\quad  \fm\mid \fn, \quad \fm\neq 1\},  
		$$
		form a basis of $\Har^1(\cB^r, \Q)^{\G_0^r(\fn)}$. 
	\end{enumerate}
\end{thm}
\begin{proof}
By Proposition \ref{propMerUnits}, the elements of $(\cO(\Omega^r)^\times)^{\G_0^r(\fn)}$ are 
	meromorphic at the cusps of $\G_0^r(\fn)$. 
	Let $C$ be the set of cusps of $X^r_0(\fn)$. For $c\in C$, 
	let $\cI_c$ be the sheaf of ideals of $\cO_{X^r_0(\fn)}$ defining the divisor $c$. 
	By \cite[2.2]{Kapranov}, some positive multiple of $c$ is a Cartier divisor,  i.e., $\cI_{c}^{\otimes n_c}$ is invertible for some $n_c\in \Z_{>0}$.
	Let $\cI=\bigotimes_{c\in C} \cI_{c}^{\otimes n_c}$. 
	Repeating the argument of the proof of Lemma 10.7 in \cite{BBP2}, with $\cL^k$ in that lemma replaced by $\cO_{X^r_0(\fn)}\otimes \cI^{-m}$, 
	$m\geq 0$, one can identify $(\cO(\Omega^r)^\times)^{\G_0^r(\fn)}$ 
	with a subgroup of the group $\cF$ of nonzero rational function on $X^r_0(\fn)$ whose divisors are supported at $C$.
	On the other hand, every $f \in \cF$ is uniquely determined by its restrictions to $Y^r_0(\fn)$. Thus, the pull back of $f$ to $\Omega^r$ gives a modular unit invariant by $\Gamma_0^r(\fn)$. Therefore we may identify
	$(\cO(\Omega^r)^\times)^{\G_0^r(\fn)}$ with $\cF$.
	
	Since $X^r_0(\fn)$ is normal, for a nonzero rational function $f$ on $X^r_0(\fn)$ the valuation $\ord_c(f)$ of $f$ at $c\in C$ is well-defined, so we get a 
	map 
	$$
\div: (\cO(\Omega^r)^\times)^{\G_0^r(\fn)} \To \bigoplus_{c\in C} \Z, \qquad f\longmapsto \div(f):=(\ord_c(f))_{c\in C}. 
	$$
	If $\div(f_1)=\div(f_2)$, then $f_1/f_2$ 
	is a rational function which is regular in codimension one. Since $X^r_0(\fn)$ is normal,
$f_1/f_2$ is everywhere regular by \cite[Thm. 4.1.14]{Liu}.  
	 On the other hand, $X^r_0(\fn)$ is projective and connected, so $f_1/f_2\in \Ci^\times$; cf. \cite[Ex. II.4.5]{Hartshorne}. 
	Thus, there is an exact sequence 
	$$
	1\To \Ci^\times \To (\cO(\Omega^r)^\times)^{\G_0^r(\fn)} \overset{\div} {\To}\bigoplus_{c\in C} \Z
	$$
	We will show that $\deg(f)=0$ for $f\in (\cO(\Omega^r)^\times)^{\G_0^r(\fn)}$, where 
	$$\deg(f):=\sum_{c\in C} \ord_c(f).$$ 
	Assuming this for the moment, 
	we conclude that $(\cO(\Omega^r)^\times)^{\G_0^r(\fn)}/\Ci^\times$ is a free abelian group of rank $\leq \#C-1$. 
	
	Let $X^r(\fn)$ be the Satake compactifications of $Y^r(\fn)= \Gamma^r(\fn)\backslash \Omega^r$,
	and let $\widetilde{C}$ be the set of cusps of $X^r(\fn)$.
	For a non-zero rational function $\tilde{f}$ on $X^r(\nfk)$ whose divisor is supported at $\widetilde{C}$, define
	$$
	\widetilde{\deg}(\tilde{f}):= \sum_{\tilde{c} \in \widetilde{C}} \ord_{\tilde{c}}(\tilde{f}).
	$$
	Since $X^r(\nfk)$ is a Galois covering of $X^r_0(\nfk)$, for each $f \in (\cO(\Omega^r)^\times)^{\G_0^r(\fn)}$ one has
	$$
	\widetilde{\deg}(f) = [\Gamma^r_0(\nfk): \Gamma^r(\nfk)\cdot \FF_q^\times]\cdot \deg (f).
	$$
	Moreover, for each $\gamma\in \GL_r(A)$ we have $\gamma \widetilde{C} = \widetilde{C}$.
	Put $f_\gamma(\boz):=f(\gamma\boz)$.
	Observe that $\ord_{\gamma^{-1} \tilde{c}}(f_\gamma)=\ord_{\tilde{c}}(f)$ for every cusp $\tilde{c} \in \widetilde{C}$.
	Hence $\widetilde{\deg}(f)=\widetilde{\deg}(f_\gamma)$. Now consider the modular unit 
	$$g(\boz)=\prod_{\gamma \in \G_0^r(\fn)\bs \GL_r(A)} f_\gamma(\boz),$$
	where the product is over a set of left coset representatives of $\G_0^r(\fn)$ in $\GL_r(A)$. It is clear that 
	$g$ is well-defined and $\GL_r(A)$-invariant. By \eqref{eqLongCohomSeq} and Lemma \ref{lemH(1)=0},  $(\cO(\Omega^r)^\times)^{\GL_r(A)}=\Ci^\times$. 
On the other hand, if $\deg(f)\neq 0$, then $\widetilde{\deg}(f)\neq 0$, and so 
$$\widetilde{\deg}(g) = [\GL_r(A):\Gamma_0^r(\fn)]\cdot \widetilde{\deg}(f) \neq 0.
$$
Thus $g$ 
cannot be a constant function, since at some cusp $\tilde{c}$ in $\widetilde{C}$ it must have nonzero order. 
Thus, $\deg(f)=0$. 

Now assume $\fn$ is square-free with $s$ prime factors. By Lemma \ref{lemNumbercusps}, the number of cusps is $2^s$. Hence 
	$$\rank_\Z((\cO(\Omega^r)^\times)^{\G_0^r(\fn)}/\Ci^\times)\leq 2^s-1. 
	$$
	On the other hand, by \eqref{eqLongCohomSeq}, $(\cO(\Omega^r)^\times)^{\G_0^r(\fn)}/\Ci^\times$ embeds via $\cP$ into 
	$\Har^1(\cB^r, \Z)^{\G_0^r(\fn)}$. 
	Since by Corollary \ref{corLinIndepTheta} the harmonic 1-cochains $\cP(\Theta_\fm)$,  $\fm\mid \fn$, $ \fm\neq 1$, 
	are linearly independent over $\Q$, we conclude that the rank of $(\cO(\Omega^r)^\times)^{\G_0^r(\fn)}/\Ci^\times$ is $2^s-1$. 
	This proves part (1). Part (2) follows from the previous discussion and Theorem \ref{thmKazdanT}. 
\end{proof}

\begin{defn} 
	The \textit{cuspidal divisor group} $\cC^r(\fn)$ of $X^r_0(\fn)$ is the subgroup  
	of the divisor class group of $X^r_0(\fn)$ generated by the Weil divisors $c-c'$, where $c$ and $c'$ 
	run over the cusps of $X^r_0(\fn)$. 
\end{defn}

\begin{thm}\label{thmCDG}\hfill 
	\begin{enumerate}
		\item If $\fn$ is square-free, then $\cC^r(\fn)$ is a finite group. 
		\item If $\fn=\fp$ is prime, then $\cC^r(\fp)$ is a cyclic group of order 
		$$
		\frac{|\fp|^{r-1}-1}{\gcd(|\fp|-1, q^r-1)}
		$$
	\end{enumerate}
\end{thm}
\begin{proof} The fact that $\cC^r(\fn)$ is finite for any $\fn$ follows from \cite{Kapranov}. We give a different proof 
	for square-free $\fn$. 
	
	Let $C$ be the set of cusps of $X^r_0(\fn)$. 
	Let $\cD=\bigoplus_{c\in C}\Z$ and $\cD^0$ be the kernel of the augmentation map $\cD\to \Z$. 
	Let $\cF$ be the subgroup of nonzero rational function on $X^r_0(\fn)$ whose divisors 
	are supported on the cusps. As in the proof of Theorem \ref{thmModUnits}, one can identify $\cF$ with $(\cO(\Omega^r)^\times)^{\G_0^r(\fn)}$, and we have the following
	exact sequence 
	$$
	1\To \C_\infty^\times \To (\cO(\Omega^r)^\times)^{\G_0^r(\fn)} \xrightarrow{\div} \cD^0\To \cC^r(\fn)\To 1. 
	$$
	By Theorem \ref{thmModUnits}, when $\fn$ is square-free, 
	the rank of $(\cO(\Omega^r)^\times)^{\G_0^r(\fn)}/\Ci^\times$ is equal to the rank of $\cD^0$. Hence $\cC^r(\fn)$ 
	is a finite group, which proves (1). 
	
	Now assume $\fn=\fp$ is prime. By Theorem \ref{thmRootTheta}, the 
	largest integer $m$ such that there exists an $m$-th root of $\Theta_\fp$ in $\cO(\Omega^r)^\times$ 
	is $(q-1)\cdot \gcd(|\fp|-1, q^r-1)$. By definition, $\theta_\fp$ is such a root; see \eqref{eq_def_theta_fp}. 
	The character $\omega_\fp$ of $\theta_\fp$ has exact order $(q-1)$; see Proposition \ref{propCharacteroftheta}. Hence 
	$\theta_\fp^{q-1}$, but no smaller power of $\theta_\fp$, is invariant under $\G^r_0(\fp)$. 
	By Theorem \ref{thmModUnits}, $(\cO(\Omega^r)^\times)^{\G_0^r(\fp)}/\Ci^\times$ is a free abelian group of rank $1$. 
	A generator of this group is a root of $\Theta_\fp$ invariant under $\G_0^r(\fp)$. Hence 
	$\theta_\fp^{q-1}$ is a generator. 
	
	By Lemma \ref{lemNumbercusps}, $X_0^r(\fp)$ has two cusps. 
	Denote the cusp corresponding to the orbit of $(1, 0, \dots, 0) \in (A/\fp)^r_{\mathrm{prim}}/\F_q^\times$ 
	under the action of $\overline{\G^r_0(\fp)}$ by $[\infty]$. By the previous paragraph, to prove (2), it is enough to 
	show that $\ord_{[\infty]}(\Theta_\fp)=1-|\fp|^{r-1}$. 
	The order of vanishing of $\Delta$ at the unique cusp of $X^r_0(1)$ is $1$; see \cite[p.\ 79]{GekelerSatake}. 
	(In fact, $X^r_0(1)$ is a weighted projective space $\mathrm{Proj}(\Ci[g_1, \dots, g_r=\Delta])$, and   
	the cusp is the vanishing locus of $g_r$.) Under the 
	natural morphism $X^r_0(\fp)\to X^r_0(1)$ the cusps of $X^r_0(\fp)$ map to the unique cusp of $X^r_0(1)$. 
	The cusp $[\infty]$ of $X^r_0(\fp)$ is unramified over the unique cusp of $X^r_0(1)$ 
	since the stabilizers of $(1,0,\cdots,0)$ in $\G^r_0(1)/\G^r(\fp)$ and $\G^r_0(\fp)/\G^r(\fp)$ are the same. 
	This implies that the order of vanishing of $\Delta$ at $[\infty]$ is also $1$. 
	On the other hand, by \eqref{eqG_fn_def} and Example \ref{exampleE_u}, the 
	numerator of $\Theta_\fp$ in \eqref{eqTheta_fn_def} is a product of Eisenstein series which do not vanish at $[\infty]$. 
	Hence $\Theta_\fp$ has a pole at $[\infty]$ of order $|\fp|^{r-1}-1$, as was required to show. 
\end{proof}

\begin{rem}
If $r\geq 3$, then the Gekeler--van der Put homomorphism  
	$$(\cO(\Omega^r)^\times)^{\G_0^r(\fn)}\overset{\cP}{\To} \Har^1(\cB^r, \Z)^{\G_0^r(\fn)}$$ 
	has finite cokernel by Theorem \ref{thmKazdanT}. From the proof of Theorem \ref{thmCDG} 
	it is easy to see that $\cP$ is generally not surjective. Indeed, $\cP(\theta_\fp)\in \Har^1(\cB^r, \Z)^{\G_0^r(\fp)}$,  
	since the character $\omega_\fp$ of $\theta_\fp$ disappears after applying $\cP$. 
	On the other hand, $\cP((\cO(\Omega^r)^\times)^{\G_0^r(\fp)})$ is generated by $\cP(\theta_\fp^{q-1})=(q-1)\cP(\theta_\fp)$. 
	In fact, we claim that 
	$$0\To \Ci^\times \To (\cO(\Omega^r)^\times)^{\G_0^r(\fp)}\overset{\cP}{\To} \Har^1(\cB^r, \Z)^{\G_0^r(\fp)}\To \Z/(q-1)\Z\To 0$$
	is exact, so $\Har^1(\cB^r, \Z)^{\G_0^r(\fp)}$ is generated by $\cP(\theta_\fp)$. To prove this claim, observe that 
	$\cP(\theta_\fp)$ is an integer multiple of a generator of $\Har^1(\cB^r, \Z)^{\G_0^r(\fp)}\cong \Z$. Thus, to show that 
	$\cP(\theta_\fp)$ itself is a generator, it is enough to show that the greatest common divisor of its values is $1$. 
	More generally, we claim that the gcd of the values of $\cP(\Theta_\fn)$ is $(q-1)\cdot \gcd(|\fn|-1, q^r-1)$. 
	To see this, the computations in Corollary \ref{corRootTheta} indicates that
	the gcd in question divides $(q-1)\cdot \gcd(|\fn|-1, q^r-1)$.
	On the other hand, we have
	$$\cP(\Theta_\fn) = (q-1)\cdot \gcd(|\fn|-1, q^r-1) \cdot \cP(\theta_\fn).$$
	Therefore the claim holds.
	 
\end{rem}

\subsection*{Acknowledgements} Part of this work was carried out while the first 
author was visiting the National Center for Theoretical Sciences in Hsinchu. 
He thanks the institute for its hospitality and excellent working conditions.


\renewcommand{\bibliofont}

\bibliographystyle{acm}
\bibliography{Bibliography.bib}

\begin{thebibliography}{10}

\bibitem{AA}
{\sc A\"{\i}t~Amrane, Y.}
\newblock Cohomology of {D}rinfeld symmetric spaces and harmonic cochains.
\newblock {\em Ann. Inst. Fourier (Grenoble) 56}, 3 (2006), 561--597.

\bibitem{BB}
{\sc Basson, D., and Breuer, F.}
\newblock On certain {D}rinfeld modular forms of higher rank.
\newblock {\em J. Th\'{e}or. Nombres Bordeaux 29}, 3 (2017), 827--843.

\bibitem{BBP1}
{\sc Basson, D., Breuer, F., and Pink, R.}
\newblock {D}rinfeld modular forms of arbitrary rank. {Part I: Analytic
  Theory}.
\newblock {\em preprint\/} (2018).

\bibitem{BBP2}
{\sc Basson, D., Breuer, F., and Pink, R.}
\newblock {D}rinfeld modular forms of arbitrary rank. {Part II: Comparison with
  Algebraic Theory}.
\newblock {\em preprint\/} (2018).

\bibitem{BBP3}
{\sc Basson, D., Breuer, F., and Pink, R.}
\newblock {D}rinfeld modular forms of arbitrary rank. {Part III: Examples}.
\newblock {\em preprint\/} (2018).

\bibitem{Bump}
{\sc Bump, D.}
\newblock {\em Automorphic forms and representations}, vol.~55 of {\em
  Cambridge Studies in Advanced Mathematics}.
\newblock Cambridge University Press, Cambridge, 1997.

\bibitem{deShalit}
{\sc De~Shalit, E.}
\newblock Residues on buildings and de {R}ham cohomology of {$p$}-adic
  symmetric domains.
\newblock {\em Duke Math. J. 106}, 1 (2001), 123--191.

\bibitem{Drinfeld}
{\sc Drinfeld, V.~G.}
\newblock Elliptic modules.
\newblock {\em Mat. Sb. (N.S.) 94(136)\/} (1974), 594--627.

\bibitem{FvdP}
{\sc Fresnel, J., and van~der Put, M.}
\newblock {\em Rigid analytic geometry and its applications}, vol.~218 of {\em
  Progress in Mathematics}.
\newblock Birkh\"{a}user Boston, Inc., Boston, MA, 2004.

\bibitem{GekelerDMHR2}
{\sc Gekeler, E.-U.}
\newblock On {D}rinfeld modular forms of higher rank {II}.
\newblock {\em J. Number Theory, to appear\/}.
\newblock DOI:10.1016/j.jnt.2018.11.011.

\bibitem{GekelerDMHR4}
{\sc Gekeler, E.-U.}
\newblock On {D}rinfeld modular forms of higher rank {IV}: {M}odular forms with
  level.
\newblock {\em J. Number Theory, to appear\/}.
\newblock DOI:10.1016/j.jnt.2019.04.019.

\bibitem{GekelerSatake}
{\sc Gekeler, E.-U.}
\newblock Satake compactification of {D}rinfeld modular schemes.
\newblock In {\em Proceedings of the conference on {$p$}-adic analysis
  ({H}outhalen, 1987)\/} (1986), Vrije Univ. Brussel, Brussels, pp.~71--81.

\bibitem{GekelerUber}
{\sc Gekeler, E.-U.}
\newblock {\"Uber Drinfeldsche Modulkurven vom Hecke-Typ}.
\newblock {\em Compositio Math. 57\/} (1986), 219--236.

\bibitem{GekelerImproper}
{\sc Gekeler, E.-U.}
\newblock Improper {E}isenstein series on {B}ruhat-{T}its trees.
\newblock {\em Manuscripta Math. 86}, 3 (1995), 367--391.

\bibitem{GekelerDelta}
{\sc Gekeler, E.-U.}
\newblock On the {D}rinfeld discriminant function.
\newblock {\em Compositio Math. 106}, 2 (1997), 181--202.

\bibitem{GekelerDMHR1}
{\sc Gekeler, E.-U.}
\newblock On {D}rinfeld modular forms of higher rank.
\newblock {\em J. Th\'{e}or. Nombres Bordeaux 29}, 3 (2017), 875--902.

\bibitem{GekelerExacSeq}
{\sc Gekeler, E.-U.}
\newblock Invertible functions on non-archimedean symmetric spaces.
\newblock {\em arXiv:1904.00844v1\/} (2019).

\bibitem{GN}
{\sc Gekeler, E.-U., and Nonnengardt, U.}
\newblock Fundamental domains of some arithmetic groups over function fields.
\newblock {\em Internat. J. Math. 6}, 5 (1995), 689--708.

\bibitem{GreenBook}
{\sc Gekeler, E.-U., van~der Put, M., Reversat, M., and Van~Geel, J.}, Eds.
\newblock {\em Drinfeld modules, modular schemes and applications\/} (1997),
  World Scientific Publishing Co., Inc., River Edge, NJ.

\bibitem{GI}
{\sc Goldman, O., and Iwahori, N.}
\newblock The space of {$\mathfrak{p}$}-adic norms.
\newblock {\em Acta Math. 109\/} (1963), 137--177.

\bibitem{Goss}
{\sc Goss, D.}
\newblock {\em Basic structures of function field arithmetic}, vol.~35 of {\em
  Ergebnisse der Mathematik und ihrer Grenzgebiete (3) [Results in Mathematics
  and Related Areas (3)]}.
\newblock Springer-Verlag, Berlin, 1996.

\bibitem{Haberli}
{\sc H\"aberli, S.}
\newblock {\em Satake compactification of analytic {D}rinfeld modular
  varieties}.
\newblock 2018.
\newblock Thesis (Ph.D.)--ETH Zurich.

\bibitem{Hartshorne}
{\sc Hartshorne, R.}
\newblock {\em Algebraic geometry}.
\newblock Springer-Verlag, New York-Heidelberg, 1977.
\newblock Graduate Texts in Mathematics, No. 52.

\bibitem{Jacquet-Shalika}
{\sc Jacquet, H., and Shalika, J.~A.}
\newblock On {E}uler products and the classification of automorphic
  representations. {I}.
\newblock {\em Amer. J. Math. 103}, 3 (1981), 499--558.

\bibitem{Kapranov}
{\sc Kapranov, M.~M.}
\newblock Cuspidal divisors on the modular varieties of elliptic modules.
\newblock {\em Izv. Akad. Nauk SSSR Ser. Mat. 51}, 3 (1987), 568--583, 688.

\bibitem{Liu}
{\sc Liu, Q.}
\newblock {\em Algebraic geometry and arithmetic curves}, vol.~6 of {\em Oxford
  Graduate Texts in Mathematics}.
\newblock Oxford University Press, Oxford, 2002.
\newblock Translated from the French by Reinie Ern\'{e}, Oxford Science
  Publications.

\bibitem{Margulis}
{\sc Margulis, G.~A.}
\newblock {\em Discrete subgroups of semisimple {L}ie groups}, vol.~17 of {\em
  Ergebnisse der Mathematik und ihrer Grenzgebiete (3) [Results in Mathematics
  and Related Areas (3)]}.
\newblock Springer-Verlag, Berlin, 1991.

\bibitem{MazurEis}
{\sc Mazur, B.}
\newblock Modular curves and the {E}isenstein ideal.
\newblock {\em Inst. Hautes \'{E}tudes Sci. Publ. Math.}, 47 (1977), 33--186
  (1978).
\newblock With an appendix by Mazur and M. Rapoport.

\bibitem{OggBAMS}
{\sc Ogg, A.~P.}
\newblock Diophantine equations and modular forms.
\newblock {\em Bull. Amer. Math. Soc. 81\/} (1975), 14--27.

\bibitem{Pal}
{\sc P\'{a}l, A.}
\newblock On the torsion of the {M}ordell-{W}eil group of the {J}acobian of
  {D}rinfeld modular curves.
\newblock {\em Doc. Math. 10\/} (2005), 131--198.

\bibitem{Pink}
{\sc Pink, R.}
\newblock Compactification of {D}rinfeld modular varieties and {D}rinfeld
  modular forms of arbitrary rank.
\newblock {\em Manuscripta Math. 140}, 3-4 (2013), 333--361.

\bibitem{SS}
{\sc Schneider, P., and Stuhler, U.}
\newblock The cohomology of {$p$}-adic symmetric spaces.
\newblock {\em Invent. Math. 105}, 1 (1991), 47--122.

\bibitem{SerreT}
{\sc Serre, J.-P.}
\newblock {\em Trees}.
\newblock Springer Monographs in Mathematics. Springer-Verlag, Berlin, 2003.
\newblock Translated from the French original by John Stillwell, Corrected 2nd
  printing of the 1980 English translation.

\bibitem{vdPut}
{\sc van~der Put, M.}
\newblock Discrete groups, {M}umford curves and theta functions.
\newblock {\em Ann. Fac. Sci. Toulouse Math. (6) 1}, 3 (1992), 399--438.

\bibitem{WeiKLF}
{\sc Wei, F.-T.}
\newblock On {K}ronecker terms over global function fields.
\newblock {\em Invent. Math. 220}, 3 (2020), 847--907.

\bibitem{Weil}
{\sc Weil, A.}
\newblock {\em Dirichlet series and automorphic forms}, vol.~189 of {\em
  Lecture Notes in Mathematics}.
\newblock Springer, Berlin, 1971.

\end{thebibliography}

\end{document}